\documentclass[12pt]{amsart}
\usepackage{amscd,amsthm,amssymb,amsfonts,amsmath,euscript}
\usepackage{mathrsfs,mathtools}
\usepackage{xcolor}
\usepackage{comment}
\theoremstyle{plain}
\newtheorem{thm}{Theorem}[section]
\newtheorem{lemma}[thm]{Lemma}
\newtheorem{prop}[thm]{Proposition}
\newtheorem{cor}[thm]{Corollary}

\theoremstyle{definition}
\newtheorem{defn}[thm]{Definition}
\newtheorem{eg}[thm]{Example}

\theoremstyle{remark}
\newtheorem{remark}[thm]{Remark}

\newtheorem*{notation}{{\bf Notation}}
\newtheorem*{thank}{{\bf Acknowledgments}}
\newcommand{\nc}{\newcommand}

\def\makeop#1{\expandafter\def\csname#1\endcsname
  {\mathop{\rm #1}\nolimits}\ignorespaces}
\makeop{Hom}   \makeop{End}   \makeop{Aut}   \makeop{Isom}  \makeop{Pic} 
\makeop{Gal}   \makeop{ord}   \makeop{Char}  \makeop{Div}   \makeop{Lie} 
\makeop{PGL}   \makeop{Corr}  \makeop{PSL}   \makeop{sgn}   \makeop{Spf}
\makeop{Spec}  \makeop{Tr}    \makeop{Nr}    \makeop{Fr}    \makeop{disc}
\makeop{Proj}  \makeop{supp}  \makeop{ker}   \makeop{im}    \makeop{dom}
\makeop{coker} \makeop{Stab}  \makeop{SO}    \makeop{SL}    \makeop{SL}
\makeop{Cl}    \makeop{cond}  \makeop{Br}    \makeop{inv}   \makeop{rank}
\makeop{id}    \makeop{Fil}   \makeop{Frac}  \makeop{GL}    \makeop{SU}
\makeop{Nrd}   \makeop{Sp}    \makeop{Tr}    \makeop{Trd}   \makeop{diag}
\makeop{Res}   \makeop{ind}   \makeop{depth} \makeop{Tr}    \makeop{st}
\makeop{Ad}    \makeop{Int}   \makeop{tr}    \makeop{Sym}   \makeop{can}
\makeop{length}\makeop{SO}    \makeop{torsion} \makeop{GSp} \makeop{Ker}
\makeop{Adm}   \makeop{Mat}
\makeop{Q-isog}
\makeop{Rad}
\makeop{Ind}
\DeclareMathOperator{\Mass}{Mass}

\DeclareMathOperator{\Irr}{Irr}
\DeclareMathOperator{\GU}{GU}

\def\makebb#1{\expandafter\def
  \csname bb#1\endcsname{{\mathbb{#1}}}\ignorespaces}
\def\makebf#1{\expandafter\def\csname bf#1\endcsname{{\bf
      #1}}\ignorespaces} 
\def\makegr#1{\expandafter\def
  \csname gr#1\endcsname{{\mathfrak{#1}}}\ignorespaces}
\def\makescr#1{\expandafter\def
  \csname scr#1\endcsname{{\EuScript{#1}}}\ignorespaces}
\def\makecal#1{\expandafter\def\csname cal#1\endcsname{{\mathcal
      #1}}\ignorespaces} 
\def\doLetters#1{#1A #1B #1C #1D #1E #1F #1G #1H #1I #1J #1K #1L #1M
                 #1N #1O #1P #1Q #1R #1S #1T #1U #1V #1W #1X #1Y #1Z}
\def\doletters#1{#1a #1b #1c #1d #1e #1f #1g #1h #1i #1j #1k #1l #1m
                 #1n #1o #1p #1q #1r #1s #1t #1u #1v #1w #1x #1y #1z}
\doLetters\makebb   \doLetters\makecal  \doLetters\makebf
\doLetters\makescr 
\doletters\makebf   \doLetters\makegr   \doletters\makegr

\normalsize

\makeop{Bl}

\def\Fpbar{\overline{\bbF}_p}
\def\Fp{{\bbF}_p}

\def\Qp{{\bbQ}_p}

\def\Zp{{\bbZ}_p}
\def\Qbar{\overline{\bbQ}}
\def\Sh{{\rm Sh}}

\newcommand{\Z}{\mathbb Z}
\newcommand{\Q}{\mathbb Q}
\newcommand{\R}{\mathbb R}
\newcommand{\C}{\mathbb C}
\newcommand{\D}{\mathbb D}    % pro algebraic torus
\newcommand{\A}{\mathbb A}    % for adele
\newcommand{\F}{\mathbb F}

\newcommand{\<}{\langle}   %\< is not defined yet.
\renewcommand{\>}{\rangle} %\> is already defined.

\nc{\embed}{\hookrightarrow}
\newcommand{\M}{\mathcal M}
\newcommand{\K}{\mathsf K}
\newcommand{\ch}{characteristic }
\newcommand{\ac}{algebraically closed }
\newcommand{\dieu}{Dieudonn\'{e} }

\nc{\ol}{\overline}
\nc{\wt}{\widetilde}
\nc{\opp}{\mathrm{opp}}
\def\ul{\underline}

\def\wh{\widehat}

\def\char{{\rm char\,}}

\makeop{Ram}
\makeop{Rep}

\def\sfF{\mathsf{F}}
\def\sfV{\mathsf{V}}
\def\sp{{\rm sp}}

\usepackage{color}
\usepackage{marvosym}

\begin{document}
\numberwithin{equation}{section}
%\rhomberwithin{section}{chapter}

%\usepackage[notref,notcite]{showkeys}

\title[The supersingular 
locus]
{On the supersingular 
locus of  Shimura varieties  for quaternionic unitary groups}
 \author{Yasuhiro Terakado}
\address{
(Terakado) School of System Design and Technology, Tokyo Denki University 
\\
5 Senju Asahi-cho  
\\
Adachi-ku, Tokyo, Japan, 120-8551} 
\email{yterakado@mail.dendai.ac.jp} 

 \author{Jiangwei Xue}
 \address{(Xue)  Collaborative Innovation Center of Mathematics, School of
  Mathematics and Statistics, Wuhan University \\
   Luojiashan  
   \\ 
  Wuhan, Hubei, P.R. China, 430072} 
  \email{xue\_j@whu.edu.cn}
  \address{(Xue) Hubei Key Laboratory of Computational Science  (Wuhan
  University) 
  \\ Wuhan, Hubei, P.R. China, 430072}

\author{Chia-Fu Yu}
\address{
(Yu) Institute of Mathematics, Academia Sinica \\
Astronomy Mathematics Building \\
No.~1, Sec.~4, Roosevelt Rd. \\ 
Taipei, Taiwan, 106319} 
\email{chiafu@math.sinica.edu.tw}

%\address{
%(Yu) National Center for Theoretical Sciences\\
%Cosmology Building\\
%No.~1, Sec.~4,  Roosevelt Rd.\\
%Taipei, Taiwan, 106319}

%\date{June 23, 2000}

\date{\today}
\subjclass[2010]{} 


\begin{abstract}
We study a Shimura variety attached to a unitary similitude group of a skew-Hermitian form   over a totally indefinite quaternion algebra over a totally real number field. 
We give a necessary and sufficient condition for the existence of skew-Hermitian self-dual lattices. 
Under this condition we  show that the  superspecial locus in the fiber
at $p$ of the associated Shimura variety is non-empty. 
 We also give an explicit formula for the number of irreducible components of the supersingular  locus when $p$ is odd and unramified in the  quaternion algebra. 
 %associated to the unitary similitude group  of a  Hermitian form with a totally indefinite quaternionic multiplication. 
\end{abstract} 

\maketitle 
\tableofcontents
\section{Introduction}\label{sec:I}

Throughout this paper $p$ denotes a rational prime number and  $N\ge 3$ denotes  a positive integer with $(p,N)=1$. 
% closure $\ol{\F}_p$ of $\F_p$. 
Let $\bfA_{g,N}$ be the moduli scheme over $\Z_{(p)}$ of principally polarized abelian varieties of dimension $g\ge 1$ with a level-$N$ structure, and let $\calA_{g,N} \coloneqq \bfA_{g,N}\otimes \ol{\F}_p$ be the geometric special fiber. 
There are very rich and complicated geometric structures on the space $\calA_{g,N}$, due to the properties of the $p$-divisible groups associated to points to be classified.   
As a result, people introduced and have been investigating geometric problems of the induced strata, notably, the Newton strata, Ekedahl-Oort strata, and central leaves.

We
recall that an abelian variety $A$ over an algebraically  closed field of characteristic $p$ is said to be superspecial (resp.~supersingular) if it is isomorphic 
(resp.~isogenous) to a product of supersingular elliptic curves. 
Let  $\calA_{g, N}^{\rm sp} \subset \calA_{g,N}^{\rm ss}\subset \calA_{g,N}$ be the superspecial (resp.~supersingular)  locus of $\calA_{g,N}$, that is,  the subspace parameterizing the superspecial (resp.~supersingular) abelian varieties in $\calA_{g,N}$.  
%that is, the corresponding abelian varieties are isogenous to a product of supersingular elliptic curves. 
Then $\calA_{g, N}^{\rm sp}$ is the unique $0$-dimensional Ekedahl-Oort stratum, and  $\calA_{g, N}^{\rm ss}$ is the unique closed Newton stratum  of $\calA_{g, N}$.  
An explicit formula for the cardinality of $\calA_{g, N}^{\rm sp}$ was given by Ekedahl \cite{Ekedahl}, using Hashimoto-Ibukiyama's mass formula \cite[Proposition 9]{HI}. 
%For any scheme $X$ of finite type over an algebraically closed field $k$, denote by $\Irr(X)$ the set of irreducible components.
%Concerning the supersingular locus,
%Extending earlier works of Ibukiyama, Katsura, and Oort \cite{IKO, KO} in lower dimensions, 
In \cite{LO}, Li and Oort investigated
the geometry of the supersingular locus, and in particular they derived a formula relating the number of irreducible components to the class number of a genus of quaternion Hermitian lattices. 
An explicit formula for the class number was given in \cite{yu:mass_siegel}. 
\begin{comment}

\begin{thm}\label{thm:I.1}
   \begin{enumerate}
       \item The moduli space $\calA_{g,N}$ is smooth of equi-dimension $g(g+1)/2$ and it has $\varphi(N)\coloneqq|(\Z/N\Z)^\times|$ irreducible components. 
       \item Every Newton stratum is non-empty and the closure of each Newton stratum is a union of Newton strata.  
       \item Every Newton stratum is equi-dimensional with dimension governed by the purity property. 
       \item In each connected component $\calA_{g,N}^0$ of $\calA_{g,N}$, every non-supersingular Newton stratum is  irreducible. The superisnuglar locus of $\calA_{g,N}^0$ has $(-1)^{g(g+1)/2}/2^g \cdot \prod_{i=1}^g \zeta(1-2i) \cdot L_p$ irreducible components, where 
\begin{equation}
L_p=\begin{cases}
    \prod_{i=1}^g (p^i+(-1)^i) & \text{if $g$ is odd;} \\
    \prod_{i=1}^c (p^{4i-2}-1) & \text{if $g=2c$ is even.}
\end{cases}
\end{equation}       
   \end{enumerate} 
\end{thm}
\end{comment}

\begin{thm}[\cite{Ekedahl, HI, LO, yu:mass_siegel}]\label{thm:I.1}
 We write $\zeta(s)$ for  the Riemann zeta function and $\GSp_{2g}$ for  the symplectic similitude group  of degree $2g$. 
 Further we put 
\begin{equation*}
    C(g, N)\coloneqq \lvert \GSp_{2g}(\Z/N\Z)\rvert 
    \cdot  \frac{(-1)^{g(g+1)}/2}{2^g} \cdot \prod_{i=1}^g \zeta(1-2i). 
\end{equation*}

{\rm (1)}
 The cardinality of the superspecial locus $\calA_{g, N}^{\rm sp}$ is equal to  
      $C(g,N) \cdot \prod_{i=1}^g (p^i+(-1)^i)$.

       {\rm{(2)}}  
        The supersingular locus $\calA_{g,N}^{\rm ss}$
    is equidimensional of dimension $\lfloor g^2/4 \rfloor$ and the number of its irreducible components 
    is 
    equal to 
  $C(g, N) \cdot \lambda_p$  where $\lambda_p$ is given by 
\begin{equation*}
    \lambda_p=\begin{cases}
    \prod_{i=1}^g (p^i+(-1)^i) & \text{if $g$ is odd;} \\
    \prod_{i=1}^c (p^{4i-2}-1) & \text{if $g=2c$ is even.}
\end{cases}
\end{equation*}   
\end{thm}
% S_{g,N} has $\varphi(N)$ connected components.
\begin{comment}
Notions of Newton and EO strata for more general Shimura varieties were also introduced and studied by many people in past decades.
For the studies of EO strata, we refer to (Goren-Oort, Wedhorn, Moonen, Zhang)
For non-emptiness and dimensions of Newton stratum, we refer to (Yu, 
For the studies of the supersingular locus, we refer 
We refer to (cite many people, non-emptiness of Newton, dimension \cite{VW, Hamacher, Lee, Zhang}) for some of major developments.  
\end{comment}

The aim of this paper is to study 
 the supersingular locus of a PEL Shimura variety of type C, and in particular to give an explicit formula for the number of the   irreducible components, generalizing Theorem~\ref{thm:I.1}. 
 
Let $F$ be a totally real field of degree $d$ with ring of integers
$O_F$, and $O_B$ a maximal $O_F$-order in a totally indefinite quaternion
algebra $B$ over $F$ which is stable under a positive involution $*$ of $B$. 
Let $b \mapsto \bar b$ denote the canonical involution of $B$. 
There is an element $\gamma\in B^\times$ such that $\gamma+ \bar 
\gamma=0$ and
$b^*=\gamma \bar b \gamma ^{-1}$ for all $b\in B$. 
A {\it polarized $O_B$-abelian scheme} (over a base scheme) is a triple $(A,\lambda,\iota)$, where $(A,\lambda)$ is a polarized abelian scheme  and $\iota: O_B\to \End(A)$ is a ring monomorphism such that $\lambda\circ \iota(b^*)=\iota(b)^t \circ \lambda$ for all $b\in O_B$ (Definition \ref{def:ab}).

Let $\mathscr D=(B, *, O_B,V,\Lambda,
\psi, h_0)$ be an integral PEL datum of type C of rank $m$ (Definition \ref{def:PELC}). In particular, $(V, \psi)$ be a $\Q$-valued skew-Hermitian $(B, *)$-module  of rank $m$, unique up to isomorphism, and  $\Lambda$ is an $O_B$-lattice in $V$. 
Let $\bfG$ be the group scheme over $\Z$ of $B$-linear
$\psi$-similitudes on $(\Lambda,\psi)$, and  $\K=\K(N)$ be the kernel of the reduction mod $N$ map on $\bfG(\wh{\Z})$. 
Let $\bfM_{\K}$ be the associated moduli scheme over $\Z_{(p)}$ of $2dm$-dimensional polarized $O_B$-abelian varieties with level-$N$
structure satisfying the determinant condition (Section \ref{ss:Sh}).  
In the special case where 
$B=\Mat_2(F)$, $O_B=\Mat_2(O_F)$, and $*$ is the transpose $t$, 
 Morita's equivalence reduction implies that the moduli scheme $\bfM_{\sf K}$ is the same as the Hilbert-Siegel moduli scheme of degree $m$ over $F$: the moduli scheme of $dm$-dimensional polarized $O_F$-abelian varieties. This case  has been studied in \cite{yu:mass_hb}. 
When $m=1$ and $d=1$ ($F=\Q$), $\bfM_{\sf K}$ is an integral model of the Shimura curve associated to the quaternion $\Q$-algebra $B$, which is also called a fake modular curve. Similarly, we call $\bfM_{\sf K}$ an integral model of a ``fake" Hilbert-Siegel modular variety (constrast to a quaterntionic Shimura variety which usually specifies to the case where $m=1$).  

In the Hilbert-Siegel case ($B=\Mat_2(F)$, $O_B=\Mat_2(O_F)$, and $*=t$), it is known that there always exists a  {\it principally} polarized $O_B$-abelian variety $(A,\lambda,\iota)$ over $\C$ (for example, one may take a product of $m$ points of the Hilbert modular variety associated to the totally real field $F$).  
However, for a general triple $(B,*,O_B)$, the existence of such an abelian variety requires the condition that $(B, *, O_B)$ extends to a  {\it principal} integral PEL datum $\mathscr D$ (see Section  \ref{ss:PEL}), in other words, there exists a \emph{self-dual} $O_B$-lattice $\Lambda$ in $(V,\psi)$.  
%However, such a lattice does not always exist.  
In the following theorem, we give a necessary and sufficient condition for the existence of a self-dual lattice $\Lambda$, and consequently we obtain a condition for the existence of a principally polarized $O_B$-abelian variety.

\begin{thm}\label{main.1}
Let $(B, *, O_B)$ be as above and $m$ be a positive integer. 
Then the following statements are equivalent:
\begin{itemize}
    \item [(a)] There exists a complex principally polarized $O_B$-abelian variety of dimension $2dm$.
    \item [(b)] There exists a self-dual $O_B$-lattice $\Lambda$ in a (unique) $\Q$-valued skew-Hermitian $(B,*)$-module $(V, \psi)$ of rank $m$.  
    \item [(c)] Either $m$ is even, or for any finite place $v$ of $F$ ramified in $B$ one has $\ord_{\Pi_v}(\gamma)$ is odd. 
    Here, $\Pi_v$ denotes  a uniformizer of the completion $B_v=B\otimes_{F} F_v$   at $v$, and $\ord_{\Pi_v}
    (\, \cdot \,)$ denotes  the $\Pi_v$-adic valuation. 
    \item[(d)] There exists a principally polarized $O_B$-abelian variety   of dimension $2dm$ over an algebraically closed  field $k$  of characteristic $p$ which  satisfies the determinant condition (see \eqref{eq:detA} for the definition). 
\end{itemize}
Under these conditions,  a self-dual $O_B$-lattice $\Lambda$ as in {\rm (b)} is unique up to isomorphism. 
%( there exists a complex principally polarized $O_B$-abelian variety of dimension $2dm$ if and only $m$ is even, or for any finite place $v$ of $F$ ramified in $B$ one has $\ord_{\Pi_v}(\gamma)$ is odd. 
%Here, $\Pi_v$ denotes  a uniformizer of the completion $B_v=B\otimes_{F} F_v$   at $v$, 
%and $\ord_{\Pi_v}
%(\, \cdot \,)$ denotes  the $\Pi_v$-adic  valuation. 
\end{thm}   
We remark that  the determinant  condition in  (d) can not be omitted. 
In Appendix we discuss the non-emptiness of the  moduli space  without the determinant condition in the case where the conditions in Theorem \ref{main.1} do not hold.

Hereafter we assume that the conditions in Theorem 
\ref{main.1} hold, and  that an integral PEL  datum $\mathscr D$ is  principal. 
%Then the associated  complex moduli space $\bfM_{\sf K}(\C)$ is non-empty. 
By the uniqueness of a self-dual lattice in a skew-Hermitian module, 
 the 
generic fiber $\bfM_{\sf K}\otimes \Q$ consists of a single Shimura variety,
rather than a union of some of them.   
%-- the superspecial locus of $\calM$, 
% for maxiaml lattices in a quaternion Hermitian form \cite{Shimura2}.  

Let $k$ be an algebraically  closed field of characteristic $p$, and let  $\calM_{\sf K} \coloneqq \bfM_{\sf K}\otimes k$ be the geometric special fiber. 
If we assume that $p$ is
unramified in $B$, then 
$\calM_{\K}$ 
 has the same number of connected components as the geometric generic fiber $\bfM_{\sf K} \otimes \ol{\Q}$ has (Lan 
\cite[Corollary 6.4.1.2]{Lan}), which is equal to  
$\varphi(N)\coloneqq|(\Z/N\Z)^\times|$. It is known  
that the ordinary locus %$\calM^{\ord}$
of $\calM_{\sf K}$ is non-empty if and only if either $m$ is even or every
place $v$ of $F$ lying over $p$ is unramified in $B$ \cite{yu:2014}. 
Here we show the
opposite extreme case.
Let ${\calM}_{\sf K}^{\rm sp}\subset {\calM}_{\sf K}^{\rm ss} \subset \calM_{\sf K}$ denote  the superspecial and 
 supersingular locus of $\calM_{\sf K}$ (Section \ref{ss:Sh}). 
%This locus is  the smallest EO stratum of $\calA_{g,N}$.  
\begin{thm}\label{main.2}
  The superspecial locus $\calM_{\sf K}^{\sp}$  is non-empty.
\end{thm}
Note that there is no assumption on  $p$ in Theorem~\ref{main.2}. 
The main step of the proof % of Theorem~\ref{main.2} 
is to construct a principally polarized \dieu $O_B\otimes \Z_p$-module  satisfying the determinant condition. 
This requires the equivalent conditions in Theorem~\ref{main.1}.  
For more details, see Section \ref{ss:spmod}.
 
%(that is, $B\otimes \Qp$ is a product of matrix algebras over unramified $\Qp$-extensions), 
In the rest of this introduction we assume that $p$ is unramified in $B$.  
%Let $(B,*,O_B, V,\psi, \Lambda, h_0)$ the unique principal integral
%PEL-datum of rank $m$ extending the input datuam $(B,*,O_B)$ for
%$\bfM$. 
Then ${\sf K}_p \coloneqq \bfG(\Z_p) \subset \bfG(\Q_p)$ is a hyperspecial parahoric subgroup. 
In this case, an exact formula for the cardinality of the superspecial locus $\calM^{\rm sp}_{\sf K}$ was given in \cite[Theorem 1.3]{yu:IMRN-2008},  using Shimura's mass formula \cite{Shimura2}. 
However, in \cite{yu:IMRN-2008} it is implicitly assumed that there exixts a self-dual $O_B$-lattice $\Lambda$ (Theorem  \ref{main.1}) and that the superspecial locus $\calM^{\rm sp}_{\sf K}$ is non-empty (Theorem \ref{main.2}). 
In this paper we also correct the formula given in {\it loc.~cit.}  for local terms at some places $v\nmid p$ of $F$ ramified in $B$ due to the conditions in Theorem~\ref{main.1} (see Remark \ref{rem:correction}).
%\chiafu{ 
%We also correct the exact formulas in \cite[Theorems 1.2 and 1.3]{yu:IMRN-2008} for the cardinality of the  superspecial locus $\calM^{\rm sp}$ (Remark \ref{rem:correction}).  points the for local terms at some places $v\nmid p$ of $F$ ramified in $B$ due to the conditions in Theorem~\ref{main.1}. (I would like to correct my formula in \cite[Theorem 1.3]{yu:IMRN-2008} somewhere, not sure where to put..It is good to put the corrected formula in Theorem 1.4 directly?).} 

In \cite{Hamacher}, Hamacher gave a formula for the
dimension of Newton strata on the reduction of PEL Shimura 
varieties (of type  A or C) with hyperspecial level structure   at $p$.  
In the moduli scheme  $\calM_{\sf K}$ of type C, the unique closed Newton stratum (called the basic locus) is precisely the supersingular locus $\calM^{\rm ss}_{\sf K}$: It is 
equidimensional of dimension (cf.~\cite[Theorem 5.1]{slope})
\begin{equation}
   \dim \calM^{\rm ss}_{\sf K} = \sum_{v|p} \left (  \lfloor f_v/2  \rfloor \frac{m(m+1)}{2}+(f_v-2 \lfloor f_v/2 \rfloor)\cdot  \lfloor m^2/4 \rfloor  \right ), 
\end{equation}
where $v$ runs over the places of $F$ over $p$ and $f_v$ is the inertia degree of $v$.

We give an explicit formula for the number of irreducuble components of $\calM_{\sf K}^{\rm ss}$. 
Let $D_{p,\infty}$ denote the unique quaternion
$\Q$-algebra ramified precisely at $\{p,\infty\}$, and 
$D$ the unique quaternion
$F$-algebra such that $B\otimes_\Q D_{p,\infty}\simeq \Mat_2(D)$. 
Let $\Delta'$ be the discriminant of $D$ over $F$. 
For a finite place $v$ of $F$, let $q_v\coloneqq p^{f_v}$ be the cardinality of the residue field  of $v$. 
%????Note that For $v \mid p$,  one has that $v \mid \Delta'$ if and only if  $f_v$ is odd. ????
%Our main result is as follows. 
\begin{thm}\label{main.3}
  Assume that $p>2$ is unramified in $B$. Then the number of irreducible components of the supersingular locus
  $\calM^{\rm ss}_{\sf K}$ is equal to 
\begin{align*}
\lvert \bfG(\mathbb Z/N\mathbb Z)\rvert  \cdot \prod_{v \mid p}
\begin{pmatrix}
f_{v}
\\ 
\lfloor f_{v}/2 \rfloor 
\end{pmatrix}^m 
\cdot 
\frac{(-1)^{dm(m+1)/2}}{2^{md}}
\cdot \prod_{j=1}^{m}
\zeta_F(1-2j) \cdot \prod_{v \mid \Delta'}\lambda_v,
\end{align*}
where $\zeta_F(s)$ is the Dedekind zeta function
 of $F$, and for $v \mid \Delta'$, 
 \begin{align}\label{eq:intro}
     \lambda_v= 
     \begin{dcases*}
        \prod_{i=1}^m (q_v^i+(-1)^i) & if $m$ is odd, or $v \nmid p$ and $\ord_{\Pi_v}(\gamma)$ is odd; 
        \\ 
\prod_{i=1}^{m/2} (q_v^{4i-2} -1) & 
otherwise. 
     \end{dcases*}
 \end{align}
\end{thm}
Here we give a sketch of the proof. We first discuss the affine Deligne-Lusztig variety $X_{\mu}(b)$ corresponding to the supersingular locus (Section \ref{ss:adlv}). 
This variety admits an action of the Frobenius twisted centralizer group  $J_b(\Q_p)$.  
The set of orbits of the irredicible  componenets $\Irr(X_{\mu}(b))$ under the action of $J_b(\Q_p)$ is in natural
bijection with the “Mirkovic-Vilonen basis” of a certain weight space of a representation of the
dual group of $\bfG_{\Q_p}$, which we will describe explicitly.  
Next we  describe the supersingular locus via the $p$-adic uniformization theorem of Rapoport and Zink \cite{RZ} as a quotient of the affine Deligne-Lusztig variety $X_{\mu}(b)$. 
Then the number of irreducible components of the supersingular  locus can be written as the cardinality  of the set $J_b(\Q_p)\backslash \Irr(X_{\mu}(b))$ multiplied by  the  {\it mass} of an inner form $I$ of $\bfG_{\Q}$. 
Here, the mass of $I$ with respect to an open compact subgroup $U$ of $I(\A_f)$  
is defined
as a weighted cardinality of the double coset space $I(\Q) \backslash I(\A_f)/U$ (Section \ref{ss:mass}).   
Finally we give an explicit formula for the mass with respect to the subgroup whose local factor at $p$ is the stabilizer of an irreducible component of $X_{\mu}(b)$ and factors outside $p$ are  the stabilizers of  self-dual lattices in skew-Hermitian modules. 
We note that our method also applies to the basic locus of  a ${\rm GU}(r,s)$ Shimura variety (of type A) associated to an imaginary quadratic field \cite{TY2}.

%This paper is orgnized as follows.

\notation 
All schemes are assumed to be locally Noetherian. If X is a scheme (resp.~a module) over a commutative ring $R$ and $R\to R'$ is a homomorphism of commutative rings, write $X_{R'}$ for $X\otimes_R R'$. Denote by $\bbN$ the set of positive integers, $\A$ the adele ring of $\Q$, and $\A_f$ the finite adele ring of $\Q$. 
%For a $\Z$-module or $\Q$-module $N$, write $N_p$ for $N\otimes_{\Z} \Zp$ or $N\otimes_{\Q} \Qp$, respectively. 
If $F$ is a number field with ring of integers $O_F$, denote by $F_v$ its completion at a place $v$ of $F$ and $O_{F_v}$ the completion of $O_F$ at $v$. 
For an $F$-module or $O_F$-module $N$, write $N_v$ for $N\otimes_F F_v$ or $N\otimes_{O_F} O_{F_v}$, respectively.

\thank Part of the present work was carried over during the authors' stay at the Korea Institute for Advanced Study. They thank Professor Youn-Seo Choi for his kind hospitality and the institute for excellent working conditions. 
Terakado is partially supported by JSPS KAKENHI Grant Number 23K19014. 
Xue is partially supported by the National Natural Science Foundation of China grant No.~12271410 and No.~12331002. 
Yu is partially supported by the NSTC grant 112-2115-M-001-010 and the Academia Sinica Investigator Grant AS-IA-112-M01.

%Yu is partially supported by the grant MoST 109-2115-M-001-002-MY3.
%and the AS IVA grant .  

\section{Local lattices}
\label{sec:LC}

We consider some variants of lattices in Hermitian spaces over quaternion algebras over  local
fields that are used in this paper. 
%As usual,
%$\delta_{ij}$ denotes Kronecker's
%symbol, namely $\delta_{ij}=1$ or $0$ according to whether $i=j$ or not. 
Let $F_0$ be a
non-Archimedean local field of \ch not equal to two, with ring of integers $O_{F_0}$. 
Let 
$F/F_0$ be a finite separable field extension, with ring of integers $O_F$. 
We fix a uniformizer $\pi$ of $O_F$. 
Let  $\mathfrak 
 D^{-1}_{F/F_0}$ be the inverse different  of
  $F/F_0$. 
%Let $\delta \in F$ be an element such that   $\mathfrak D^{-1}_{F/F_0}=\delta^{-1} O_F$. 

  \subsection{Lattices in symplectic spaces}\label{ss:split}
\begin{defn}\label{def:sym}
    A \emph{symplectic $F$-space} is a pair $(V, \phi_F)$, where $V$ is a finite dimensional $F$-space and $\phi_F: V \times V \to F$ is a non-degenerate alternating $F$-bilinear pairing. 
    An {\it $O_F$-lattice $\Lambda$ in} $(V, \phi_F)$  is a finite free $O_F$-submodule $\Lambda \subset V$ such that  $F \Lambda= V$.  Its {\it dual lattice}  is 
    \[\Lambda^{\vee, \phi_F}\coloneqq \{x\in V\ \mid  \phi_F(x,\Lambda)\subset O_{F} \}.\]
    A lattice $\Lambda$ is said to be  \emph{self-dual} 
if $\Lambda^{\vee, \phi_F}=\Lambda$. 
    Two symplectic $F$-spaces $(V, \phi_{F})$ and $(V', \phi'_{F})$ (resp.~$O_F$-lattices $(\Lambda, \phi_{F})$ and $(\Lambda', \phi'_{F})$) are said to be \emph{isomorphic} if there is an $F$-linear isomorphism $f: V\to V'$ (resp.~an $O_F$-linear isomorphism $f : \Lambda\to \Lambda'$) which preserves the pairings. 
\end{defn}
For any integer $n \geq 1$, 
there exists a unique symplectic $F$-space of dimension $2n$. 
The classification of $O_F$-lattices is also well-known:
\begin{lemma}\label{symclass}
 Let $\Lambda$ be an $O_F$-lattice in a symplectic $F$-space $(V, \phi_F)$ of dimension $2n$. 
Then there exist a sequence of integers $d_1 \leq \cdots\leq  d_n$ and an $O_F$-basis $e_1, \ldots, e_{2n}$ of $\Lambda$ such that 
\begin{itemize}
    \item $\phi_F(e_i, e_j)=0$ except for $j-i=\pm n$; and 
    \item 
    $\phi_F(e_i, e_{n+i})=\pi^{d_i}$ for $i=1, \ldots, n$.
\end{itemize}
Moreover, the sequence $(d_1, \ldots, d_n)$ determines $\Lambda$ uniquely up to isomorphism. 
\end{lemma}
 Let $(V, \phi_F)$ be a symplectic $F$-space of dimension $2n$. 
 We define an $F_0$-group $\GSp_{F_0}(V, \phi_F)$  by 
\begin{align}\label{eq:gsp}
 \GSp_{F_0}(V, \phi_F)(R)= \{g\in \End_{F \otimes_{F_0} R}(V_R)  \mid  \exists\, c(g)\in R^\times
  \ \text{s.t.}\ \phi_F(gx,gy)=c(g)\phi_F(x,y), \ \forall\, x,y\in
  V_R \} 
\end{align}
for any commutative $F_0$-algebra $R$. 
We define an $F_0$-group $\Sp_{F_0}(V, \phi_F)$ by  the exact sequence 
\[ 1\longrightarrow \Sp_{F_0}(V, \phi_F) \longrightarrow \GSp_{F_0}(V, \phi_F) \xlongrightarrow{c} \bbG_{{\rm m}, F_0} \longrightarrow 1.\]
By the definition, we have $\Sp_{F_0}(V, \phi_F)\simeq \Res_{F/F_0}(\Sp_{2n, F})$, where ${\rm Sp}_{2n, F}$ is the symplectic group over $F$ and $\Res_{F/F_0}$ is the Weil restriction of scalars from $F$ to $F_0$.

\begin{lemma}\label{c_sym}  
Let $\Lambda$ be an $O_F$-lattice in a symplectic $F$-space $(V, \phi_F)$, and   
  $\Stab \Lambda$ its stabilizer in  $\GSp_{F_0}(V, \phi_F)(F_0)$. Then  the homomorphism $c: \GSp_{F_0}(V, \phi_F)(F_0) \to F_0^{\times}$ maps $\Stab \Lambda$ onto  $O_{F_0}^{\times}$. 
\end{lemma}
\begin{proof}
By Lemma \ref{symclass}, there exists 
    an $O_F$-basis $e_1, \ldots, e_{2n}$ of $\Lambda$ such that 
    $\phi_F(e_i, e_j)=0$ 
    except for $j-i=\pm n$. 
    For any $t \in O_{F_0}^{\times}$, 
    we put $g = \diag(1^n, t^n) \in \GL_{2n}(O_F)=\Aut_{O_F}(\Lambda)$. 
   Then $\phi_F(gx, gy)=t \cdot \phi_F(x, y)$ for all $x, y \in \Lambda$. 
This implies that  $g \in \Stab \Lambda$ and  $c(g)=t$. 
\end{proof}

%First, suppose that $B=\Mat_2(F)$.  By  Lemma~\ref{LC.5}, there exists
%a $B$-basis $\{e_1, \cdots, e_n\}$ of $V$  such that
%$\varphi_B(e_i, e_j)=\delta_{ij}$. 
%We put $\Lambda=\sum_{i=1}^n \grb e_i$,
%where $\grb$ is the fractional left $O_B$-ideal generated by
%$\left(\begin{smallmatrix}
 % 1 & 0 \\ 0 & 1/\delta
%\end{smallmatrix}\right)$. 
%Then a straightforward calculation shows 
%$\Lambda^\vee=\Lambda$. To show its uniqueness, we claim  
%that  a self-dual lattice $\Lambda$ in $(V,
%  \psi)$ is always a maximal lattice with scale  $\mathfrak a$ in $(V,
%  \varphi_B)$.  
 %Here, the scale $s(\Lambda)$ of a lattice  $\Lambda$ in $(V, \varphi_B)$ is
% the two-sided fractional $O_B$-ideal generated by the
%elements $\varphi_B(x, y)$ for all $x, y\in \Lambda$,  %similar to the symplectic case.
%\footnote{We are following 
%Jacobowitz's terminology  \cite[\S4]{Jacobowitz-HermForm}. 
%The scale $s(\Lambda)$ here is called the \emph{norm} of $\Lambda$ by 
%Shimura in 
 % \cite[\S2.3]{Shimura1963-AltHermForms}.}
%and  we say that  $\Lambda$ is
%maximal if it is a maximal one among the   lattices in $V$ with
%the same scale.  
%  Indeed,  we have  
%$s(\Lambda)\subseteq \mathfrak a$ by  %(\ref{eq:LC.15})
%,
%and there exists a maximal lattice
%$M\supseteq \Lambda$ with $s(M)=\mathfrak a$ thanks to 
%\cite[Proposition~2.10]{Shimura1963-AltHermForms}. It satisfies that   $\Lambda=\Lambda^\vee\supseteq
%M^\vee\supseteq M \supseteq \Lambda$, and hence  $\Lambda=M$.    
%This verifies the claim,  and   the uniqueness follows from 
 %\cite[Propositions~2.11]{Shimura1963-AltHermForms}.
\subsection{Hermitian lattices over division quaternion algebras}\label{ss:nsp}
In this and next subsections, 
let 
$B$ be a quaternion $F$-algebra (i.e.~a central simple $F$-algebra of dimension $4$). 
Let $x\mapsto \bar x\coloneqq \Tr_{B/F}(b)-b$ denote the 
canonical involution on   $B$. 
Let $*$ be an involution on $B$ of the first kind, that is, it fixes
$F$ element-wisely. 
We assume that $*$ is an orthogonal involution  on
  $B$ \cite[Definition~2.5]{book-of-involution}. 
  Then there exists an
  element $\gamma\in B^\times$ such that
  \begin{equation}
    \label{eq:LC.30}
    \bar \gamma+\gamma=0 \quad \text{and} \quad b^*= \gamma \bar b
    \gamma^{-1} \quad \text{for all} \quad  b\in B.
  \end{equation} 
Let $O_B$ be a maximal order in $B$ that is stable under $*$. 
  
In this subsection, we assume that $B$ is a  division algebra.
Then $O_B$ is the unique 
maximal order. We choose a uniformizer $\Pi$ of $O_B$ such that $\Pi \cdot \bar{\Pi}=\pi$. 
In this case, the group 
$B^\times$ normalizes $O_B$. 

\begin{defn}\label{def:quat}
A {\it Hermitian $(B, \bar{\cdot})$-module} is a pair $(V, \varphi_B)$, where 
$V$ is a finite free left $B$-module and $\varphi_B : V \times V \to B$ is a non-degenerate pairing such that 
\begin{equation}
\varphi_B(y,x)=\overline{\varphi_B(x, y)}
    \quad \text{and} \quad 
   \varphi_B(ax, by)=a \varphi_B(x, y)\bar{b},  
   \quad 
   {\text{for all}} \ 
a,b \in B, \ 
x, y \in V.
\end{equation}
A {\it Hermitian $(O_B, \bar{\cdot})$-lattice} $(\Lambda, \varphi_B)$ in $(V, \varphi_B)$ is a finite free left $O_B$-submodule $\Lambda \subset V$ such that  $B \Lambda= V$.  
 We often omit to mention the underlying space $V$. 
 The \emph{dual lattice of $\Lambda$} is 
\begin{equation}\label{eq:dual}
      \Lambda^{\vee, \varphi_B}\coloneqq \{x\in V\ \mid  \varphi_B(x,\Lambda)\subset O_{B} \}. 
\end{equation} 
\end{defn}
For any integer $n \geq 1$, there is a unique quaternion $(B, \bar{\cdot})$-module $(V, \varphi_B)$ of rank $n$ up to isomorphism \cite[Theorem 3.1]{Jacobowitz-HermForm}.

Let $i$ be an integer. 
We write  $(\pi^i)$ for the rank-one Hermitian $(O_B, \bar{\cdot})$-lattice  equipped  with a basis $e$ over $O_B$ and a form $\varphi_B$ such that  
  $\varphi_B(e, e)=\pi^i$.  
%If $\Lambda$ has an orthogonal basis $x_1, \ldots, x_n$ with $\varphi_B_B(x_j, x_j)=(a_j)$, we will write  
%\[ \Lambda \simeq (a_1) \oplus \cdots \oplus (a_n).\]
Further, the \emph{hyperbolic plane} $H(i)$ is defined  as the rank-two lattice equipped with basis $e, f$ and a form $\varphi_B$ such that 
\begin{align*}
    \begin{pmatrix}
        \varphi_B(e, e) & \varphi_B(e, f)
        \\
        \varphi_B(f, e) & \varphi_B(f, f)
    \end{pmatrix}
    = 
    \begin{pmatrix}
         0 & \Pi^i
         \\
         \bar{\Pi}^i  & 0
    \end{pmatrix}. 
\end{align*}  
 According to  \cite[\S4 and Proposition~6.1]{Jacobowitz-HermForm}, any Hermitian $(O_B, \bar{\cdot})$-lattice  admits a splitting 
\begin{equation}\label{eq:Jordan}
    \Lambda \simeq  \bigoplus _{i \in \Z}\Lambda_i, \quad 
    \Lambda_i 
 = \begin{dcases*}
     (\pi^{i/2}) \oplus \cdots \oplus (\pi^{i/2}) & if $i$  is even; 
     \\
     H(i) \oplus \cdots \oplus H(i) 
     & if $i$ is odd.  
    \end{dcases*}
\end{equation}
Note that if $i$ is even then   $(\pi^{i/2}) \oplus (\pi^{i/2}) \simeq H(i)$. 
Further, we have that $(\pi^i)=\Pi^{2i}\cdot (\pi^i)^{\vee, \varphi_B}$ and $H(i)=\Pi^{i} \cdot H(i)^{\vee, \varphi_B}$.
These imply  the following: 
\begin{lemma}\label{quatclass}
{\rm (1)}  
 There  exists a Hermitian 
  $(O_B, \bar{\cdot})$-lattice $\Lambda$ of  rank $n$ such that    $\Lambda=\Pi^i \cdot \Lambda^{\vee, \varphi_B}$  if and only 
  if either  
  $n$ or $i$ is even. 
 
  {\rm (2)} 
  Such a  lattice is unique up to isomorphism if exists, and written as 
  \begin{equation}\label{eq:modular}
      \Lambda \simeq 
      \begin{dcases*}
         H(i) \oplus \cdots \oplus H(i) & if $n$ is even; 
         \\ 
         H(i) \oplus \cdots  \oplus H(i) \oplus (\pi^{i/2})  
         & if $n$ is odd (and $i$ is even).
      \end{dcases*}
  \end{equation}
\end{lemma}
Let $(V, \varphi_B)$ be a Hermitian $(B, \bar{\cdot})$-module. We define an  $F_0$-group $\GU_{F_0}(V, \varphi_B)$  by 
\begin{align}\label{eq:nqs}
 \GU_{F_0}(V, \varphi_B)(R)= \{g\in \End_{B \otimes_{F_0} R}(V_R)  \mid  \exists\, c(g)\in R^\times
  \ \text{s.t.}\ \varphi_B(gx,gy)=c(g)\varphi_B(x,y), \ \forall\, x,y\in
  V_R \} 
\end{align}
for any commutative $F_0$-algebra $R$. 
We also define an $F_0$-group ${\rm U}_{F_0}(V, \varphi_B)$ by  the exact sequence 
\[1\longrightarrow {\rm U}_{F_0}(V, \varphi_B)  \longrightarrow \GU_{F_0}(V, \varphi_B) \xlongrightarrow{c} \bbG_{{\rm m},F_0} \longrightarrow 1.\] 
%\begin{align}\label{eq:u}
%{\rm U}_{F_0}(V, \varphi_B)(R) = \{g \in \GU_{F_0}(V, \varphi_B)(R) \mid c(g)=1\}.
%\end{align}
 \begin{lemma}\label{c}
Suppose that $B$ is a division  algebra. 
Let $(\Lambda, \varphi_B)$ be a Hermitian 
  $(O_B, \bar{\cdot})$-lattice,  in $V=B\Lambda$. 
  Let $\Stab \Lambda$ be the stabilizer of $\Lambda$ in $\GU_{F_0}(V, \varphi_B)(F_0)$. 
  Then the homomorphism $c : \GU_{F_0}(V, \varphi_B)(F_0) \to F_0^{\times}$ maps $\Stab \Lambda$ onto $O_{F_0}^{\times}$. 
\end{lemma}
\begin{proof}
Take $t \in O_{F_0}^{\times}$. 
Suppose that there is a splitting $\Lambda\simeq \Lambda_1 \oplus \Lambda_2$,  and  that there are elements $g_j \in \End_{O_B}(\Lambda_j)$ with    $\varphi_B \vert_{\Lambda_j}(g_jx, g_j y)=t \cdot \varphi_B \vert_{\Lambda_i}(x, y)$ for all $x, y\in \Lambda_j$, $j=1, 2$. 
Then the sum $g_1 \oplus g_2$ can be regarded as an element of $\Stab \Lambda$ with similitude factor $t$.
Therefore, 
by    
 \eqref{eq:Jordan},  we may assume   $\Lambda \simeq H(i) = O_Be \oplus O_Bf$ or $\Lambda \simeq (\pi^{i})$ for some $i$. 
In the first case, if we define  an element $g\in \GL_{O_B}(\Lambda)$ by $g e=e$ and $gf=t f$, then   $c(g)=t$. 
In the second case, 
we identify $\Lambda$ with $O_B$, and we regard 
 $O_B$-linear endomorphisms of $\Lambda$ as the right multiplications of  elements of $O_B$. 
Since the reduced
norm  ${\rm Nrd}_{B/F} : O_B^{\times} \to 
O_F^{\times}$  is surjective, 
 there exists an element $u \in O_B^{\times}$ such that $u \bar{u}=t$. 
    For $x, y \in \Lambda$, we have $\varphi_B(u x, u y)=(x \cdot u)(   \ol{y\cdot u})=t x \bar{y} =t  \varphi_B(x, y)$, as desired. 
\end{proof} 

     \subsection{Field-valued skew-Hermitian  lattices over quaternion algebras}\label{sec:LC.2}  
\begin{defn}\label{def:skew}
An {\it $F_0$-valued 
skew-Hermitian
  $(B,*)$-module}
  %with $\epsilon\in\{\pm1\}$ here,  means 
%(or {\it $(B,*)$-module} in order to distinguish the involution
%involved) 
is 
a pair $(V,\psi)$, 
where $V$ is a finite free left $B$-module and $\psi:V\times V\to F_0$ is a
non-degenerate $F_0$-bilinear pairing such that 
\begin{equation}
  \label{eq:e-h}
  \psi(y,x)=- \psi(x,y) \quad \text{and}\quad 
  \psi(a x,y)=\psi(x, a^* y), \quad \text{for all} \ a\in B, \  x,y\in V.
\end{equation}

An {\it $F_0$-valued 
skew-Hermitian 
  $(O_B,*)$-lattice} and its {\it dual lattice $\Lambda^{\vee, \psi}$}  are defined in the same way as in Definition \ref{def:quat}.

\end{defn}
%that is, $x^*=\gamma \bar x
%\gamma^{-1}$ for some $\gamma\in B^\times$ with $\bar \gamma+\gamma=0$. 

% $\psi(b x,y)=\psi(x, b^*
% y)$ for $b\in B$ and $x,y\in V$. 
%We shall drop $*$ from the notation
%if there is no confusion. 

% Suppose that $B=\Mat_2(K)$. Put
% \begin{equation}
%   \label{eq:psiC}
%   \psi_C(x,y)\coloneqq\varphi_B(x,Cy), \quad \text{where}\quad C\coloneqq
%   \begin{pmatrix}
%     0 & 1 \\
%     -1 & 0
%   \end{pmatrix}.
% \end{equation}
% Using $\bar C=-C$ and $C^{-1} \bar b C=b^t$ (the transpose of $b$), 
% one gets
% \begin{equation}
%   \label{eq:psiC2}
%   \psi_C(y,x)=\epsilon\psi_C(x,y) \quad \text{and} \quad
%   \psi_C(bx,y)=\psi_C(x,b^ty), \quad \forall\,b\in B,\ x,y\in V. 
% \end{equation}
% In other words, $(V,\psi_C)$ is a $K_0$-valued $(\epsilon,t)$-Hermitian 
% $B$-module. We have a decomposition of $V$ into $K$-subspaces with
% respect to the pairing $\psi_C$:
% \begin{equation}
%   \label{eq:dec}
%   V=V_1\oplus V_2, \quad V_i=e_i V, \ i=1,2, 
% \end{equation}
% where $e_1=\diag(1,0)$ and $e_2=\diag(0,1)$. 
% % and
% % $V_i\coloneqqe_iV$ for $i=1,2$. We have the decomposition
% %of $F$-vector subspaces which respects $\psi_C$. 
% Thus, the classification of $K_0$-valued 
% $(\epsilon,*)$-Hermitian $B$-modules can be reduced to that of
% $K_0$-valued symmetric or alternating $K$-modules $(V_1,\psi_C)$.  

For an $F_0$-valued skew-Hermitian $(B, *)$-module $(V, \psi)$, we define $F_0$-groups $\GU_{F_0}(V, \psi)$ and ${\rm U}_{F_0}(V, \psi)$ in the same way as in \eqref{eq:nqs}.

  \subsubsection*{The split case}
  Now we assume  that $B$ is the matrix algebra. 
 We can take  an isomorphism $B \simeq \Mat_2(F)$ which identifies  $O_B$ with $\Mat_2(O_F)$. 
Let $\gamma \in B^{\times}$ be as in \eqref{eq:LC.30}. 
Then 
$\gamma$ normalizes $\Mat_2(O_F)$  and hence belongs to 
$F^\times \cdot  \GL_2(O_F)$. 
Without changing $*$, we may assume that
$\gamma\in \GL_2(O_F)$.  
% As in the previous subsection, we can write 
% $\gamma=\pi^ru$ for an element $u \in O_B^{\times}$ with $\bar{u}=-u$ and an integer $r$. 
For $g \in B$, 
let $g \mapsto g^t$ denote the  transpose.  
 We write  $C \coloneqq \begin{psmallmatrix}
         0 & 1
         \\
         -1  & 0
    \end{psmallmatrix} \in B$. 
    Then $\bar{C}=-C$ and $C^{-1} \bar{x} C= 
    x^t$.  

    Let $(\Lambda, \psi)$ be an $F_0$-valued skew-Hermitian $(B, *)$-lattice, in $V=B\Lambda$. 
   Let 
    \[ \tilde{\psi}(x, y) \coloneqq \psi(x, \gamma C y).\]
    Then $(V, \tilde{\psi})$ is an $F_0$-valued skew-Hermitian $(B, t)$-module. 
    Indeed, we have 
    \[ (\gamma C)^{-1}x^*(\gamma C)=C^{-1} \bar{x}C=x^t, \quad 
    \gamma^*=\gamma \bar{\gamma}\gamma^{-1}=-\gamma.\]

   Now we take an element $\delta \in O_F$  such that  $\mathfrak D_{F/F_0}^{-1}=\delta^{-1} O_F$. 
   We define $\tilde{\psi}_F : V \times V \to F$ as the unique $F$-bilinear alternating pairing such that 
    \begin{equation}\label{eq:lift}
    \tilde{\psi}(x, y)= 
    \Tr_{F/F_0}
\big( 
\delta^{-1}  \tilde{\psi}_F(x, y) \big), \quad x, y \in V.
\end{equation}
Then  $(V, \tilde{\psi}_F)$ is an $F$-valued skew-Hermitian $(B, t)$-module. 

Let $(\Lambda, \tilde{\psi}_F)$ be the restriction of $\tilde{\psi}_F$ to $\Lambda$. 
Then the assignment $(\Lambda, \psi) \mapsto (\Lambda, \tilde{\psi}_F)$ gives an equivalence of categories between the category of $F_0$-valued skew-Hermitian $(O_B, *)$-lattices and the category  of $F$-valued skew-Hermitian $(O_B, t)$-lattices, which preserves direct sums. 
Moreover, we have 
   \[ \Lambda^{\vee, \psi}=\{x\in V \mid \delta^{-1} \tilde{\psi}_F(x,\Lambda)
     \subseteq 
    \mathfrak D_{F/F_0}^{-1}\}=\{x\in V \mid  \tilde{\psi}_F(x,\Lambda)
     \subseteq 
   O_F\}=\Lambda^{\vee, \tilde{\psi}_F}.\]
 In particular, $\Lambda$ is self-dual with respect to $\psi$ if and only if it
 is so with respect to $  \tilde{\psi}_F$.
 
Finally let $V_1 \coloneqq \begin{psmallmatrix}
         1 & 0
         \\
         0  & 0
    \end{psmallmatrix}V$, regarded as an $F$-space,  
    and let  $\phi_F$ be the restriction of  $\tilde{\psi}_F$ to  $V_1$. 
   We similarly define an $O_F$-lattice $(\Lambda_1, \phi_F)$. 
    By  Morita equivalence, this     assignment $(\Lambda, \tilde{\psi}_F) \mapsto (\Lambda_1, \phi_F)$ gives an 
   equivalence of categories between the category of  $F$-valued skew-Hermitian $(O_B, t)$-lattices to the category   of $O_F$-lattices  in symplectic $F$-spaces, which preserves  direct sums and self-dual lattices. 

By Lemma \ref{symclass},  there exists a unique 
 self-dual $O_F$-lattice in a symplectic $F$-space of rank   $2n$ for each $n\geq 1$.   
%(see for instance \cite[Ch.~I, Corollary 3.5]{MH}). 
This and the above construction imply  the following. 
\begin{prop}\label{sd-split}
    Suppose that $B$ is the matrix algebra. 
Then, for each $n \geq 1$,  there exists a unique $F_0$-valued skew-Hermitian $(B, *)$-module of rank $n$ up to isomorphism. 
The same is true for a    self-dual $(O_B, *)$-lattice. 
\end{prop}
By  Morita equivalence, there are isomorphisms of $F_0$-groups 
\begin{align}\label{eq:morita}
\GU_{F_0}(V, \psi) \simeq \GSp_{F_0}(V_1, \phi_F) \quad \text{and} \quad {\rm U}_{F_0}(V, \psi) \simeq \Sp_{F_0}(V_1, \phi_F). 
\end{align}
  In particular, we have ${\rm U}_{F_0}(V, \psi)(F_0) \simeq \Sp_{2n}(F)$ where $n=\rank_{B}V$. 

\subsubsection*{The non-split case} 
We next assume that $B$ is a division algebra.  Let $(\Lambda, \psi)$ be an $F_0$-valued skew-Hermitian 
  $(O_B, *)$-lattice, in $V=B \Lambda$.  Put
%\begin{equation}
 % \label{eq:varphi}
  $\varphi(x,y)
  \coloneqq
  \psi(x,\gamma y)$,  where  $\gamma\in B^\times$ is  
 defined as in 
  (\ref{eq:LC.30}). 
%\end{equation}
It follows from 
  \eqref{eq:LC.30} and \eqref{eq:e-h}   that 
\begin{equation*}
 % \label{eq:varphi_r}
  \varphi(y,x)=\varphi(x,y)\quad \text{and} \quad
\varphi(bx,y)=\varphi(x, \bar b y) \quad \text{for all} \quad b\in B, \ x,y\in V.
\end{equation*}
%In other words, $(V,\varphi_B)$ is an $F_0$-valued 
%$(-\epsilon,\bar{\cdot})$-Hermitian $B$-module. 
Let 
    $\varphi_B :V\times V\to B$ 
be the unique Hermitian $(B, \bar{\cdot})$-form   such that 
\[\varphi(x,y)=\Tr_{B/F_0} (\delta^{-1}\varphi_B(x,y)), \quad x, y \in V.\] 
Let    $(\Lambda, \varphi_B)$ be   the restriction. 
The assignment $(\Lambda, \psi) \mapsto (\Lambda, \varphi_B)$ gives an equivalence of categories  between the category  of $F_0$-valued skew-Hermitian $(O_B, *)$-lattices and that of Hermitian $(O_B, \bar{\cdot})$-lattices. 
%It preserves direct sums.  

%Following the usual
%convention, $\varphi_B_B$ is assumed to be $B$-linear in its first variable.  
%Recall that we define an element $\delta \in F$ by $\mathfrak D^{-1}_{F/F_0}=\delta^{-1}O_F$. 
The inverse  different  $\mathfrak D_{B/F}^{-1} \coloneqq 
\{ x \in B \mid  \Tr_{B/F} (xO_B)
\subset 
O_F\}$ 
is generated by $\Pi^{-1}$ as a fractional ideal.  
It follows that  for any  element  $x\in V$   
\begin{equation*}
 % \label{eq:LC.12x}
  \psi(x,\Lambda)\subset O_{F_0}\iff \varphi(x,\gamma^{-1} \Lambda)\subset
  O_{F_0}\iff \bar{\gamma}^{-1}\varphi_B(x,\Lambda)\subset  \Pi^{-1}O_B. 
\end{equation*}
%\begin{align} 
%\label{eq:a}\Pi^r O_B=\bar{\gamma}\delta^{-1}\mathfrak D_{B/F}^{-1}
% =\bar{\gamma}\delta^{-1} \Pi^{-1}
%   O_B.
 %  \end{align}
%This  is a principal two-sided
  % fractional $O_B$-ideal. 
   %   Let $\Lambda^{\vee, \psi}$ (resp.  $\Lambda^{\vee, \varphi_B_B}$) be the dual lattice of $\Lambda$ with respect to $\psi$ (resp. $\varphi_B_B$). 
This implies that
\begin{equation}
  \label{eq:LC.15}
     \Lambda^{\vee, \psi}=\{x\in V \mid \varphi_B(x,\Lambda)\subset \bar{\gamma} \Pi^{-1}O_B \}= \bar{\gamma}\Pi^{-1} \Lambda^{\vee, \varphi_B}
     =\Pi^{\ord(\gamma)-1}
     \Lambda^{\vee, \varphi_B}, 
\end{equation}
where  $\ord : B^{\times} \to \Z$ denotes the valuation on $B^{\times}$ 
normalized by $\ord(\Pi)=1$. 
Hence, for any integer $i$,  we have that $\Lambda=\Pi^i\Lambda^{\vee, \psi}$ if and only if $\Lambda=\Pi^{i+\ord(\gamma)-1}\Lambda^{\vee, \varphi_B}$. 
%By \eqref{eq:LC.15}, a lattice $(\Lambda, \psi)$ is  self-dual if and only if $\Lambda=\Pi^{\ord(\gamma)-1}  \Lambda^{\vee, \varphi_B}$.  

This argument  and  Lemma \ref{quatclass} imply   the following.

\begin{prop}\label{prop:herm-self-dual}
Suppose that $B$ is a division quaternion algebra. 
Then, for each $n \geq 1$, there exists a unique $F_0$-valued skew-Hermitian $(B, *)$-module $(V, \psi)$ of rank $n$ up to isomorphism. 
Further, for any integer $i$, 
there  exists an  
  $(O_B, *)$-lattice of rank $n$ with $\Lambda=\Pi^{i}\Lambda^{\vee,\psi}$ if and only 
  if either 
  $n$ is even or  $i+\ord(\gamma)$ is odd.   
  If this condition holds, then such a  lattice is unique up to isomorphism. 
\end{prop}
Since   $\gamma$ commutes with elements of $\GU_{F_0}(V, \psi)$,  
 we have isomorphisms of $F_0$-groups 
\begin{equation}\label{eq:gu}
  \GU_{F_0}(V, \psi) = \GU_{F_0}(V, \varphi_B) \quad \text{and} \quad  {\rm U}_{F_0}(V, \psi) = {\rm U}_{F_0}(V, \varphi_B). 
\end{equation}
\begin{remark}
    The first half of each of Propositions \ref{sd-split} and \ref{prop:herm-self-dual} is a special case of \cite[Propositions~2.1 and 3.3]{Shimura1963-AltHermForms}.
\end{remark}
 % By Lemmas and \ref{c}, we have the following: 
%  \begin{lemma}
  %    Let $\Lambda$ be an  $F_0$-valued skew-Hermitian 
%  $(O_B, *)$-lattice in $V=B\Lambda$, and 
% write   
 % $\Stab \Lambda$ for its stabilizer in  $\GU_{F_0}(V, \psi)(F_0)$. Then the homomorphism  
 % $c : \GU_{F_0}(V, \psi)(F_0) \to F_0^{\times}$ maps $\Stab \Lambda$ onto $O_{F_0}^{\times}$.
 % \end{lemma}
 % \begin{lemma}\label{LC.5}
 % For any integer $n\ge 1$,  there is a unique $F_0$-valued skew-Hermitian
 % $(B, *)$-module $(V,\psi)$ of rank $n$ up to isomorphism.
%\end{lemma}
% \begin{proof}
%   This is a well-known result and follows immediately from Kneser's Theorem
%   on the vanishing of the Galois cohomology $H^1(F,G)$ 
%   for any semi-simple simply connected
%   group $G$. Alternatively, we can also reduce  
%   to classify either $(+,\text{-})$-Hermitian
%   $B$-modules $(V,\varphi_B_B)$ or $F_0$-valued alternating
%   $K$-modules $(V_1,\psi_C)$ as above, according as $B$ 
%   is a division algebra or
%   not. The latter classifications are elementary and there
%   is only one isomorphism class in each case.
% \end{proof}

  %\begin{proof}
  %  This is  
   % Alternatively, the following argument  combining with the fact that 

   %immediately from Kneser's Theorem
%  on the vanishing of the Galois cohomology $H^1(F_0,G)$ 
%  for any semi-simple simply connected
%  group $G/F_0$ ? ${\rm U}_{F_0}(V, \psi)$?.
 % \end{proof}
\section{Moduli spaces and the superspecial locus} \label{ssnempt}
% Non-emptiness of the superspecial loci of quaternionic Shimura varieties

%\subsection{Skew-Hermitian forms over totally indefinite quaternion algebras}
\subsection{Integral PEL datum of type C}\label{ss:PEL}
Let $F$ be a totally real number field of degree $d$ with ring of integers $O_F$. 
Let $B$ be a  quaternion $F$-algebra which is totally indefinite (i.e.~$B\otimes_F \R \simeq {\Mat}_2(\R)$  for any real embedding $F \hookrightarrow \R$). 
Let $b\mapsto \bar b \coloneqq {\rm Tr}_{B/F}(b)-b$ denote the canonical  involution of $B$. 
We assume that $B$ is equipped with a positive involution $*$, that is, an involution such that ${\rm Tr}_{B/\Q}(b b^*)>0$ for any $x \in B-\{0\}$. 
Then, 
as in \cite[Section 21]{mumford:av},  there is an element $\gamma\in B^\times$ such that $\gamma^2\in F$  is  totally negative in $F$ and 
\begin{equation}
\label{eq:gamma}
    \bar \gamma+\gamma=0 \quad \text{and} \quad b^*= \gamma \bar b
    \gamma^{-1} \quad  {\text{for all}} \quad b\in B.
\end{equation}
Moreover we can choose an isomorphism $B  \otimes_{\Q} \R \simeq \Mat_2(\R)^d$ 
carrying the involution $*$ 
into the involution $(X_1, \ldots, X_{d}) \mapsto (X_1^t, \ldots, X_d^t)$. 
Let $O_B$ be a maximal $O_F$-order in $B$ which is stable under $*$. 
 
A $\Q$-valued skew-Hermitian $(B, *)$-module $(V, \psi)$ is   defined in the same way as in Definition \ref{def:skew}. 
For each positive integer $m$, 
there exists a unique $\Q$-valued skew-Hermitian $(B, *)$-module of rank $m$ up to isomorphism, by Propositions~\ref{sd-split} and \ref{prop:herm-self-dual}. 
An $O_B$-lattice $\Lambda$ in $V$ is said to be {\it self-dual} (with respect to $\psi$) if $\Lambda= \Lambda^{\vee, \psi} \coloneqq \{ x \in V \mid  \psi(x, \Lambda) \subset \Z\}$. 

For any commutative $\Q$-algebra $R$, 
we write $V_R  \coloneqq  V\otimes_\Q R$, and write $\End_{B\otimes_\Q R}(V_R)$ for the ring of $B\otimes_\Q R$-linear endomorphisms of $V_R$.
We define a $\Q$-group $\bfG=\GU_{\Q}(V, \psi)$ by 
%be the $\Q$-group of $B$-linear 
%$\psi$-similitudes on $V$, defined by 
\begin{align}
  \label{eq:defG}
  \bfG(R) =\{g\in \End_{B \otimes_{\Q}R} (V_R)  \mid  \exists\, c(g)\in R^\times
  \ \text{s.t.}\ \psi(gx,gy)=c(g)\psi(x,y), \ \forall\, x,y\in
  V_R \}. 
\end{align}
The group $\bfG$ is connected and reductive. 
Further it 
 satisfies the Hasse principle, that is, the local-to-global map $H^1(\Q, \bfG)\to \prod_{v \le \infty} H^1(\Q_v, \bfG)$ is injective (\cite[Section 7]{Kottwitz}). 

%The similitude factors define a character $c : \bfG \to \Gm$. 
%We put   $\bfG^1 \coloneqq \ker c$. 
We define a $\Q$-group $\bfG^1={\rm U}_{\Q}(V, \psi)$ by  the exact sequence 
\[1 \longrightarrow \bfG^1  \longrightarrow \bfG \xlongrightarrow{c} \bbG_{{\rm m}, \Q} \longrightarrow 1,\]
where $c$ denotes the similitude character.   
 This group $\bfG^1$ is  semi-simple and simply-connected. 
 From Kneser's theorem \cite[Theorem 6.4, p.~284]{PR} it follows that   $H^1(\Q_{\ell}, {\bf{G}}^1)=1$ for any prime $\ell$.  
 This fact and the  above exact sequence imply that 
\begin{equation}\label{eq:sim}
c(\bfG(\Q_{\ell}))=\Q_{\ell}^{\times}.
\end{equation}

%\begin{prop}\label{prop:sd}
%Let $(B, *, O_B)$ be as above, $m$ be a positive integer, and  $(V, \psi)$ be the  unique  $\Q$-valued skew-Hermitian $(B, *)$-module of rank $m$. %and $\gamma\in B^\times$ be  as above. 
% there exists a self-dual $O_B$-lattice $\Lambda$ in $(V, \psi)$  if and only if $m$ is even, or for any finite place $v$ of $F$ ramified in $B$ one has $\ord_{\Pi_v}(\gamma)$ is odd. 
  % Under these conditions,  such a lattice $\Lambda$ is unique up to isomorphism.  
%end{prop}
 
\begin{defn}\label{def:PELC}
An \emph{integral PEL datum of type C} is a septuple 
$\mathscr D=(B,*,O_B,V,\psi, \Lambda, h_0)$  where 
\begin{itemize}
    \item[(i)] $(B,*,O_B)$ is as above;
    \item[(ii)]$(V,\psi)$ is a $\Q$-valued skew-Hermitian $(B, *)$-module; 
    \item[(iii)] $\Lambda$ is an  $O_B$-lattice in $V$; 
    \item[(iv)] $h_0:\C \to \End_{B\otimes_\Q \R}(V_\R)$ is an 
%$B\otimes_\Q \R$-linear
$\R$-algebra 
  homomorphism % where $V_\R\coloneqqV\otimes_\Q \R$,  
  such that 
\[ \psi(h_0(i)x, h_0(i)y)=\psi(x,y) \quad {\text{for all}}  \quad  x,y \in
 V_\R, \]
% \coloneqq :=
%for any $x, y\in V_\R=V\otimes\R$, one has 
%  $\phi_F(h_0(i)x, h_0(i)y)=\psi(x,y)$ and 
  and that the symmetric form  $(x,y) \coloneqq \psi(h_0(i)x,y)$ is
  positive definite  on $V_\R$. 
\end{itemize}
A datum $\mathscr D$  is said to be \emph{principal}  
if $\Lambda$ is self-dual with respect to $\psi$. 
\end{defn} 
For a $\Q$-valued skew-Hermitian $(B, *)$-module $(V,\psi)$,  a  map  $h_0$ as in (iv) always exists and the group $\bfG^1(\R)$ acts transitively on the set of all such maps \cite[Lemma 4.3]{Kottwitz}.
The map  $h_0$  endows $V_{\R}$ with a complex structure,  and hence it gives a
decomposition $V_{\C}=V^{-1,0} \oplus V^{0,-1}$ of complex subspaces. 
Here, $V^{-1,0}$ (resp.~$V^{0,-1}$) denotes the subspace where $h_0(z)$ acts by $z$ (resp.~$\bar z$). 

 Let $\char_F(b) \in O_F[T]$ be the reduced characteristic polynomial of $b \in O_B$, and let $\char(b)\coloneqq {\rm Nr}_{F/\Q}\char(b) \in \Z[T]$ be the one 
 from $B$ to $\Q$.
As in \cite[Section 2.3]{yu:D}, the characteristic polynomial of $b\in O_B$  on $V^{-1,0}$ is given by 
\begin{align}\label{eq:char}
    \char(b \mid V^{-1,0})=\char(b)^m \in \Z[T]. 
\end{align}

For an abelian scheme $A$ over a base scheme $S$, let $\End_{S}(A)$ denote the ring of $S$-linear endomorphisms of $A$. 
\begin{defn}\label{def:ab}
 Let $(B,*,O_B)$ be as above. 

 (1) 
An \emph{$O_B$-abelian scheme} over a base scheme $S$ is a pair $(A, \iota)$, where $A$ is an abelian scheme over $S$ and $\iota$  is a monomorphism of rings $\iota:O_B\to \End_S(A)$. 

(2) 
A \emph{(principally) polarized $O_B$-abelian scheme} is a triple $(A,\lambda, \iota)$, where $(A, \iota)$ is an $O_B$-abelian scheme and  $\lambda : A \to A^{t}$ is a (principal) polarization such that $\lambda\circ \iota(b^*) =\iota(b)^t \circ \lambda$. 

(3) The {\it determinant condition} for an $O_B$-abelian scheme $(A, \iota)$ over a $\Z_{(p)}$-scheme is the
equality of characteristic polynomials of degree $2dm$: 
\begin{equation}\label{eq:detA}
\char(\iota(b) \mid \Lie(A))
=\char(b \mid  V^{-1,0}) \in O_S[T] \quad \text{for all} \quad b \in  O_B. 
\end{equation}
\end{defn}
Note that condition \eqref{eq:detA} implies  the $S$-scheme $A$ has relative dimension  $2dm$.

%Note that a triple $(B,*,O_B)$ extends to a principal integral PEL datum of type $C$ of rank $m$ if and only is there exists a self-dual $O_B$-lattice in a skew-Hermitian $B$-module $(V, \psi)$ of rank $m$. 

%\begin{lemma}\label{lem:complex}
%Let $m$ be a positive integer. There exists a  principally polarized $O_B$-abelian variety $(A,\lambda,\iota)$ over $\C$ of dimension $2dm$ if and only if the triple $(B, *, O_B)$ extends to a principal integral PEL datum $\mathscr D=(B,*,O_B,V,\psi, \Lambda, h_0)$ of type C of rank $m$. 
% $\dim_\Q V=4dm$.
%\red{(We should restrict only to principal one, otherwise both statements are automatically true.)}
%\end{lemma}

%This theorem and the above argument imply  Theorem~\ref{main.1}.

\subsection{\dieu modules}
\label{ss:dieu}
Let $k$ be an \ac field of characteristic $p$.
Let $W(k)$ be the ring of Witt vectors over $k$ with the absolute Frobenius morphism  $\sigma :W(k) \to W(k)$. 
  Let $W(k)[{\sf F, V}]$ be the quotient ring of the associative free $W(k)$-algebra generated by the indeterminates ${\sf F, V}$ with respect to the relations 
 \begin{align*}
 {\sf {FV}=VF}=p, \quad {\sf F}a=a^{\sigma}{\sf F}, \quad {\sf V}a^{\sigma}=a{\sf V} \quad  {\text{for all}} \quad a \in W(k).
 \end{align*}
 \begin{defn}\label{def:dmod}

 (1) 
 A \emph{Dieudonn\'e module} $M$ over $k$ is a left $W(k)[{\sfF,\sfV}]$-module which is
 finitely generated and free as a $W(k)$-module. 
 
  (2) 
A \emph{polarization} on a \dieu module $M$ is an alternating form 
 $\<\,,\,\> :  M \times M \to W(k)$
  such that  
  \[\<{\sf F}x,y\>=\<x, {\sf V}y\>^{\sigma} \quad \text{for all} \quad x, y \in M.\] 
 A polarization  $\<\,,\,\>$ is called a \emph{principal polarization} 
 if it is a perfect pairing. 

 (3) 
Let $O$ be a $\Zp$-algebra with an involution $*$. 
An {$O$-\dieu module} $M$ over $k$ is a 
 \dieu module over $k$ endowed with an  $O$-action  commuting  with the operators $\sf F$ and $\sf V$. 
 An $O$-\dieu module $M$ over $k$  is called {\it (principally)  polarized} if it is  
 endowed with a (principal) polarization $\<\, , \, \>$  satisfying   
 $\<bx, y\>=\<x, b^*y\>$ for all $x, y \in M$ and $b\in O$. 
   \end{defn}
   Let $(B,*,O_B)$ be as in Section \ref{ss:PEL}. 
\begin{defn}\label{def:detM}
 An \dieu $O_B\otimes \Zp$-module $M$ of $W(k)$-rank $4dm$ is said to satisfy {\it determinant condition} if the following equality of  polynomials holds:
\begin{equation}\label{eq:detM}
 \char  (b \mid M/{\sf V}M)=\char(b)^m \pmod p \in k[T] \quad \text{for all} \quad  b\in O_B, 
 \end{equation}
 where $\char(b)\in \Z[T]$ is the reduced characteristic polynomial of $b$ from $B$ to $\Q$, cf.~\eqref{eq:char}.  
\end{defn}  
   For an abelian variety $A$ over $k$, let  $A[p^{\infty}]$ be its $p$-divisible group over $k$ and   $M(A)$ the  covariant \dieu module of $A[p^{\infty}]$; see \cite{zink:cartier} for the covariant \dieu theory. 
  %defined by  $M(A) \coloneqq \Hom_{W(k)}(\D(A[p^{\infty}]),\D(\mu_{p^{\infty}}))$ 
  As $k$ is a perfect field, one may identify $M(A)$ with the dual \dieu module $\Hom_{W(k)}(\D(A[p^{\infty}]), \D(\mu_{p^{\infty}}))$,  where  $\D$ denotes the  contravariant Dieudonn\'e functor. 
  A (polarized) $O_B$-abelian variety $A$ over $k$ induces a (polarized) $O_B \otimes \Z_p$-Dieudonn\'e module $M(A)$  over $k$. 
 There is a natural isomorphism $\Lie(A) \simeq M(A)/{\sf V}M(A)$ of $O_B \otimes k$-modules. 
 Therefore,  
 an $O_B$-abelian variety $A$ over $k$ satisfies the determinant condition \eqref{eq:detA} if and only if the associated \dieu module $M=M(A)$ satisfies the determinant condition \eqref{eq:detM}. 

Let $v$ be a  finite place $v$ of $F$. 
Let $F_v$ be the completion of $F$ at
 $v$,  $O_v=O_{F_v}$ the ring of integers, and $\pi_v$ a uniformaizer of $F_v$. 
 Let $e_v$ and $f_v$ denote the ramification index and the inertial degree of $v$, respectively. 
Write 
\begin{equation} \label{eq:Fp}
F \otimes_{\Q}{\Qp}=\prod_{v|p} F_v, \quad O_F\otimes_{\Z} {\Zp} = \prod_{v|p} O_{v}
\end{equation}
as a product of local fields and their ring of integers, respectively. 
Similarly,  let $B_v=B \otimes _F F_v$ and $O_{B_v}= O_{B} \otimes_{O_F}O_{F_v}$. 
Then 
we have 
\begin{equation}\label{eq:dec}
    B\otimes_{\Q} \Qp=\prod_{v|p} B_v, \quad O_B\otimes_{\Z} {\Zp} = \prod_{v|p} O_{B_v}. 
\end{equation} 
A (polarized) \dieu $O_B\otimes \Zp$-module $M$ has the decomposition with respect to \eqref{eq:Fp}:
\begin{equation}\label{eq:M=Mv}
  M=\bigoplus_{v|p} M_v
\end{equation}
and each $M_v$ is a (polarized) $O_{B_v}$-\dieu module. 
Suppose  $\rank_{W(k)}M=4dm$. 
In this case, $M$ satisfies the determinant condition  \eqref{eq:detM} if and only if for each  $v \mid p$ the following equality of polynomials of degree $2m[F_v:\Q_p]$ holds:  
\begin{align}\label{eq:detM_v}
  \char  (b \mid M_v/{\sf V}M_v)=\char(b)^m \pmod p \in k[T] \quad \text{for all} \quad  b\in O_{B_v}.  
\end{align}
Here, $\char(b)$ is the reduced characteristic polynomial of $b$ from $B_v$ to $\Q_p$,  defined in the same way as in Definition \ref{def:detM}.

If $v$ is unramified in $B$, we choose an identification $B_v = \Mat_2(F_v)$ such that $O_{B_v}=\Mat_2(O_v)$. 
Now suppose that 
 $v$ is ramified in $B$. Then $B_v$ is the division quaternion $F_v$-algebra and $O_{B_v}$ is the unique
maximal order of $B_v$. 
Let $\Pi_v$ be a uniformizer of $B_v$, and  $b \mapsto \bar b$ the canonical involution on $B_v$. 
Let $F_v'$ be the unramified quadratic extension of $F_v$ with ring of integers $O_v'$. The non-trival automorphism of $F_v'/F_v$ is also denoted by $a\mapsto \bar a$. 
We choose a presentation
\begin{equation}\label{eq:OBv}
    O_{B_v}=O_v'[\Pi_v]
\end{equation}
subject to the following relations
\begin{equation}\label{eq:Pi}
 \ol \Pi_v=-\Pi_v, \quad \Pi_v \ol \Pi_v=\pi_v, \quad \Pi_v a=\bar a \Pi_v\quad \forall\, a\in O_v',
 \end{equation}
for which the canonical involution of $B_v$ leaves $F'_v$ stable and induces the involution $\bar{\cdot}$ on $F_v'/F_v$. 
We also have  
\begin{equation}
    \label{eq:OBv2}
   O_{B_v} =\left \{\begin{pmatrix}
    a & -b \\
    \pi_v b & \bar a
\end{pmatrix} \mid  a, b\in O_v' \right \}, \ \Pi_v=\begin{pmatrix}
    0 & -1 \\
    \pi_v  & 0
\end{pmatrix}, \  \text{and} \ O_{B_v} \otimes_{O_v} O_v'=\begin{pmatrix}
    O_v'       &   O_v' \\
    \pi_v O_v' &   O_v'
\end{pmatrix}. 
\end{equation}
Let $F_v^{\rm ur}$ be the maximal unramified subextension of $F_v/\Qp$ and $O_v^{\rm ur} $ its ring of integers. Write the set of embeddings $\Hom_{\Zp}(O_v^{\rm ur}, W(k))=\{\sigma_i\}_{i\in \Z/f_v\Z} $ such that $\sigma \circ \sigma_i=\sigma_{i+1}$. Then we have decompositions
\begin{equation}\label{eq:Bv}
    O_v\otimes_{\Zp} W(k)=\prod_{i\in \Z/f_v\Z} \breve O_{v}^i, \quad 
    O_{B_v}\otimes_{\Zp} W(k)=\prod_{i\in \Z/f_v\Z} \breve O_{{B_v}}^i.
\end{equation}
If $v$ is unramified in $B$, then $\breve O_{B_v}^i= \Mat_2(\breve O_v^i)$. 
If $v$ is ramified in $B$, then using \eqref{eq:OBv2} we have 
\begin{equation}
    \label{eq:OBvi}
    \breve O_{B_v}^i=O_{B_v}\otimes_{O_v} O_v' \otimes_{O_v'} \breve O_{Fv}^i = \begin{pmatrix}
    \breve O_{v}^i &  \breve O_{v}^i \\
    \pi_v \breve O_{v}^i & \breve O_{v}^i 
\end{pmatrix}.
\end{equation}
With respect to the decomposition \eqref{eq:Bv}, we have 
\begin{equation}\label{eq:Mv}
   M_v=\bigoplus_{i\in \Z/f_v\Z} M_v^i, 
\end{equation}
where $M_v^i$ is the $\sigma_i$-component of $M_v$ and it is a $W(k)$-valued  (skew-Hermitian) $\breve O_{{B_v}}^i$-module.

\subsection{Proof of Theorem~\ref{main.1}}
%\begin{proof}
%[\bf Proof of Theorem~\ref{main.1}]
%The equivalence of (a) and (b) is shown by Lemma~\ref{lem:complex} and that of (b) and (c) is shown by Proposition~\ref{prop:sd}.
(a)$\implies$(b). Let $(A,\lambda,\iota)$ be a  principally polarized $O_B$-abelian variety of dimension $2dm$ over $\C$. 
Then the pair $(V,\psi)\coloneqq (H_1(A(\C),\Q), \<\, ,\,\>_\lambda)$ is a $\Q$-valued skew-Hermitian $(B, *)$-module, where $\<\, ,\,\>_\lambda$ is the alternating pairing induced by  $\lambda$.   
We have that $\dim_\Q V=2\dim A=4dm$. 
Further, the group $\Lambda\coloneqq H_1(A(\C),\Z)$ is an $O_B$-lattice in $(V, \psi)$.
Moreover, the natural identification $\Lie(A)=V_\R$ gives rise to complex structure $J$ on $V_\R$. 
Finally, if we let  $h_0$ be  the unique $\R$-algebra homomorphism $\C \to \End_{B \otimes _{\Q}\R}(V_\R)$  sending $a+bi$ to $a I_{V_\R}+bJ$,  
then it satisfies condition (iv) in  Definition~\ref{def:PELC} by the Riemannian condition (cf.~\cite[Th\'{e}or\`{e}me 4.7]{deligne:travaux}).

(b)$\implies$(a). 
Let $\Lambda$ be a self-dual $O_B$-lattice in $(V, \psi)$. 
There always exists a map $h_0$ as in Definition \ref{def:PELC} (iv), and we obtain a principal integral PEL datum $\mathscr D$. 
It gives rise to an abelian variety   $A(\C)=(V_\R, h_0(i))/\Lambda$ with the induced additional structures, where $(V_\R, h_0(i))$ is the complex vector space $V_\R$ with complex structure $h_0(i)$.

(b)$\iff$(c).
 The assertion can be reduced to the local one, which follows from Propositions~\ref{sd-split} and \ref{prop:herm-self-dual}. 
 We prove that a self-dual $O_B$-lattice $\Lambda$ is unique if it exists. 
  Suppose that $\Lambda'$ is another self-dual lattice in $(V, \psi)$. 
   By Propositions~\ref{sd-split} and \ref{prop:herm-self-dual}, 
   the completions $\Lambda_v$ and $\Lambda'_v$ at every finite place $v$ are isomorphic. 
 Hence two lattices $\Lambda$ and $\Lambda'$ lie in the same genus.  
  The isomorphism classes of lattices in the genus are classified by the double coset space ${\rm DS}(\bfG^1,U^1)=\bfG^1(\Q)\backslash \bfG^1(\A_f)/U^1$, where %$\bfG^1={\rm U}_B(V,\psi)$ and 
  $U^1$ is the stabilizer in $\bfG^1(\A_f)$ of the $O_B\otimes \wh \Z$-lattice $\Lambda \otimes \wh \Z$.  
  Since  $\bfG^1(\R)$ is non-compact, the strong approximation theorem implies that the space ${\rm DS}(\bfG^1,U^1)$ is a singleton.

(b)$\implies$(d). Let $\mathscr D$ be a principal integral PEL-datum of type C of rank $m$, and let $\bfG=\GU_B(V,\psi)$, and $X$ be the $\bfG(\R)$-conjugacy class of $h_0$. Choose a special pair $i: (T,h_T)\embed (\bfG,X)$ of the Shimura datum $(\bfG,X)$, where $T$ is a maximal torus of $\bfG$ defined over $\Q$ and $h_T:\bbS \to T_\R$ is a homomorphism of $\R$-groups such that $i(h_T)\in X$. Such a special pair always exists; see \cite[Section 5.1]{deligne:travaux}. 
Moreover, since $(G,X)$ is a PEL-type Shimura datum, $(T,h_T)$ is a CM pair \cite[Section A.3]{milne:1988invent} (also cf.~\cite[pp.~325--326]{milne:michigan}), namely, the cocharacter $\mu_T=h_{T,\C}(z,1)$ satisfies the Serre condition, or equivalently, the image in $\bfG(\Q)\backslash X\times \bfG(\A_f)/{\rm Stab}_{\bfG(\A_f)} (\Lambda\otimes \wh \Z)$, correspoinds to a complex principally polarized $O_B$-abelian variety $(A,\lambda, \iota)$ of dimension $2dm$ in which $A$ is a CM abelian variety. By CM theory, $(A,\lambda, \iota)$ is defined over $\Qbar$ and it has good reduction everywhere. Reduction modulo $p$ of $(A,\lambda,\iota)$ gives a desired abelian variety over $\Fpbar$, as a specialization of an $O_B$-abelian variety of characteristic zero satisfies the determinant condition.

%Choose a maximal CM subfield $E$ of $\End_{B}(V)$ such that the adjoint $\dagger$  with respect to $\psi$ induces the canonical involution on $E$ and choose an $\R$-algebra homomorphism $h_0: \C \to E_\R$ satisfies condition (iv) of Definition. Using the argument of the proof of (b)$\implies$(a), we obtain a complex principally polarized $O_B$-abelian variety $(A,\lambda, \iota)$ of dimension $2dm$ which is a CM abelian variety. By CM theory, $(A,\lambda, \iota)$ is defined over $\Qbar$ and it has good reduction everywhere. Reduction modulo $p$ of $(A,\lambda,\iota)$ gives a desired abelian variety over $\Fpbar$.

(d)$\implies$(c). Let $(A,\lambda,\iota)$ be a principally polarized $O_B$-abelian variety over $k$ satisfying the determinant condition. It suffices to show that 
if there is a place $v$ of $F$ ramified in $B$ such that $\ord_{\Pi_v}(\gamma)$ is even, then $m$ is even. Without changing the involution $*$, we may assume that $\ord_{\Pi_v}(\gamma)=0$.
Suppose that $v \mid \ell$ for some prime $\ell\neq p$. Then the $\ell$-adic Tate module 
$T_\ell(A)$ is a $\Z_\ell$-valued self-dual skew-Hermitian $O_{B}\otimes \Z_\ell$-lattice and its $v$-component 
%$T_v$ 
is a $\Z_\ell$-valued self-dual skew-Hermitian $O_{B_v}$-lattice. By Proposition~\ref{prop:herm-self-dual}, $m$ must be even. 
Suppose now that $v \mid p$. 
Then the \dieu module $(M,\<\, ,\>)$ of $(A,\lambda,\iota)$ is a principally polarized \dieu $O_B\otimes \Zp$-module of $W(k)$-rank $4dm$ and its $v$-component $M_v$ is a principally polarized $O_{B_v}$-\dieu module. 
Since $M$ satisfies the determinant condition, so as  $M_v$. 
By \cite[Proposition~5.6(2)(a)]{yu:D} $M_v$ is a free $O_{B_v}\otimes_{\Zp} W(k)$-module of $W(k)$-rank $4m[F_v:\Qp]$. Write $M_v=\oplus_{i} M_v^i$ as in \eqref{eq:Mv} and then each 
$(M^i_v,\<\, ,\>)$ is a $W(k)$-valued self-dual skew-Hermitian free $\breve O_{B_v}^i$-lattice. 

Let $L$ be the field of fractions of the ring $W(k)$.
The reduced trace $\Tr_{B_v/\Qp} :B_v\to \Qp$ induces a map $\Tr_{B_v/\Qp} \otimes L: B_v\otimes_{\Qp} L\to L$ by $L$-linearlity. 
We have a decomposition  $B_v\otimes_{\Qp} L\simeq \prod_{i\in \Z/f_v\Z} \breve B_v^i$ as in \eqref{eq:Bv} and $\Tr_{B_v/\Qp}\otimes L =\sum_{i\in \Z/f_v\Z} \Tr_{\breve B^i_v/L}$, where $\Tr_{\breve B^i_v/L}:\breve B^i_v \to L$ denotes the reduced trace. 

Consider the component $M_v^0$ at $i=0$; it is a free $\breve O_{B_v}^0$-module of rank $m$. 
Let $\mathfrak D_{F_v/\Qp}^{-1}$ be the inverse different of $F_v/\Q_p$ and $\delta_v \in F_v$ be an element with $\mathfrak D_{F_v/\Qp}^{-1}=\delta_v^{-1}O_{F_v}$. 
Then there is a unique skew-Hermitian form 
\[ \<\,,\>_{B_v}: M_v^0\times M_v^0 \to \Pi_v^{-1} \delta_v^{-1} \breve O^0_{B_v} \] \text{such that} 
\[ \<x,y\>=\Tr_{\breve B_v^0/L} \<x,y\>_{B_v}, \forall\, x,y \in M_v^0. \]
We put 
\[ \psi_{B_v}(x,y)\coloneqq \delta_v \<x,\gamma \Pi_v y\>_{B_v}: M_v^0\times M_v^0 \to \breve O^0_{B_v}.  \]
Since $M^0_v$ is self-dual with respect to $\<\, ,\>$ and $\gamma$ is a unit in $O_{B_v}$, $M_v^0$ is self-dual with respect to $\psi$. 
Using \eqref{eq:gamma} and \eqref{eq:Pi}, one  computes
\[ (\gamma \Pi_v)^*=\gamma \ol{\gamma \Pi_v} \gamma^{-1} = \gamma \ol{\Pi}_v \ol{\gamma} \gamma^{-1}=\gamma \Pi_v. \]
So $\psi_{B_v}$ is a perfect and skew-Hermitian form on $M_v^0$ with respect to an involution $'$.
For $b\in B_v$, we have 
\begin{align*}
 \psi_{B_v}(bx,y) & =\delta_v \<bx,\gamma \Pi_v y\>_{B_v}=\delta_v \<x,\gamma \bar b \Pi_v y\>_{B_v},
 \\
 \psi_{B_v}(x,b'y) & =\delta_v \<x,\gamma \Pi_v b' y\>_{B_v}.     
\end{align*} 
So we get 
\[ b'=\Pi_v^{-1} \bar b \Pi_v. \]
Let $J(\breve O^0_{B_v})$ be the Jacobson radical of $\breve O^0_{B_v}$. 
By \eqref{eq:OBvi} we have that 
\[ J(\breve O^0_{B_v})=\begin{pmatrix}
    \pi_v \breve O_{v}^0 &   \breve O_{v}^0 \\
    \pi_v \breve O_{v}^0 & \pi_v \breve O_{v}^0 
\end{pmatrix}, \quad \text{and} \quad \breve O^0_{B_v}/J(\breve O^0_{B_v})\simeq k\times k. \]
Moreover, one easily computes that
\begin{equation}\label{eq:prime}
   \begin{pmatrix}
    a_1 & 0 \\
    0 & a_2
\end{pmatrix}'=\begin{pmatrix}
    a_1 & 0 \\
    0 & a_2
\end{pmatrix}, \quad \forall\, \begin{pmatrix}
    a_1 & 0 \\
    0 & a_2
\end{pmatrix}\in \breve O^0_{B_v}. 
\end{equation}
So the involution $'$ induces a trivial involution on $k\times k$.
Put 
\[ V \coloneqq M^0_v/J(\breve O^0_{B_v}) M_v^0, \]
and let $\ol{\psi}_{B_v}$ be the pairing on $V$ induced by $\psi_{B_v}$. Then $V$ is a free $k\times k$-module of rank $m$ and from \eqref{eq:prime} $\ol{\psi}_{B_v}:V\times V\to k\times k$ is non-degenerate and alternating. Put $e_1=(1,0)$ and $e_2=(0,1)$, the standard idempotents of $k\times k$ and set $V_i\coloneqq e_iV$ for $i=1,2$. 
Then the restriction $\ol{\psi}_{B_v}:V_1\times V_1\to k$ is a non-degenerate symplectic $k$-vector space and therefore $m$ is even. This proves Theorem~\ref{main.1}. \qed
%\end{proof}

% (if degree $4m$ such that there is an embedding $E\embed B$
\subsection{Existence of  superspecial abelian varieties with additional structures}\label{ss:spmod}
\begin{defn}
Let $M$ be a \dieu module over $k$ satisfying 
 \begin{align}\label{dimcodim}
 \dim_k M/{\sf F}M=\dim_k M/{\sf V}M=g. 
 \end{align}
Such a module $M$ is called \emph{superspecial} if it further satisfies 
\begin{align}\label{a(M)}
a(M)\coloneqq \dim_k M/({\sf F}, {\sf V})M =g.
\end{align}
%An  abelian variety $A$ is called  \emph{superspecial} if its \dieu module is superspecial. 
\end{defn}
We remark that (\ref{dimcodim}) and (\ref{a(M)}) imply  ${\sf V}^2M=pM$. 
Conversely, suppose $M$ is a finite and free $W(k)$-module together with a $\sigma^{-1}$-linear operator $\sfV:M \to M$ satisfying ${\sf V}^2M=pM$.
%The  last condition implies  that
Then we have $pM\subset {\sf V}M$ and hence  the operator ${\sf F}\coloneqq p {\sf V}^{-1}: M [1/p] \to M[1/p]$ is stable on $M$. It follows that  the $W(k)[{\sf F}, {\sf V}]$-module $M$ is a superspecial \dieu module.

Recall that an abelian variety $A$ over $k$ 
of dimension $g$ is called \emph{superspecial} (resp.~\emph{supersingular}) if it is  isomorphic (resp. isogenous) 
over $k$ to a  product $E_1 \times \cdots \times E_g$ of supersingular elliptic curves $E_1, \ldots, E_g$ over $k$. 

By a theorem of Oort 
\cite[Theorem 2]{oort:product}, an abelian variety $A$ over $k$  is superspecial if and only if its \dieu module $M(A)$ is superspecial. 
\begin{thm}\label{exist}
Let $(B, *, O_B)$ be as in Section \ref{ss:PEL} and $m$ be a positive integer. 
If the conditions in Theorem \ref{main.1} hold, then 
 there exists a $2dm$-dimensional 
principally polarized superspecial $O_B$-abelian variety $(A, \lambda, \iota)$ over $k$
which satisfies the determinant condition. 
\end{thm} 
%This theorem implies that  there is a triple $(A, \lambda, \iota)$ such that $A$ is superspecial and $(A, \iota)$ 
 % satisfies the determinant condition. 
\begin{proof}
We show the problem reduces to 
 Proposition  \ref{existloc} below. Suppose that  there exists a  principally polarized superspecial
$O_B\otimes \Z_p$-Dieudonn\'e module $M$ of $W(k)$-rank $4dm$ which satisfies
the determinant condition.
We write $\mathscr G_p$ for the $p$-divisible group with 
additional structure corresponding to $M$. 
%Then $\mathscr G$ can be decomposed into local factors as $\mathscr G=\bigoplus_{v\mid p}\mathscr G_{v}$ according to the decomposition $O_B\otimes Z_p=\bigoplus _{v\mid p}O_{B,v}$. 
It suffices to show that there is a principally polarized superspecial
$O_B$-abelian variety $(A, \lambda, \iota)$ over $k$ realizing $\mathscr G_p$.  
 
By \cite[Theorem 2.1]{slope}, 
there exists a \emph{supersingular} polarized (not necessarily
principally) $O_B$-abelian variety $(A',\lambda',\iota')$ 
of dimension $2dm$. 
 By the proof of \cite[Proposition 3.1]{slope}, there is  a
quasi-isogeny $\phi_p : A'[p^{\infty}] \to \mathscr G_p$ of polarized
$p$-divisible groups with $O_B \otimes \Z_p$-actions.
Let $S$ denote the set of primes $\ell$ such that $\ell \neq p$ and $\ell \mid \deg \lambda'$.
For each $\ell \in S$  we have 
$H^1(\Q_\ell, \bfG^1)=0$, and hence 
there is an $O_B \otimes \Z_\ell$-linear
isomorphism $T_\ell(A') \otimes _{\Z_{\ell}} \Q_{\ell} \simeq 
V\otimes_{\Q} \Q_\ell$ preserving the pairings $\< \, , \, \>_{\lambda'}$ and $\psi_\ell$.
Now let $\mathscr G_\ell$ be a  principally polarized
$\ell$-divisible group   with $O_B\otimes \Z_\ell$-action such that
the associated $\ell$-adic Tate module $T_\ell(\mathscr G_\ell)$ with additional structures is
isomorphic to the lattice  $(\Lambda_{\Z_\ell},\psi_\ell)$. 
Then there is a quasi-isogeny $\phi_\ell : A'[\ell^{\infty}] \to \mathscr G_\ell$ of polarized
$\ell$-divisible groups with $O_B \otimes \Z_\ell$-actions.
% the Newton slope of  the $p$-divisible group $A'[p^{\infty}]$ is the
% same as that of $\mathscr G$. 
Further we choose a product $N$ of powers of primes in $S \cup \{p\}$ 
 such that $N\phi_\ell$ is an isogeny  for all
$\ell\in S\cup\{p\}$. Replacing $\phi_\ell$ by $N\phi_\ell$ and
$\lambda'$ by $N^2 \lambda'$, we may assume that the $\phi_\ell$'s are
isogenies. 

%By replacing $\phi$ by $p^r\cdot\phi$ for a sufficiently large $r$, 
%we may assume that $\phi$ is an isogeny. 
Let $A\coloneqq A'/(\prod_{\ell \in \{p\}\cup S} \ker \phi_\ell)$. Further, let  $\lambda$ and $\iota$ be the 
 polarization and $O_B$-multiplication of $A$ induced by $\lambda'$ and $\iota'$, respectively.  
Then  $\lambda$ is principal since it induces the given principal polarization of 
$\mathscr G_\ell$ for each $\ell\in S \cup\{p\}$. 
  Moreover,  we have an isomorphism of $A[p^{\infty}] \simeq 
\mathscr G_p$ of $p$-divisible groups compatible with additional structures. 
 \end{proof}

%We apply \eqref{eq:dec_M} to the self-dual skew-Hermitian 
%$O_B$-lattice $(\Lambda,\psi)$,  and get the decomposition
%$\Lambda_p=\bigoplus_{v \mid p} \Lambda_v$. Observe that the central idempotents
%in the decomposition \eqref{eq:dec} are fixed by the involution
%$*$. 
%This implies that $\psi_p(\Lambda_v,\Lambda_{v'})=0$ if $v\neq
%v'$. Let $\psi_v$ be the restriction of $\psi_p$ to $\Lambda_v$. Then
%$(\Lambda_v,\psi_v)$ is a self-dual $\Zp$-valued skew-Hermitian
%$O_{B_v}$-lattice of $O_{B_v}$-rank $m$. %Further, we have a 
%decomposition 
%\begin{equation}\label{eq:dec_L} 
%$(\Lambda_p,\psi_p)=\bigoplus_{v|p} (\Lambda_v,\psi_v)$. 
%of self-dual $\Zp$-valued skew-Hermitian lattices 
%One easily shows 
%\[ \ol{\Pi}_v=-\Pi_v \quad \text{and}\quad \ol{a+b\Pi_v}=\bar{a}-b\Pi_v \quad \text{for all} \quad a,b\in F_v'.\]
We retain the notation 
from Section \ref{ss:dieu}. 
Let $v$ be a place of $F$ ramified in $B$. 
Let $(F_v')^{\rm ur}$ be the maximal unramified subfield
extension of $\Q_p$ in $F_v'$, and let $(O_{v}')^{\rm ur}=O_{(F_v')^{\rm ur}}$ be  the ring of integers. 
Further let $\Hom_{\Zp}((O_v')^{\rm ur},W(k))$ denote the set of embeddings of $(O_v')^{\rm ur}$ into $W(k)$
over $\Zp$. 
Since the inertial degree of $(F_v')^{\rm ur}/\Q_p$ is 
$2f_v$, we may write $\Hom_{\Z_p}((O_v')^{\rm ur},W(k))=\{\tau_j\}_{j\in
  \Z/2f_v\Z}$ such that $\sigma \circ \tau_j=\tau_{j+1}$. 
  For an $O_{B_v}$-\dieu module $M_v$, we have 
 a decomposition
\begin{equation}
  \label{eq:dec_Mj}
  M_v=\bigoplus_{j\in \Z/2f_v\Z} M_v^j, 
\end{equation}
where $M_v^j$ is  the $\tau_j$-component of $M_v$. 
By \cite[Lemma 5.2 (2)]{yu:D}, the module $M_v$ satisfies the determinant  condition \eqref{eq:detM_v} if and only if the $k$-vector space 
$(M_v/\sfV M_v)^j$ has the same dimension for all $j\in \Z/2f_v\Z$. 
% where $B$ is a totally indefinite quaternion algebra. 
\begin{prop}\label{existloc}
Let $(B, *, O_B)$ be as in Section \ref{ss:PEL} and $m$ be a positive integer. 
If the conditions in Theorem \ref{main.1} hold, then  there exists a principally polarized superspecial $O_B\otimes
\Z_p$-Dieudonn\'e module 
$M$ of $W(k)$-rank $4dm$ which satisfies the determinant
 condition. 
\end{prop}
\begin{proof}
%Since $*$ is a positive involution, by 
%\cite[Section 21]{mumford:av},
%there is an element $\gamma \in
%B^\times$ such that $b^{*}=\gamma \bar {b} \gamma^{-1}$ for all $b\in
%B$, $\gamma+\bar{\gamma}=0$, and $\gamma^2$ is a totally negative
%element in $F^\times$. 
% We have decompositions
% $F\otimes \Z_p=\bigoplus _{v \mid p}F_v$ and $O_F\otimes
% \Z_p=\bigoplus_{v \mid p}\mathcal O_v$.  
% Similarly, we have $B_p=\bigoplus_{v \mid p}B_v$ and $O_B \otimes \Z_p=\bigoplus _{v \mid p}\mathcal O_{B_v}$. 
% Put $\gamma=(\gamma_v)_{v \mid p}$. 
% Then the involution $*$ on $B$ induces an involution $(\cdot)^{*}=\gamma_v\bar{(\cdot)}\gamma^{-1}_v$ on $B_v$.
It suffices to show that for each   $v \mid p$ there exists a principally polarized 
$O_{B_v}$-Dieudonn\'e module $M_v$ of $W(k)$-rank
$4m[F_v :\Q_p]$ satisfying   the determinant condition \eqref{eq:detM_v}.  
% where $W=W(k)$ is the ring of Witt vectors over $k$.  
In fact, the direct sum $M \coloneqq \bigoplus_{v\mid p}M_v$ of such modules $M_v$ with additional structures satisfies the desired
properties.  

Let $\gamma \in B^{\times}$ be as in \eqref{eq:gamma}. 
By the decomposition \eqref{eq:dec}, 
one has that   $\gamma=(\gamma_v)_{v \mid p}$ with
$\gamma_v\in B_v^\times$.
The involution $*$ on $B$ induces an involution
$b \mapsto b^*=\gamma_v\bar{b}\gamma^{-1}_v$ on $B_v$.

First we assume that $B_v$ is the matrix algebra. 
 We can take  an isomorphism $B_v \simeq \Mat_2(F_v)$ which identifies  $O_{B_v}$ with $\Mat_2(O_v)$. 
   Since $O_B \otimes \Z_p$ is stable under $*$, the ring $O_{B_v}$ is normalized by $\gamma_v$.  
   This implies that  $\gamma_v$ belongs to $F_v^{\times}\cdot\GL_2(O_v)$, and hence we have  $\gamma_v=\pi_v^{a_v} u_v$ for some elements $u_v\in
\GL_2(O_v)$ and $a_v\in \Z$.  Further we have $b^*=u_v \bar b u_v^{-1}$ for any  $b\in B_v$. 

% The involution  $*$ on $B$ is unchange if one replace $\gamma$ by $u\gamma$ for an element $u \in F^{\times}$. 
% Hence, by the weak approximation, we may assume that 
% $\gamma$ belongs to $\GL_2(O_v)$ for all $v \mid p$ which splits in $B$. 

The construction of $M_v$ in this case reduces to 
the ``Hilbert-Siegel'' case. 
As in \cite[Lemma 4.5]{2003Fourier}, 
there exists a superspecial principally polarized $O_v$-Dieudonn\'e
module $N_1$ of $W(k)$-rank $2[F_v:\Q_p]$, equipped with a non-degenerate alternating pairing $\phi_1 : N_1 \times N_1 \to W(k)$ 
 such that $\phi_1(bx,y)=\phi_1(x,by)$ for any $b\in O_v$.  
The existence of such a pairing and  \cite[Proposition 2.8 (2) $\Rightarrow$ (4)]{2003Fourier} imply that  $N_1$ satisfies the condition corresponding to  \eqref{eq:detM_v}. 
For general $m \geq 1$, we put  
$(N,\phi)\coloneqq(N_1^{\oplus m}, \phi_1^{\oplus m})$. 
Then $(N,\phi)$ 
%this module has a structure of
is of
$W(k)$-rank $2m[F_v:\Q_p]$. 
 
Now we set $M_v\coloneqq N\oplus N=(O_v\oplus O_v)\otimes_{O_v} N$.
We regard the module $M_v$ as a left $O_{B_v}=\Mat_2(O_v)$-module, 
with the multiplications in the same way on column vectors. 
We construct a polarization on  $M_v$. 
We first put 
 \[\psi : M_v \times M_v \to W(k),  \quad  \psi((x_1, x_2),(y_1,y_2)) \coloneqq 
 \phi(x_1,y_1)+\phi(x_2,y_2). \] 
Then a direct computation shows that 
$\psi (bx,y)=\psi(x,b^ty)$ for $x,y \in M_v$ and $b \in \Mat_2(O_v)$. 
 We define a  polarization $\langle\,, \rangle$ on $M_v$ by 
 \[\langle\,, \rangle : M_v \times M_v \to W(k), \quad 
 \langle x, y\rangle 
 \coloneqq\psi (x, C^{-1}u_v^{-1}y), \] where 
$C = \left(
\begin{smallmatrix}
 0  &  1 \\ -1 & 0
 \end{smallmatrix}
 \right)$. 
 Since $C u_v\in \GL_2(O_v)$,   we have  $C^{-1} u_v^{-1}
 M_v=M_v$. 
 It follows that  the dual lattice of $M_v$ with respect to
 $\<\,,\>$ is equal to the one  with respect to $\psi$, which is
 $M_v$. 
 Hence the polarization $\<\,,\>$ is  principal. 
 Further we have that  
 \begin{align*}
 b^*& =u_v \bar{b}u_v ^{-1}=u_v C b^t C^{-1}u_v^{-1} \quad \text{and} 
 \\ 
 \langle bx, y\rangle &  = 
 \psi (bx, C^{-1}u_v^{-1}y)=\psi (x, b^t C^{-1}u_v^{-1}y)=\psi (x,
 C^{-1}u_v^{-1}b^*y)=\langle x, b^*y \rangle.\end{align*}
Thus the \dieu module $M_v$ with the $O_{B_v}$-action and  polarization $\< \, , \, \>$ satisfies the desired properties. 

Next we assume that $B_v$ is a division algebra. 
%We will construct a principally polarized $\mathcal O_{B_v}$-algebra
%of rank $4d$. 
For brevity, we write $\gamma$, $e$,
$f$, $\pi$, $\Pi$ for $\gamma_v$, $e_v$, $f_v$, $\pi_v$ and $\Pi_v$,
respectively. Then $\gamma=\pi^a u$ for some integer $a\in \Z$ and
$u\in O_{B_v}$ with $\ord_{\Pi}(u)=0$ or $1$. 
By Theorem \ref{main.1} 
%Since $(\Lambda_v,\psi_v)$
%is a self-dual $\Zp$-valued skew-Hermitian $O_{B_v}$-lattice of
%$O_{B_v}$-rank $m$, 
%by Proposition %\ref{prop:herm-self-dual}
we have that either $\ord_\Pi(u)=1$, or both
$\ord_{\Pi}(u)=0$ and $m=2n$ is even.
Let $\Tr_{B_v/\Qp}$
be the reduced trace from $B_v$ to $\Qp$, and 
 $\delta^{-1} \in F_v$ be a generator of the inverse different $\mathfrak D_{O_v/\Zp}^{-1}$ of $O_v$ over $\Zp$.

{\bf Case (a)} $\ord_\Pi(u)=1$. Let $(L_1,\varphi_B)=(O_{B_1}
e_1, (1))$ be the $O_{B_v}$-valued Hermitian $O_{B_v}$-lattice of rank
one with $\varphi_B(e_1,e_1)=1$. 
Put 
\[\psi_1(x,y)\coloneqq\Tr_{B_v/\Qp} 
\big(\varphi_B(x,
\delta^{-1} u^{-1} y)\big)\]
 for $x,y\in L_1$. 
One checks that  $(L_1,\psi_1)$ is a
self-dual $\Zp$-valued skew-Hermitian $O_{B_v}$-lattice of
$O_{B_v}$-rank one. We choose an element $\alpha\in O_{B_v}$ such that
$\alpha \bar \alpha=p$. Define an $O_{B_v}$-linear map 
$\sfV:L_1\to L_1$ by $\sfV e_1=\alpha e_1$. 
This map $\sfV$ defines an $O_{B_v}$-\dieu module 
$M_1\coloneqq L_1\otimes_{\Zp} W(k)$ by the usual $\sigma^{-1}$-linear
extension on $W(k)$, i.e., $\sfV (x\otimes a)=\sfV(x)\otimes a^{\sigma^{-1}}$ for
$x\in L_1$ and $a\in W(k)$. 
Then we have  
\[\varphi_B(\sfV e_1,\sfV e_1)=\varphi_B(\alpha e_1,\alpha
e_1)=p\varphi_B(e_1,e_1)\]
 and hence $\varphi_B(\sfV x,\sfV
y)=p\varphi_B(x,y)$ for $x, y\in L_1$. 
It follows that 
\begin{equation}
  \label{eq:pol.1}
\begin{split}
\psi_1(\sfV x,\sfV y)& =\Tr_{B_v/\Qp} \varphi_B(\sfV x, \delta^{-1} u^{-1} \sfV y)
\\ &=\Tr_{B_v/\Qp} \varphi_B(\sfV x, \sfV (\delta^{-1} u^{-1}  y)) \\
&= \Tr_{B_v/\Qp} p \varphi_B(x, \delta^{-1} u^{-1}  y)=p\psi_1(x,y).   
\end{split}   
\end{equation}
Let $\<\, , \>_1: M_1\times M_1\to W(k)$ be the alternating pairing
extending from $\psi_1$ by $W(k)$-linearity. Then \eqref{eq:pol.1} implies that  $\<\sfV x, \sfV y\>=p\<x,y\>^{\sigma^{-1}}$ for $x,y\in M_1$. 
Since $\ord_\pi(\alpha^2)=\ord_\pi(\alpha \bar \alpha)=e$, we have
$\sfV^2 L_1=p L_1$ and $\sfV^2 M_1=p M_1$. Thus,
$(M_1,\<\,, \>_1)$ is a principally polarized superspecial
$O_{B_v}$-\dieu module of $W(k)$-rank $4[F_v:\Qp]$. 
Finally, we have  $(M_1/\sfV M_1)^j=(L_1/\sfV
L_1)\otimes_{(O_v')^{\rm ur},\tau_j} k$ for all $j \in \Z/f\Z$, and in particular  they have the same dimension. 
Hence $M_1$ satisfies  condition \eqref{eq:detM_v}. 
Put $(M_v,\<\,, \>_v)\coloneqq(M_1,\<\, ,\>_1)^{\oplus m}$. Then
$(M_v,\<\,,\>_v)$ satisfies the desired properties. 

{\bf Case (b)} $\ord_\Pi(u)=0$ and $m=2n$ is even. 
Let $L_1\coloneqq H(-1)=O_{B_v} e_1+O_{B_v} e_2$ and $\varphi_B$ be the
Hermitian form defined by 
\[ \varphi_B(e_1,e_1)=\varphi_B(e_2,e_2)=0, \quad
\varphi(e_1,e_2)=\Pi^{-1}. \]
Put 
\[\psi_1(x,y)\coloneqq \Tr_{B_v/\Qp} \big( \varphi_B(x, \delta^{-1} u^{-1}y) \big) \]
 for
$x,y\in L_1$. The same computation shows that
$(L_1,\psi_1)$ is a self-dual $\Zp$-valued skew-Hermitian
$O_{B_v}$-lattice of rank two. We choose two elements $\alpha, \beta \in O_{B_v}$ such that
$\alpha \Pi^{-1} \bar \beta =p \Pi^{-1}$ and $\ord_{\Pi}
\alpha=\ord_{\Pi}(\beta)=e$. For example if $e=2c$ is even, put
$\alpha=\pi^c$ and $\beta=p \pi^{-c}$, and if $e=2c+1$, put
$\alpha=\pi^c \Pi$ and $\beta=p \pi^{-c-1} \Pi$. 
 Define an $O_{B_v}$-linear map 
$\sfV:L_1\to L_1$ by $\sfV e_1=\alpha e_1$ and $\sfV
e_2=\beta e_2$. 
This gives an $O_{B_v}$-\dieu module 
$M_1\coloneqq L_1\otimes_{\Zp} W(k)$. One also easily checks 
$\varphi_B(\sfV x, \sfV y)=p\varphi_B(x,y)$ and $\psi_1(\sfV x,\sfV
y)=p\psi_1(x,y)$ for $x,y\in L_1$ from \eqref{eq:pol.1}. 
Define the polarization 
$\<\,,\>_1:M_1\times M_1 \to W$ in the same way and we have $\<\sfV x,
\sfV y\>_1=p\<x,y\>_1^{\sigma^{-1}}$ for all $x,y\in M_1$. 
%We have  $\varphi_B(\sfV e_1,\sfV e_1)=\varphi_B(\alpha e_1,\alpha
%e_1)=p\varphi_B(e_1,e_1)$ and $\varphi_B(\sfV x,\sfV
%y)=p\varphi_B(x,y)$ for $x, y\in L_1$. 
Furthermore, we see
$\sfV^2 L_1=p L_1$ and $\sfV^2 M_1=p M_1$, so $M_1$ is
superspecial. Finally, since $(M_1/\sfV M_1)^j=(L_1/\sfV
L_1)\otimes_{(O_v')^{\rm ur},\tau_j} k$ for all $j\in \Z/2f \Z$, 
they have the same dimension. Thus,
$(M_1,\<\,, \>_1)$ is a principally polarized superspecial
$O_{B_v}$-\dieu module of $W(k)$-rank $8[F_v:\Qp]$ which satisfies  condition \eqref{eq:detM_v}. 
% Put $(M_v,\<\,, \>_v)\coloneqq(M_1,\<\, ,\>_1)^{\oplus  m}$. 
Then the polarized \dieu module 
$(M_v,\<\, , \>)\coloneqq(M_1,\<\, ,\>_1)^{\oplus n}$ satisfies the desired 
properties. 
\end{proof}
%\begin{remark}
% There is a similar result for a  totally \emph{definite} quaternion algebra $B$, i.e.~of type D  (\cite[Thm.~7.2]{yu:D}). 
%However, for this case, it is not clear whether the  assertion corresponding  to Theorem \ref{exist} is true.  
%end{remark}

\subsection{Shimura varieties and moduli 
 spaces}\label{ss:Sh}
%We recall quaternionic Shimura varieties.   following Shimura 
%\cite{} and Deligne \cite{}. 
%Let $\scrD$ be a tuple $(B,*,O_B, V,\psi, \Lambda, h_0)$, where
%%\begin{itemize}
%\item 
%  $B$ is a totally indefinite quaternion algebra over a totally real field $F$ of degree $d$ over $\mathbb Q$;
%  \item $*$ is  a positive involution on $B$;
%  \item  
% $O_B$ is a maximal order stable under the involution $*$; 
% \item $V$ is a finite faithful left $B$-module;
%\item $\psi:V\times V\to \Q$ is a non-degenerate alternating form such that
%  $\psi(bx,y)=\psi(x,b^*y)$ for all $x,y\in V$ and $b\in B$; 
 % \item $\Lambda$ is a full $O_B$-lattice in $V$; 
%\item $h_0:\C \to \End_{B\otimes \R}(V_\R)$ is an $\R$-linear algebra
 % homomorphism, % where $V_\R\coloneqqV\otimes_\Q \R$,  
%  such that 
%\[ \phi_F(h_0(i)x, h_0(i)y)=\psi(x,y)  \quad {\text{for all}} \quad x,y \in
% V_\R=V\otimes\R, \]
%for any $x, y\in V_\R=V\otimes\R$, one has 
%  $\phi_F(h_0(i)x, h_0(i)y)=\psi(x,y)$ and 
%and the symmetric form  $(x,y)=\psi(x,h_0(i)y)$ is
%  definite (positive or negative) on $V_\R$.   
% \end{itemize}
% In addition, for any prime number  $\ell$ we assume that 
 % \begin{itemize}
   %  \item[(a)]  $O_{B}\otimes_{\Z} \Z_{\ell}$
  %   is a  
   %    maximal order in $B\otimes_\Q \Q_{\ell}$;
   %  \item[(b)] $\Lambda\otimes_\Z \Z_{\ell}$ is self-dual with
    %   respect to the pairing $\psi$.   
   % \end{itemize}

   Let $\mathscr D=(B,*,O_B,V,\psi, \Lambda, h_0)$ be an integral PEL datum of type C, and $\bfG$ be the associated $\Q$-group defined as in \eqref{eq:defG}. 
We define a homomorphism $h : \Res_{\C/\R}\mathbb G_{m, \C} \to \bfG_{\R}$ by restricting $h_0$ to $\C^{\times}$. 
Composing $h_\C$ with the map $\C^{\times} \to \C^{\times} \times \C^{\times}$ where $z \mapsto (z,1)$ then gives $\mu_{h} : \C^{\times} \to \bfG(\C)$. 
Moreover, there is an isomorphism $\End_{B \otimes _\Q \C}(V_{\C}) \simeq \Mat_{2m}(\C)^{d}$,  inducing an embedding of $\bfG(\C)$ into $\GL_{2m}(\C)^d$.  
Up to 
conjugation in $\bfG(\C)$, 
the cocharacter $\mu_h$  is expressed  as  
\begin{equation}
\label{eq:mu_h}
 \mu_{h}(z)=((\diag(z^{m}, 1^{m}), \ldots, (\diag(z^{m}, 1^{m})) 
 \in \bfG(\C) \subset \GL_{2m}(\C)^{d}.
\end{equation}
Let $X$ be the $\bfG(\R)$-conjugacy class of $h$. 
Then the pair $(\mathbf G,X)$ is a Shimura
datum \cite[(2.1.1)]{Deligne}. 
The reflex field of $(\bfG, X)$ is $\Q$ \cite[Section 7]{Shimura1966}. 
%\footnote{By this we mean that the pair $(\mathbf G,X)$ satisfies the
 % axioms for defining Shimura varieties 
 % \cite[(2.1.1.1)-(2.1.1.3)]{Deligne}, even though the group $\mathbf G$ may
 % not be connected.}.  

For any compact open 
subgroup $\K \subset \mathbf G(\A_f)$,  the  Shimura variety   associated to $(\mathbf G,X)$ of level 
$\K$ is defined by   
\begin{align*}
\Sh_{\K}(\bfG, X)_{\C} \coloneqq  
\bfG(\Q)\backslash X \times \bfG(\A_f)/{\K}.  
\end{align*}
This is a quasi-projective normal complex algebraic variety. 
Further, it admits the canonical model $\Sh_{\K}(\bfG, X)$  defined over the reflex field $\Q$.

For the remainder of this paper, we assume that the conditions in Theorem \ref{main.1} hold and fix a \emph{principal} integral PEL-datum 
$\mathscr D$ of type C. 
Further we fix a prime $p$.  
Let $\Qbar \subset \C$ denote the algebraic closure of $\Q$ in $\C$, and fix an embedding $\Qbar \embed \Qbar_p$ into an algebraic closure $\Qbar_p$ of $\Qp$.

 The lattice $\Lambda$  gives a model  over $\Z$ of the $\Q$-group $\bfG$, denoted again by $\bfG$. 
 Now we fix  an integer $N\geq 3$ with $p \nmid N$. 
 We   
  define a compact open subgroup  ${\sf K}^p(N)$ of $\bfG(\mathbb A_f^p)$ by
 \[{\sf K}^p(N)=\ker \Big(\bfG(\widehat{\Z}^p) \to 
 \bfG(\widehat{\Z}^p/N\widehat{\Z}^p)=\bfG(\Z/N\Z)\Big).\]
 We set ${\sf K}_p=\bfG(\Z_p)$ and ${\sf K}={\sf K}_p \cdot {\sf K}^p(N)\subset \bfG(\mathbb A_f)$. 
 Let $\mathbf M_{\sf K}=\mathbf M_{\sf K}(\mathscr D)$ be  the contravariant functor from the category of locally Noetherian schemes over $\Z_{(p)}$ to the category of sets which takes a connected scheme $S$ over $\Z_{(p)}$  to the set of isomorphism
classes of tuples $(A,  \lambda, \iota, \bar{\eta})$ where 
\begin{itemize}
    \item $(A, \lambda, \iota)$ is a principally polarized $O_B$-abelian scheme over $S$ as in Definition \ref{def:ab} which satisfies the determinant condition. 
    \item 
    $\bar{\eta}$   is a $\pi_1(S, \bar{s})$-invariant ${\sf K}^p(N)$-orbit of $O_B \otimes \wh{\Z}^p$-linear isomorphisms $\eta : \Lambda \otimes \wh{\Z}^p \xrightarrow{\sim} 
    T^p(A_{\bar s})$ 
which preserve the  pairings
\begin{align*}
\psi : \Lambda \otimes \wh{\Z}^p 
\times 
\Lambda \otimes \wh{\Z}^p 
\to \wh{\Z}^p 
\quad \text{and} \quad 
\<\, ,\,\>_\lambda :  \wh{T}^p(A_{\bar s}) \times 
 \wh{T}^p(A_{\bar s}) 
  \to \wh{\Z}^p(1)
 \end{align*}
up to a scalar in $(\wh{\Z}^p)^{\times}$. Here, 
  $\bar{s}$ is a geometric
point of $S$, $A_{\bar{s}}$ is the fiber of $A$  over $\bar{s}$,  $\wh{T}
^p(A_{\bar s})$  is its  prime-to-$p$ Tate module, and $\<\, ,\,\>_\lambda$ is the alternating pairing induced by $\lambda$. 
\end{itemize}
Two tuples $(A,  \lambda, \iota, \bar{\eta})$ and $(A',  \lambda', \iota', \bar{\eta}')$ are said to be isomorphic if there exists an $O_B$-linear
isomorphism of abelian schemes $f : A
\xrightarrow{\sim} A'$
such that $\lambda=f^t \circ \lambda' \circ f$ and $\bar{\eta}'=\ol{f \circ \eta}$. See    \cite[1.4.1]{Lan} for more details. 
 
By \cite{Kottwitz} and \cite[Ch.2]{Lan}, 
 the  functor $\mathbf M_{\sf K}$ is represented by a quasi-projective scheme (denoted again by) $\mathbf M_{\sf K}$ over $\Z_{(p)}$. 
 We remark that $\bfM_{\sf K}$ is isomorphic to the moduli problem of  prime-to-$p$  isogeny classes
of abelian schemes with a $\Z^{\times}_{(p)}$-polarization which was studied in \cite{Kottwitz}, under the assumption that $\Lambda$  is self-dual   (\cite[Prop.~1.4.3.4]{Lan}).        
 
 %We write $\mathbf M_{{\sf K}}^{\rm sp}$ for the subset of $\mathbf M_{{\sf K}}(k)$ consisting of  points whose underlying abelian varieties are superspecial. 
%Then $\mathbf M_{{\sf K}}^{{\rm sp}}$ is non-empty.  
%We can reduce this problem to   constructing a Dieudonn\'{e} module with certain additional structures (cf. \ref{existloc}).  We omit details.\footnote{The reader is referred to an earlier version \cite{TY} for the proof.}

%If $B\simeq \Mat_2(F)$, then we will assume $B=\Mat_2(F)$,
%$O_B=\Mat_2(O_F)$ and $*$ is the transpose.
%In this case, by the
%Morita equivalence, the moduli space defined in Definition~\ref{ModSp}
%is a Hilbert-Siegel modular variety.  
% Assume that $p$ is unramified in $B$. 

%Here we regard $\bfG=\GU_{O_B}(\Lambda, \psi)$ as the automorphism group scheme over $\Z$ defined by the lattice $\Lambda$.  

%The moduli scheme 
% $\mathbf M_{\sf K}$ is  a quasi-projective scheme defined 
%over $\Spec \mathbb Z_{(p)}$. 
%By Lemma~\ref{MKint} we may and will identify $\bfM_{\sf K}$ with $\bfM_{\sf K}^{\rm int}$. 
% which is of relative dimension $d\frac{m(m+1)}{2}$.
%In the present case,
% $\bfM_{\sf K}$ is isomorphiparameterizes isomorphism classes of %$2dm$-dimensional
% principally poliarized $O_B$-abelain varieties $(A,\lambda,\iota,\eta)$
% with level-$N$ structure.
 When $B=\Mat_2(F)$, $O_B=\Mat_2(O_F)$, and $*$ is the transpose, Morita equivalence implies   that 
  $\bfM_{\sf K}$ is isomorphic to the Hilbert-Siegel moduli space, which classifies  $dm$-dimensional
 principally polarized $O_F$-abelian varieties 
 with level-$N$ structure. 

Since the group $\bfG$ satisfies the Hasse Principle, 
  the generic fiber  $\bfM_{\sf K}\otimes_{\Z_{(p)}} \Q$ is  isomorphic to the canonical model  
  $\Sh_{\sf K}(\bfG, X)$ (rather than a finite union of them). 

% By \cite{Yu?}, further we have  \[\ker^1(\Q, G)=0.\] 

%The similitude character $c$ induces a surjective map
%\[ \bfG(\Q) \backslash \bfG(\A_f) / \bfG(\wh{\Z}) \to \Q^{\times} \backslash \A_f^{\times} / \wh{\Z}^{\times}. \]

We write $\calM_{\sf K} \coloneqq  \bfM_{\sf K} \otimes _{\Z_{(p)}}k$ for the geometric special fiber of $\bf M_{\sf K}$. 
 Further we write 
 \[ \calM_{\sf K}^{\rm sp} \subset \calM_{\sf K}^{\rm ss} \subset \calM_{\sf K} \]
 for the 
 superspecial and supersingular locus: the largest reduced closed subschemes such that 
 \begin{align*}
 \calM_{\sf K}^{\rm sp}(k) & = \{(A, \lambda, \iota, \bar{\eta})  \in \calM_{\sf K}(k) \mid   \text{$A$ is superspecial}\}, 
 \\
  \calM_{\sf K}^{\rm ss}(k) & = \{(A, \lambda, \iota, \bar{\eta})  \in \calM_{\sf K}(k) \mid   \text{$A$ is supersingular}\}.
 \end{align*}
 By Theorem \ref{exist}, there exists a 
principally polarized superspecial $O_B$-abelian variety $(A, \lambda, \iota)$ over $k$
which satisfies the determinant condition. 
Such a triple $(A, \lambda, \iota)$  together with a level structure $\bar\eta$ gives a $k$-point of the superspecial locus $\calM^{\rm sp}_{\sf K}$. 
Thus we see Theorem \ref{main.2}.

 \section{Irreducible components of the supersingular locus}\label{ss:irr}

%and can be regarded as the canonical integral model of a Shimura variety of Hodge type with a hyperspecial level ${\sf K}_p$ at $p$. 
%We choose a base point $x\in \mathbf M_{\sf K}^{{\rm sp}}$, 
%and consider  
  % its associated inner form $I$ of $G$  as in $\S$\ref{sunif}.   

 %Clearly, result \eqref{P.2} for
%$\pi_0(\Sh_{U})$ remains true.  

%Let $\ker^1(\Q, \mathbf G)$ denote the kernel of
%the local-global map $H^1(\Q, \mathbf G)\to \prod_{p\le \infty} H^1(\Qp, \mathbf G)$

% Let $\scrD=(B, *, V, O_B, \Lambda)$ be an integral PEL-datum. 
% Let $\mathcal A\to \mathbf{M}_{{U}^p(N)}$ be the corresponding universal abelian schemes over $\mathbf {M}_{{U}^p(N)}$. 
%By \cite[Lemma 1.2.5.9]{Lan}, 
%there exists an $O_B \otimes \mathbb Z_p$-module $\Lambda_0$ such that $\Lambda_0 \otimes \mathbb C \cong V_0$ as $O_B \otimes_{\mathbb Z}\mathbb C$-modules. 

 Let $\mathscr D=(B, *, V, \psi, O_B, \Lambda, h_0)$ be a principal integral PEL-datum of type C.  
 In this section, we fix a prime $p>2$ which is {\it unramified in  $B$}, i.e.  for each $v \mid p$,  the extension $F_v/\Q_p$ is unramified and the $F_v$-algebra $B_v$ is isomorphic to $\Mat_2(F_v)$.

 For each prime $\ell$, we have decompositions 
 $F \otimes_{\Q} \Q_{\ell}=\prod_{v \mid \ell}F_v$ and $B \otimes_{\Q}\Q_{\ell}=\prod_{v \mid \ell}B_v$, where  $v$ denotes a finite place of $F$.  
 For each $v \mid \ell$, 
we write $V_v \coloneqq V \otimes _{F}F_v$, and write $(V_v, \psi_v)$ for   the associated  $\Q_{\ell}$-valued skew-Hermitian $(B_v, *)$-module (Definition \ref{def:skew}).  
 
For any commutative $\Q_{\ell}$-algebra $R$, we have 
 \begin{align}\label{eq:bfG}
    \bfG(R) =  
\Bigg\{ 
(r, (g_v)_v) \in R^{\times} \times  \prod_{v \mid \ell} \GU_{\Q_{\ell}}(V_{v}, \psi_{v})(R)  
\mid r=c(g_v)
  \text{ for all }  v \mid \ell
\Bigg\}.
 \end{align}
  We put $m =\rank_B V$. 
 Let $\Delta$ denote the discriminant of $B$ over $F$. 
 Let $(V_{1,v}, \phi_{F_v})$ be the  symplectic $F_v$-space  of dimension  $2m$ when $v \nmid \Delta$, and $(V_v, \varphi_{B_v})$  the Hermitian $(B_v, \bar{\cdot})$-module of rank $m$  when $v \mid \Delta$, unique up to isomorphism in either case (Defninitions \ref{def:sym} and \ref{def:quat}). 
By equalities \eqref{eq:morita} and 
\eqref{eq:gu}, 
we have  isomorphisms of $\Q_{\ell}$-groups 
\begin{align}\label{eq:bfG_2}
\GU_{\Q_{\ell}}(V_{v}, \psi_{v})  \simeq 
   \begin{dcases*}
   \GSp_{\Q_{\ell}}(V_{1,v},  \phi_{F_v})  & if  $v \nmid \Delta$; 
   \\ 
   \GU_{\Q_{\ell}}(V_v, \varphi_{B_v}) & if   $v \mid \Delta$. 
   \end{dcases*}
\end{align}

\subsection{Irreducible  components of affine Deligne-Lusztig varieties}\label{ss:adlv}
We recall some general 
facts about affine Deligne-Lusztig varieties and their irreducible components. 
Let $k$ be an algebraically closed of characteristic $p$, and 
 $L$ be the field of fractions of the ring $W(k)$ of Witt vectors  over $k$.  
Let $G$ be a connected reductive group over $\Z_p$. 
In particular its generic
fiber $G_{\Q_p}$ is an unramified reductive group over $\Q_p$, i.e. quasi-split and splits
over an unramified extension of $\Q_p$.
 %completion $\widehat {\Q_{p}^{{\rm ur}}}$ of the maximal unramified extension $\Q_{p}^{\rm ur}$ of $\Q_p$.  
We fix  a maximal torus and a Borel subgroup $T \subset B \subset G$, and  
we may assume both are defined over $\Z_p$ as in \cite[A.4]{VW}. 
Let $(X^*(T), \Phi, X_*(T), \Phi^{\vee})$ be the corresponding  root datum. 
We write $X_*(T)^{+}$ for the set of dominant elements of $X_*(T)$. 
For $\mu \in X_*(T)^{+}$ and $b \in G(L)$,   the {\it affine Deligne-Lusztig variety} $X_{\mu}(b)$  associated to  ($G, \mu, b$) is a locally closed subscheme of the Witt vector partial affine flag variety ${\rm Gr}_{G}$ (\cite{BS, Zhu}) whose $k$-points are  
\[X_{\mu}(b)(k)=\{ g \in G(L)  \mid  g^{-1}b \sigma(g)\in G(W(k)) \mu(p)  G(W(k))\}/ G(W(k)). 
\] 
Further we  define a $\Q_p$-group $J_b$ by
\begin{align}\label{eq:J_b}
  J_b(R) = \{g\in G(L\otimes_{\Q_p}R ) \mid  g^{-1}b \sigma(g)=b\}
\end{align}
for any $\Q_p$-algebra $R$. 
Then $J_b(\Q_p)$ naturally acts on $X_{\mu}(b)(k)$ by left multiplication. 
 
% We first fix notation and recall some facts about the set of $J_b(\Q_p)$-orbits of irreducible components of  $X_{\mu}(b)$.  

Note that $T_L$ is a split maximal torus in $G_L$. 
Let $\sigma$ be the Frobenius of $L$ over $\Q_p$, acting on the group $X_*(T)$.  
Let $X_*(T)^{\sigma}$ and  $X_*(T)_{\sigma}$ denote the groups of $\sigma$-invariants and  $\sigma$-coinvariants of $X_*(T)$, respectively.  
For each $\lambda \in X_*(T)$, we write   $\underline{\lambda}$ for its image in $X_*(T)_{\sigma}$, and 
 write   $\lambda^{\diamond} \coloneqq f^{-1}\sum_{j=0}^{f-1} \sigma^j(\lambda) \in X_*(T)_{\Q}$ where $f \geq 1$ is an integer with $\sigma^{f}(\lambda)=\lambda$.   
Then  $X_*(T)_{\sigma,\Q} \xrightarrow{\sim} X_*(T)^{\sigma}_{\Q}$ where $ \underline{\lambda}\mapsto \lambda^{\diamond}$. 
Moreover, let  $\pi_1(G)=X_*(T)/\sum_{\alpha \in \Phi^{\vee}}\Z\alpha$ denote the Borovoi’s
fundamental group and $\lambda^{\natural}$  be the image of $\lambda$ in $\pi_1(G)_{\sigma}=\pi_1(G)/(1-\sigma)\pi_1(G)$.  
For $\lambda, \lambda' \in  X_*(T)_{\Q}=X_{*}(T) \otimes \Q$, we write $\lambda \leq \lambda'$ if  $\lambda'-\lambda$ is a non-negative rational linear  
combination of positive coroots. 

Let $B(G)$ be the set of $G(L)$-$\sigma$-conjugacy classes $[b] \coloneqq  \{g^{-1}b \sigma (g) \mid  g \in G(L)\}$ of elements $b \in G(L)$. 
Kottwitz showed that a class $[b] \in B(G)$ is uniquely determined by two invariants: the Kottwitz point 
$\kappa_G(b) \in \pi({G})_{\sigma}$ and  the Newton point 
$\nu_G(b) \in 
X_*(T)_{\Q}^{+}$ (\cite[4.13]{Kottwitz97}).  
%Here, $X_*(T)_{\Q}^{+}$ denotes the set of dominant elements of $X_*(T)_{\Q}$.
The set $B(G)$ naturally forms a poset with $[b]\leq[b']$ 
if 
$\kappa_G(b) = \kappa_G(b')$ and $\nu_G(b) 
\leq  \nu_G(b')$. 
We put  
\begin{equation}\label{eq:BGmu}
 B(G, \mu) \coloneqq  \{[b]\in 
B(G) \mid \nu_G(b) \leq \mu^{\diamond}, \kappa_G(b) = \mu^{\natural}\}.
\end{equation}
For $\mu \in X_*(T)^+$ and $b \in G(L)$, the variety  $X_{\mu}(b)$ is nonempty if and
only if $[b] \in B(G, \mu)$. 
   
By  \cite[Lemma/Definition 2.1]{HV}, 
there exists a unique element  $\underline{\lambda}_G(b) \in X_*(T)_{\sigma}$    such that 
\begin{itemize}
    \item[(i)]
$\underline{\lambda}_G(b)^{\natural}=\kappa_G(b)$ and
\item[(ii)] $\nu_G(b)-\underline{\lambda}_G(b)^{\diamond}$ is equal to a linear combination of simple coroots with coefficients in $[0, 1) \cap \Q$. 
\end{itemize}
This element $\underline{\lambda}_G(b)$ can be regarded as ``the best integral approximation" of the Newton point $\nu_G(b)$. 

Let $\widehat{G}$ be the Langlands dual of $G$ defined over $\overline{\Q}_{\ell}$ for a prime $\ell$  with $\ell \neq p$. Let $\wh{B}$ be a Borel subgroup of $\wh{G}$ 
with maximal torus $\wh{T}$, such that $X_*(T)^{+} =X^*(\wh{T})^+$.  
We write $V_{\mu}$ for the irreducible $\widehat{G}$-module of highest weight $\mu$. 
Let $V_{\mu}(\underline{\lambda}_G(b))$ be the sum of $\lambda$-weight spaces $V_{\mu}(\lambda)$ for $\lambda \in X_{*}(T)=X^*(\wh{T})$ satisfying  $\lambda \equiv \underline{\lambda}_G(b) \pmod {1-\sigma}$.  

Let $\Irr(X_{\mu}(b))$ (resp. ${\rm Irr}^{\rm top}(X_{\mu}(b))$) denote the set of irreducible components (resp. top-dimensional irreducible components) of $X_{\mu}(b)$. The following theorem was  conjectured by Chen and X. Zhu, and proved by Nie and Zhou-Y. Zhu.  
\begin{thm} [{\cite[Theorem 4.10]{Nie}, \cite[Theorem A]{ZZ}}]
There is an equality
\begin{equation}\label{eq:gsatake}
   \lvert  J_b(\Q_p)\backslash \Irr^{{\rm top}} (X_{\mu}(b)) \rvert  = \dim_{\overline{\Q}_{\ell}} V_{\mu}(\underline{\lambda}_G(b)). 
\end{equation}
\end{thm}
 
Now we apply the above results to the affine Deligne-Lusztig variety corresponding to the supersingular locus of the moduli space associated to the datum $\mathscr D$.  
In the rest of this subsection, we set   $G \coloneqq \bfG_{\Z_p}$. 
Recall we assume that  $F_v/\Q_p$ is unramified and $B_v \simeq \Mat_2(F_v)$ for each $v \mid p$. 
 We write  $f_v=[F_v :\Q_p]$. 
By \eqref{eq:bfG} and \eqref{eq:bfG_2}, we have an isomorphism 
 %Let $\tilde{G}=\Res_{F/\Q_{p}}\GSp_{2n, F}$, whose group of $R$-values for a $\Q_p$-algebra $R$ is given by 
%\begin{align}
%\tilde{G}(R)=
%\left\lbrace ((g_v)_{v}, (r_v)_v) \in \prod_{v\mid p}\GL_{2n}(R\otimes F) \times (R\otimes F)^{\times} \mid  
%g_v^t C g_v =r_v C \ {\rm for \ all} \ v \mid p \right\rbrace.
%\end{align}
\begin{align}\label{G_v}
G(R) \simeq 
\Bigg\{ (r, (g_v)_{v}) \in R^{\times} \times \prod_{v\mid p} \GL_{2m}(F_v \otimes_{\Q_p} R) \mid  rC= 
g_v^t C g_v \text{ for all } v \mid p \Bigg\} 
\end{align} 
for any commutative $\Q_p$-algebra $R$. 
Here, we write $C\coloneqq\begin{psmallmatrix}
 0 &I_m \\
 -I_m & 0
 \end{psmallmatrix}$. 

We fix an isomorphism     $F_v\otimes_{\Q_p} \ol{\Q}_p  = \prod_{j \in \Z/f_v\Z}\ol{\Q}_p$ for each $v \mid p$. 
We define a set $\Psi$ by 
\[\Psi \coloneqq \bigsqcup_{v \mid p} (\Z/f_v\Z),\]
 and we regard  
$G(\ol{\Q}_p)$ via \eqref{G_v} as a subgroup of the product $\ol{\Q}_p^{\times} \times \big( \prod_{ j \in \Psi}\GL_{2m} (\ol{\Q}_p) \big)$. 
Let  $T\subset {G}$ be the maximal torus consisting of all diagonal matrices in $G$, 
 parameterized in the following way:
\begin{align*}  
(\ol{\Q}_p^{\times})^{1+  m \cdot \lvert \Psi \rvert}  & \to T(\ol{\Q}_p), 
\\ 
\big( r,  (t^j_1, \ldots, t^j_m)_{j \in \Psi}  \big)
 & \mapsto \left(r,  \begin{pmatrix}
    r \diag(t^j_{1}, \ldots, t^j_{m}) & 
    0 
    \\ 
0  & 
\diag((t^j_{1})^{-1}, \ldots, (t^j_{m})^{-1})
\end{pmatrix}_{{j \in \Psi}}  \right).
\end{align*}   
 Let $\omega, (\epsilon_{i}^j)_{j \in \Psi, 1 \leq i \leq m} \in X^*(T)$ be the characters   defined by 
 \[ \omega \big( r,  (t^k_1, \ldots, t^k_m)_{k \in \Psi}  \big)=r,  \quad  \epsilon^j_{i} \big( r,  (t^k_1, \ldots, t^k_m)_{k \in \Psi}  \big)=t^j_{i}. \]      
 These characters  give a basis of   $X^*(T)$. 
Let  $\omega^*, (\epsilon^{j*}_{i})_{j \in \Psi, 1\leq i \leq m}$ denote the dual basis for $X_*(T)$:  For $t \in \ol{\Q}_p^{\times}$ we have 
\begin{align*}
    \omega^*(t) & = \left( 
    t, \begin{pmatrix}
        tI_m & 0 
        \\ 
        0 & 
        I_m
    \end{pmatrix}_{k\in \Psi }
    \right), 
    \\
\epsilon_{i}^{j*}(t)  & = (1, (g^{k})_{k \in \Psi}), 
 \quad  g^{k}= \begin{cases*}
    \diag(1, \ldots, 1, \overset{i}{t}, 1, \ldots, 1, \overset{m+i}{t^{-1}}, 1, \ldots, 1) & if $k=j$; 
    \\ 
    I_{2m} & 
    if $k \neq j$.
    \end{cases*}
\end{align*}
The Frobenius $\sigma$ acts on $X_{*}(T)$ by  
$\sigma(\omega^*)=\omega^*$ and $\sigma(\epsilon^{j*}_{i})=
\epsilon^{j+1*}_{i}$.

Let $B$ be the Borel subgroup  consisting of all upper triangular matrices in $G$. 
 The corresponding simple roots and  coroots are
\begin{align} \alpha^j_{1}& =\epsilon^j_{1}-\epsilon^j_{2},   &  & 
\ldots, &    
 \alpha^j_{m-1}&=\epsilon^j_{m-1}-\epsilon^j_{m}, &   \alpha^j_{m}&=\omega+
2\epsilon^j_{m} 
 & \in X^*(T), 
\\ \label{eq:coroots} \alpha_{1}^{j\vee}&=\epsilon^{j*}_{1}-\epsilon^{j*}_{2},  &  & 
\ldots, &  
\alpha_{m-1}^{j\vee}& =
 \epsilon^{j*}_{m-1}-\epsilon^{j*}_{m}, &   \alpha_{m}^{j\vee}&=\epsilon^{j*}_{m} & \in X_*(T),
 \end{align}
 varying $j \in \Psi$.

Let $[\mu]$ be the conjugacy class of the cocharacter $\mu_h$ attached to the datum $\mathscr D$ as in Section \ref{ss:Sh}. 
We fix an  embedding 
$\ol{\Q} \hookrightarrow \ol{\Q}_p$ and 
regard $[\mu]$  as a $W$-orbit in $X_*(T)$, where  $W \coloneqq N_G(T)/T$ denotes the  Weyl group. 
The dominant representative of $[\mu]$ in $X_*(T)$ is denoted by
$\mu$.  
 The description of $\mu_h$ in \eqref{eq:mu_h} implies  that 
 \begin{equation}\label{eq:mu}
  \mu=\omega^* \in X_*(T).    
 \end{equation}

Recall that a class $[b]\in B(G)$ is called basic if its  Newton point $\nu_G(b)$ lies in $X_*(Z_G)_{\Q}$, where $Z_G$ is  the center of $G$.  
Let 
 $[b]$ be the unique basic class in $B(G, \mu)$  \cite[6.4]{Kottwitz97}. 
Then 
\begin{equation}\label{eq:nu}
\nu_{G}(b)=
 \omega^* -\frac{1}{2} \sum_{j \in \Psi, 1\leq i\leq m} \epsilon^{j*}_{i} \in X_*(T)_{\Q}.
 \end{equation}
 In fact, $\nu_{G}(b)$  is characterized by the properties that $\nu_{G}(b) \leq \mu^{\diamond}(=\omega^*)$, and that $\nu_{G}(b) \in X_*(Z_G)_{\Q}$  since $[b] \in B(G, \mu)$ is basic. 
 The RHS of \eqref{eq:nu} satisfies these properties    as  
 \[\omega^*-({\rm RHS})=\sum 2^{-1} i  \alpha_{i}^{j\vee}, \quad (2 \cdot {\rm RHS})(t) =(t^2, (tI_{2m})_{j \in \Psi}) \in Z_G.\]   
Moreover, the map  $X_*(T) \to   \Z$ where   $c \omega^* + \sum a^j_{i} \epsilon^{j*}_{i}\mapsto c$ induces an identification  $\pi_{1}(G)_{\sigma}=\pi_1(G)  \xrightarrow{\sim} \Z$.  
By \eqref{eq:BGmu} and \eqref{eq:mu}, we have 
\begin{equation}\label{eq:kappa}
 \kappa_{G}(b)=1.
\end{equation}
\begin{lemma}\label{lgb}
Let $[b] \in B(G, \mu)$ be the basic class   and  
 $\underline{\lambda}_G(b) \in X_*(T)_{\sigma}$ be  the element  satisfying properties {\rm (i)} and  {\rm (ii)} above. 
 Then there is an equality  
\begin{align*}
\underline{\lambda}_G(b)  =\omega^* - \sum_{v \mid p} 
\left( \left\lceil \frac{f_v}{2} \right\rceil 
 \sum_{1 \leq i \leq m,   i : {\text{odd}}}  
   \epsilon^{0_v*}_{i}
    + \left\lfloor 
\frac{f_v}{2} 
\right\rfloor \sum_{1 \leq i \leq m,  i : {\text{even}}}
\epsilon^{0_v*}_{i} 
  \right) \pmod{1-\sigma},
\end{align*} 
where $0_v$ denotes the zero element of $\Z/f_v\Z$ for  $v \mid p$. 
\end{lemma}
 \begin{proof}
 We have   $({\rm RHS})^\natural=1$ and hence it satisfies 
 property (i). 
Moreover, 
 we have $(\epsilon^{0_v*}_{i})^{\diamond}=f_v^{-1}\sum_{j \in \Z/f_v\Z}\epsilon^{j*}_{i}$ for each $(v, i)$,   
 and hence 
 \begin{align*}
    &  \nu_G(b)-({\rm RHS})^{\diamond}
     \\ 
   = &    \sum_{\substack{v \mid p, \, j \in \Z/f_v\Z,\\[2pt]  1 \leq i \leq m, \, i: \text{ odd}}}
         \left( 
          \left\lceil \frac{f_v}{2} \right\rceil\frac{1}{f_v} 
         -\frac{1}{2}\right)
         \epsilon_{i}^{j*} 
         +
         \sum_{\substack{v \mid p, \, j \in \Z/f_v\Z, \\[2pt] 1 \leq i \leq m, \,    i : \text{ even}}}
         \left( 
         \left\lfloor \frac{f_v}{2} \right\rfloor \frac{1}{f_v} 
         -\frac{1}{2}\right)
         \epsilon_{i}^{j*} 
         \\
         = & 
         \sum_{\substack{v \mid p, \, f_v: \text{ odd}, \\[2pt]  j \in \Z/f_v\Z}}
         \frac{1}{2f_v}  
        \left(  \sum_{\substack{1 \leq  i \leq m, \\[2pt]  i :\text{ odd}}}  \epsilon^{j*}_{i}-
        \sum_{\substack{1 \leq  i \leq m, \\[2pt] i :\text{ even}}}  \epsilon^{j*}_{i}
        \right)
        \\
        = &  \sum_{\substack{v \mid  p, \, f_v: \text{ odd}, \\[2pt]  j \in \Z/f_v\Z}}
         \frac{1}{2f_v}  
          \sum_{\substack{1 \leq  i \leq m\\[2pt] i: \text{ odd}}}\alpha_{i}^{j\vee}.
 \end{align*}
 Thus the RHS satisfies property (ii). 
 \end{proof}  
\begin{prop}\label{prop:ADLV}
Let $[b]\in B(G, \mu)$ be the basic class. Then  
\[ \lvert J_b(\Q_p)\backslash \Irr(X_{\mu}(b)) \rvert = \prod_{v \mid p} 
\begin{pmatrix}
f_v
\\ 
\lfloor f_v/2 \rfloor
\end{pmatrix}^m. \]
\end{prop}
\begin{proof}
 Recall that we identify $X_{*}(T)$ with $X^*(\wh{T})$, and write $V_{\mu}(\underline{\lambda}_G(b))= \bigoplus  V_{\mu}(\lambda)$ where the sum is taken over all $\lambda \in X_{*}(T)$ with $\lambda \equiv \underline{\lambda}_G(b) \pmod {1-\sigma}$. 
 Since $\mu$ is minuscule,  for any $\lambda \in X_*(T)$ we have  $\dim_{\overline{\Q}_{\ell}} V_{\mu}(\lambda)=1$ or $0$ according as $\lambda\in W\cdot \mu$ or not. 
Hence  
\begin{equation}\label{eq:min}
\dim_{\overline{\Q}_{\ell}}V_{\mu}(\underline{\lambda}_G(b))= \# \{ \lambda \in W \cdot \mu \mid  \lambda \equiv \underline{\lambda}_G(b) \pmod{1-\sigma}\}.
\end{equation}

There is a decomposition of the Weyl group   $W=\prod_{j \in \Psi}W^{j}$ such that  $W^{j}$ is generated by elements  switching $\omega^*$ with $\omega^*-\epsilon^{j*}_{i}$  and those 
permuting the elements  $\epsilon^{j*}_{1}, \ldots, \epsilon^{j*}_{m}$ (so that    $W^j \simeq  (\Z/2\Z)^{m} \rtimes \mathfrak S_m$). 
Hence, the orbit $W \cdot \mu=W \cdot \omega^*$ consists of $2^{m \cdot \lvert \Psi \rvert}$-elements of the form    $\omega^*+\sum a^j_{i}\epsilon^{j*}_{i}$ where $a^j_{i}=-1$ or $0$. 

Further, a short computation shows that the submodule 
$(1-\sigma) X_{*}(T)$ 
consists of all elements 
$\lambda = c \omega^*+\sum_{j \in \Psi, 1\leq i \leq m}a^j_{i}\epsilon^{j*}_{i} \in X_{*} (T)$ satisfying $c=0$ and $\sum_{j \in \Z/f_v\Z} a^j_{i}=0$ for all $v \mid p$ and $1\leq i\leq m$. 
%In fact, for any $\lambda=\sum_{(v, j, i)}a_{v,j,i}\epsilon^{j*}_{i}+c\omega^*$ we have  that 
%$\lambda-\sigma({\lambda})=\sum_{(v, i, j)} (a_{v,j,i}- a_{v,  j-1,i})\epsilon^{j*}_{i}$. 
%Conversely, if    $\sum_{j \in \Z/f_v\Z} a_{v,j,i}=0$  then  the element  $\alpha \coloneqq \sum_{(v, j,i)}(\sum_{k \in \{0, \ldots, j\} \subset \Z/f_v \Z}  a_{v, k,i})\epsilon^{j*}_{i}$ satisfies  $\tau-\sigma(\alpha)=\alpha$. 
This and Lemma \ref{lgb} imply that an element  $\lambda \in X_*(T)$ satisfies $\lambda \equiv \underline{\lambda}_G(b) \pmod {1-\sigma}$ 
if and only if for all  $v$ and $i$ it satisfies  $\sum_{j \in \Z/f_v\Z}
a^j_{i}=-
\lceil f_v/2 \rceil$ or $-\lfloor f_v/2 \rfloor$ according as $i$ is odd or even.

Now let $S$ be the set of all $m$-tuples  $(J_1, \ldots, J_m)$ of subsets    $J_{i}\subset 
 \Psi$ such that for each $v \mid p$ it satisfies 
$\lvert J_i\cap (\Z/f_v\Z) \rvert = \lceil f_v/2 \rceil$ or $\lfloor f_v/2 \rfloor$  according as $i$  is odd  or even. 
Note that we have $\begin{pmatrix}
f_v
\\ 
\lceil f_v/2 \rceil
\end{pmatrix}=
\begin{pmatrix}
f_v
\\ 
\lfloor f_v/2 \rfloor
\end{pmatrix}$ and hence  $\lvert S \rvert = \prod_{v \mid p}\begin{pmatrix}
f_v
\\ 
\lfloor f_v/2 \rfloor
\end{pmatrix}^m$.  
The above argument shows that  
the assignment $(J_1, \ldots, J_m) \mapsto \omega^*-\sum_{1\leq i \leq m, \, j \in J_i}\epsilon^{j*}_{i}$ induces  a bijection from the set $S$ to the set on RHS of \eqref{eq:min}. 

Since $\mu$ is minuscule, $X_{\mu}(b)$ is equi-dimensional and in particular  
$\Irr^{\rm top}(X_{\mu}(b))=\Irr(X_{\mu}(b))$. 
These facts and 
equalities    \eqref{eq:gsatake} and     \eqref{eq:min}  imply the assertion.  
%of cocharacters  $\lambda \in X_{*}(T)$ such that $\lambda\in  W \cdot \mu$ and $\underline{\lambda}_G(b)= \lambda \pmod{1-\sigma}$.  
\end{proof}

 \subsection{The group of self-quasi-isogenies of a supersingular abelian variety}\label{ss:mass}
Recall from \S\ref{ss:Sh} that $\bfM_{\sf K}$ denotes the moduli scheme attached to the principal integral PEL-datum $\mathscr{D}$ and $N \geq 3$ with $p \nmid N$. 
We write 
 $\calM_{\sf K}  \coloneqq  \bfM_{\sf K}\otimes_{\Z_{(p)}} k$ for the special fiber. Our assumption that $p$ is unramified in $B$ implies  $\calM_{\sf K}$ is a smooth algebraic variety over $k$. 
Take a point  $x \in \mathcal M_{{\K}}(k)$, and let $(A, \lambda, \iota)$ denote  the principally  polarized $O_B$-abelian variety over $k$  corresponding to $x$.  
We  write $\End_{B}^0(A) \coloneqq \End_{B}(A) \otimes \Q$, and  define a $\Q$-group  $I$ by 
 \begin{align*}
 I(R)
  = \{g \in (\End_{B}^0(A) \otimes_{\Q} R)^{\times}  \mid  \exists c(g) \in  R^{\times} \text{ s.t. } g'\cdot g=  \id \otimes c(g) \}
 \end{align*} 
for any commutative $\Q$-algebra $R$. 
Here, $g \mapsto g'$ is the Rosati involution induced by  $\lambda$. 

By Theorem \ref{main.2}, the supersingular locus $\calM_{\sf K}^{\rm ss}$ is non-empty. 
Further, the description of the Newton point of the basic class $[b]\in B(\bfG_{\Z_p}, \mu)$ in \eqref{eq:nu} implies that the supersingular locus is precisely
the {\it basic locus} in the sense of \cite[Definition 8.2 and Example 8.3]{VW}. 
Moreover, the group $\bfG_{\Q}$ satisfies the Hasse principle. 
Hence the $p$-adic uniformization
theorem of Rapoport and Zink {\cite[Theorem 6.30]{RZ}} applies to the supersingular locus. 
 Note that in {\it loc. cit.} they described the completion of the integral model along
the basic locus 
as a quotient of what is now called a Rapoport–Zink formal scheme. 
A description using an affine Deligne-Lusztig variety was given  in \cite[Corollary 7.2.16]{XZ} and \cite[Proposition 5.2.2]{HZZ}. 
\begin{thm}[\cite{RZ, XZ, HZZ}]\label{inner}
 Assume that $x$ is lying on the supersingular locus $\M_{\K}^{\rm ss}(k)$. 

 (1) 
The $\Q$-group  $I$ is an inner form of $\bfG_{\Q}$, and such that $I(\R)$ is compact modulo center. 
Further,  there are natural identifications  
\begin{align*}
I_{\Q_{\ell}} 
=
\begin{cases} 
 \bfG_{\Q_{\ell}}  & \  {\rm if} \ \ell \neq p, 
\\ 
 J_b & \ {\rm if} \ \ell=p. 
 \end{cases}
 \end{align*}

(2) 
For any point $x' \in \calM^{\rm ss}_{\sf K}(k)$, the associated $\Q$-group $I'$ is isomorphic to $I$ 
as inner forms of $\bfG_{\Q}$. 

 (3) 
 There is an isomorphism of perfect  schemes 
\[ \Theta : 
I(\Q) \backslash 
X_{\mu}(b) \times \mathbf  G(\A_f^{p})/{\K}^p(N)
\xrightarrow{\sim}
 \mathcal M_{{\K}}^{{\rm ss}, \rm pfn},  
\]
where $\mathcal M_{{\K}}^{\rm{ss}, {\rm pfn}}$ denotes the perfection of $\mathcal M_{{\K}}^{{\rm ss}}$. 
\end{thm}
For the remainder of the paper, we fix a point $x\in \M^{\rm ss}_{\K}(k)$ and write $I$ for the associated $\Q$-group. 
 We define a $\Q$-group  $I^1$ by the exact sequence 
 \begin{equation}\label{eq:exact}
  1 \to I^1 \to I \xrightarrow{c}  \bbG_{{\rm m}, \Q} \to 1. 
 \end{equation}
% similarly to $\bfG^1$.  
Let $\Q_{>0}$ (resp. $\R_{>0}$) be the subgroup of $\Q^{\times}$  (resp. $\R^{\times}$) consisting of positive rational (resp. real) numbers. 
 \begin{lemma}\label{c(I)} 
 The  image of the homomorphism $c 
 : I(\Q) \to \Q^{\times}$ is equal to the subgroup  $\Q_{>0}$. 
\end{lemma}
\begin{proof}
   Kneser’s theorem and Hasse principle \cite[Theorems 6.4 and 6.6]{PR} show  that the natural map $H^1(\Q, I^1) \to H^1(\R, I^1)$ is injective. 
   This and the above exact sequence %\eqref{eq:exact}
   imply  that  $c(I(\Q))=c(I(\R))\cap \Q^{\times}$. 
  Moreover we have $c(I(\R))=\R_{>0}$  since the Rosati involution is a positive involution.  
\end{proof}
%\begin{align}\label{tau} 
%\tau(I^1)=1.\end{align}
%We fix a point $x \in \calM_{\K}^{\rm ss}(k)$. 
%Let $I$ be the algebraic $\Q$-group associated to $x$. 
Let $U$ (resp. $U^1$) be an open compact subgroup of $I(\A_f)$ (resp. $I^1(\A_f)$).  
Let $[g] \in I(\Q) \backslash I(\A_f)/U$ be a double coset, represented by an element $g \in I(\A_f)$. 
We write \[\Gamma_{g} \coloneqq I(\Q) \cap g^{-1}Ug.\] 
Then we have $c(\Gamma_g) \subset \Q_{>0} \cap \wh{\Z}^{\times}=\{1\}$, and hence $\Gamma_g \subset I^1(\Q)$. 
Since $I^1(\R)$ is compact, $\Gamma_g$ is finite. 
The \emph{mass of $I$ with respect to $U$} is defined by  
\[\Mass(I, U) \coloneqq \sum_{[g] \in I(\Q) \backslash I(\A_f)/U} \frac{1}{\lvert \Gamma_{g} \rvert}.  \] 
The \emph{mass of $I^1$ with respect to $U^1$} is defined similarly and denoted by $\Mass(I^1, U^1)$. 
    
\begin{lemma}\label{massI}
Let $U \subset I(\A_f)$ be an open compact subgroup, and let $U^1=U \cap I^1(\A_f)$.  
 Assume that the similitude Assume character $c : I(\A_f) \to \A_f^{\times}$ maps $U$ onto $\wh{\Z}^{\times}$. 
Then 
\begin{equation*}   \Mass \big(I,  U\big) =\Mass \big(I^1,  U^1\big).
\end{equation*}
\end{lemma}
\begin{proof}
We put $Z\coloneqq I(\Q) \backslash I(\A_f) / U$ and we  
 claim that   $Z$ can be naturally identified with the set  $I^1(\Q) \backslash I^1(\A_f) /U^1$.   
%equality \eqref{eq:sim}, and 
The assumption and Lemma \ref{c(I)}   imply  that 
 the similitude character induces the trivial map
 $Z \xrightarrow{c}   \A_f^{\times} / 
 \Q^{\times}_{>0} \cdot 
 \wh{\Z}^{\times}=1$.    
 Hence each coset of     $Z$ can be  represented by an element of $I^1(\A_f)$. 
Take $g_1, g_2 \in I^1(\A_f)$, $f \in I(\Q)$, $h \in  U$, and suppose that  $fg_1h=g_2$.  Then 
  we have $c(f) \in \Q_{>0}$, $c(h) \in \wh{\Z}^{\times}$, and  $c(f)c(h)=1$. 
  Hence  
 $c(f)=c(h)=1$, and this completes the proof of the claim. 
 The above argument for $\Gamma_g$   shows that  
 \[I(\Q) \cap g^{-1}  U g =I^1(\Q) \cap g^{-1}  U^1 g\] 
 for any  $g \in I^1(\A_f)$. 
 Thus we 
 see the assertion. 
 \end{proof}
 Now let $D_{p,\infty}$ be  the unique quaternion
$\Q$-algebra ramified precisely at $\{p,\infty\}$. 
Further, let 
$D$ be the unique quaternion
$F$-algebra such that $B\otimes_\Q D_{p,\infty}\simeq \Mat_2(D)$. 
%Equivalently, $\inv_v(D) = \inv_v((D_{p, \infty}\otimes_{\Q} F)\otimes_{F}B)$ for all $v$.  
We write $\Delta'$ for the discriminant of $D$ over $F$.  
Let $v$ denote a finite place of $F$. 
If $v \nmid p$, then we may identify  $B_v$ with $D_v \coloneqq D\otimes _{F}F_v$ and in particular we have that  $v \mid \Delta$ if and only if $v \mid \Delta'$. 
 
%Let $\bbH$ be the algebra of Hamilton's quaternions.  
% Then 
%\[I^1(\R) \simeq \{ g \in \Mat_m(\bbH)^d \mid \bar{g}^tg=1\},\]
% where  $\bar{g}$ is  the conjugate of $g \in \bbH$
 
 For each prime $\ell$,
the $\Q_{\ell}$-groups $\bfG_{\Q_{\ell}}^1$  and $I^1_{\Q_{\ell}}$ has a  decomposition   $\bfG^1_{\Q_{\ell}} =\prod_{v \mid \ell}\bfG_v^{1}$ and $I^1_{\Q_{\ell}}=\prod_{v\mid \ell}I_{v}^1$, respectively.
Similarly, we define a $\Q_p$-group $J_b^1$ as the subgroup of $J_b$ consisting of elements with trivial similitude factor:  
It has a decomposition 
$J^1_{b}=\prod_{v \mid p}J^1_{b, v}$.

Suppose $\ell \neq p$ and $v \mid \ell$. Then equalities  
  \eqref{eq:bfG},  
\eqref{eq:bfG_2}, and Proposition \ref{inner} (1) imply that    
\begin{align}\label{eq:G_v^1}
I^1_v  =\bfG^1_v ={\rm U}_{\Q_{\ell}}(V_v, \psi_v) \simeq  
   \begin{dcases*}
   \Res_{F_v/\Q_{\ell}}(\Sp_{2m, F_v})  & if $v \nmid \Delta'$; 
   \\ 
   {\rm U}_{\Q_{\ell}}(V_v, \varphi_{D_v}) & if $v \mid \Delta'$.  
   \end{dcases*}
\end{align}
Here, the pair $(V_v, \psi_v)$ is    the $\Q_{\ell}$-valued skew-Hermitian $(D_v, *)$-module of rank $m$, and  
 $(V_v, \varphi_{D_v})$ is the Hermitian $(D_v, \bar{\cdot})$-module of rank $m$ with respect to the  canonical involution  $b \mapsto \bar{b}$ on $D_v$ (Definitions \ref{def:skew} and \ref{def:quat}), 
 unique up to isomorphism in either case. 

Suppose $v \mid p$. Then Proposition \ref{inner} (1) and  
\cite[Section 4.1]{yu:IMRN-2008} imply that
%\cite[Section 2.5]{yu:mass_hb}, 
\begin{align}\label{eq:J_v^1}
I_v^1=
J_{b,v}^1 \simeq   \begin{dcases*}
\Res_{F_v/\Qp} \big( \Sp_{2m, F_v} \big)  & if $v \nmid \Delta'$; 
 \\ 
 {\rm U}_{\Q_p}(V_v, \varphi_{D_v}) & if $v \mid \Delta'$.  
\end{dcases*}
\end{align}
%where the pair  $(V_v, \varphi_{D_v})$ denotes  the unique Hermitian $(D_v, \bar{\cdot}\,)$-module of rank $m$ up to isomorphism.
%().  
%. 

%\begin{proof} 
%%If $v \mid \Delta'$, \cite{yu:mass_siegel}??.
%\end{proof}
 %We apply  formula  (\ref{geomMass})  to the group  scheme $G=I_1$. 

%Equalities \eqref{eq:G_v^1} and \eqref{eq:J_v^1} show that 
%the $\Q_{\ell}$-group $I_v^1$ is an inner form of $\Res_{F_v/\Q_{\ell}}(\Sp_{2m, F_v})$ for any  $\ell$ and $v \mid  \ell$. 

For any $v$, there is a canonical Haar measure  
on $I_v^1(\Q_{\ell})$. 
  Here we recall the construction given in \cite[Section 4]{Gross2}.
   The $\Q_{\ell}$-group $I_v^1$  is an inner form of the unramified group $\Res_{F_v/\Q_{\ell}}(\Sp_{2m, F_v})$ as  in \eqref{eq:G_v^1} and \eqref{eq:J_v^1}. 
 Let $\omega_v$ be an invariant differential of top degree on $\Res_{F_v/\Q_{\ell}}(\Sp_{2m, F_v})$ with nonzero
reduction on the special fiber of the  canonical integral model. 
We fix an inner twisting  
$f : I_v^1\to \Res_{F_v/\Q_{\ell}}(\Sp_{2m, F_v})$  over an extension of $\Q_{\ell}$. 
Then the pull-back $\omega_v^*\coloneqq f^*(\omega_v)$ is  an invariant differential form  on $I_v^1$. 
It is defined over $\Q_{\ell}$, and induces 
%(cf. \cite[Section 15, p.476]{JL}). 
a Haar measure $\lvert \omega^*_v \rvert$ on $I_v^1(\Q_{\ell})$. 

%Let $\mu$ be the Haar measure on ${\Sp}_{2m}(F_v)$ which gives a hyperspecial subgroup  $\Sp_{2m}(O_{F_v})$ volume one.  

Further, let $M_v^{\vee}(1)$ be the twisted dual of the motive of Artin-Tate type associated to $I_v^1$, and $L(M_v^{\vee}(1))$ be the local $L$-factor \cite[(1.6) and (5.1)]{Gross2}. 
For any parahoric subgroup  $U_v^1$ of $I_v^1(\Q_{\ell})$,   we put 
\begin{equation}\label{lpKp}
  \lambda_v(U_v^1) \coloneqq  \Bigg( L(M_v^{\vee}(1)) \cdot 
  \int_{U_v^1} 
 \lvert \omega_v^* \rvert \Bigg) ^{-1}.
\end{equation}
When $U_v^1$ is hyperspecial, we have that $\lambda_v(U_v^1)=1$ by \cite[Proposition 4.7]{Gross2}. 
In the next subsection we will give a description of $\lambda_v(U_v^1)$ for maximal parahoric subgroups. 
 \begin{prop}\label{mass}
Let $U^1=\prod_{v}U_{v}^1$ be an open compact subgroup of $I^1(\A_f)$ such that $U_v^1$ is a parahoric subgroup of $I_v^1(\Q_{\ell})$, and let $S$  be the finite set of finite places where $U_v^1$ is not hyperspecial. 
Then \begin{align*}
     \Mass \big(I^1,  U^1\big)
 = 
\frac{(-1)^{dm(m+1)/2}}{2^{md}}
\cdot \prod_{j=1}^{m}
\zeta_F(1-2j) \cdot \prod_{v\in S}\lambda_v(U_v^1)
\end{align*}
where $\zeta_F(s)$ is the Dedekind zeta function
 of $F$. 
 \end{prop}
 \begin{proof}
We apply the mass formula of  Gan, Hanke, and J.-K. Yu \cite[Proposition 2.13 and Section 9]{GHY} to the $\Q$-group  $I^1$ and we have that\footnote{In  {\it{loc.~cit.}}, the symbol of absolute value is missing.}
 \begin{equation*}
  \Mass \big(I^1,  U^1\big)
 = 
  \frac{1}{2^{md}} \cdot \left\lvert \prod_{j=1}^m\zeta_{F}(1-2j) \right\rvert 
  \cdot \prod_{v\in S}\lambda_v(U_v^1).
 \end{equation*}
 Moreover, the functional equation (cf. \cite[Section 2.2]{yu:mass_hb}) shows the product $\prod_{j=1}^m\zeta_{F}(1-2j)$ has sign 
 ${(-1)^{dm(m+1)/2}}$.
 \end{proof}
 \subsection{Stabilizers of lattices} 
 First we assume $v \mid \Delta'$, that is,  $D_v$  is the division algebra.  
 Let $(V_v, \varphi_{D_v})$ be  the unique Hermitian  $(D_v, \bar{\cdot})$-module of rank $m$ up to isomorphism (Definition \ref{def:quat}). 
 We recall some facts about the stabilizers of $O_{D_v}$-lattices in $(V_v, \varphi_{D_v})$.  
Let $c$ be an integer such that $0 \leq c \leq \lfloor m/2 \rfloor$.  
We define  an $O_{D_v}$-lattice $L_c$ in $(V_v, \varphi_{D_v})$ by 
\begin{align}\label{eq:L_c}
    L_c \coloneqq 
    \begin{cases*}
      H(-1)^{c} \oplus  H(0)^{m/2-c} & if $m$ is even; 
      \\ 
      H(-1)^{c} \oplus H(0)^{(m-1)/2-c} \oplus (1) & if $m$ is odd, 
    \end{cases*}
\end{align}
where $H(i)$ denotes a hyperbolic plane; see Section \ref{ss:nsp}. 
%We note that $L_0$ is self-dual, and any self-dual Hermitian $(O_{B_v}, \bar{\cdot})$-lattice of rank $m$ is isomorphic to $L_0$. 

We further define subgroups   
  \begin{align}\label{eq:P_c}
 P_c\subset \GU_{\Q_{\ell}}(V_v, \varphi_{D_v})(\Q_p), \quad  P^1_{c}\subset {\rm U}_{\Q_{\ell}}(V_v, \varphi_{D_v})(\Q_{\ell})
 \end{align}
 as the stabilizers of $L_c$.  
%\begin{equation}\label{eq:P^1_c}
%\calP_c\coloneqq \Stab \calL_c, \quad 
 %    \calP^1_{c} \coloneqq \Stab  \calL_c.
%\end{equation}
Then $P_c$ is a maximal parahoric subgroup of $\GU_{\Q_{\ell}}(V_v, \varphi_{D_v}) (\Q_{\ell})$, and  any maximal parahoric subgroup of $\GU_{\Q_{\ell}}(V_v, \varphi_{D_v}) (\Q_{\ell})$ is conjugate to $P_c$ for some $0 \leq c \leq \lfloor m/2 \rfloor$. 
 Similar statements hold true for the subgroups $P_c^1$ (cf. \cite[Theorem 3.13]{PR} or \cite{Tits}). 

 Let $\underline{P^1_c}$ be the smooth model of $P^1_c$ over $\Z_{\ell}$, and $\overline{P_c^1}$ the maximal reductive quotient  of the special fiber $\underline{P_c^1}\otimes_{\Z_{\ell}} \F_{{\ell}}$. 
 Then 
\begin{align}\label{eq:rdt}
\overline{P^1_{c}}= 
   % \Res_{\F_{q_v}/\F_p}\Sp_{2m, \F_{q_v}} & if $v \nmid \Delta'$; 
   % \\ 
    \Res_{\F_{q_v^2}/\F_{\ell}} \big( \Sp_{2c, \F_{q_v^2}} \big)  \times  \Res_{\F_{q_v}/\F_{\ell}}\big(  {\rm U}_{m-2c, \F_{q_v}} \big). 
\end{align}
Here,  $\F_{q_v}$ denotes   the residue field of $O_{F_v}$ and  $\F_{q_v^2}$ denotes its  quadratic extension. 
Further, ${\rm U}_{n, \F_{q_v}}$ denotes  the unitary group in $n$ variables over $\F_{q_v}$ (cf. \cite[Lemma 3.5.2]{Harashita}). 
From \cite[(2.6) and (2.12)]{GHY}, it follows that 
%using Iwahori subgroups (see also  \cite[(6.14)]{TY2}), 
    \begin{align*}
 % \label{vol}
 \lambda_v(P^1)
 %\Bigg( p^{\dim \bfG^1{v, \F_p}} \cdot  \lvert \bfG^1_v(\F_p) \rvert^{-1} \cdot  \int_{{P_{v,c}^1}} 
 %%\lvert
 % C^*_{I^1_{\Q_p}}
 %\rvert \Bigg)^{-1}
 = 
\frac{
{p}^{-N(\Res_{\F_{q_v}/\F_{\ell}} (\Sp_{2m, \F_{q_v}}))}\cdot \lvert  \Sp_{2m}(\F_{q_v})  \rvert }{ 
{p}^{-N(\overline{{P}^1})} \cdot \lvert  \overline{{P}^1}(\F_{{\ell}}) \rvert 
},  
\end{align*}
where $N(G)$  denotes the number of  positive roots of an $\F_{\ell}$-group $G$. 
Moreover, for any extension $\F_q/\F_{\ell}$ of degree $f$, we have (\cite[Chapter 1]{Carter})   
\begin{align*}
   N(\Res_{\F_{q}/\F_{\ell}}(\Sp_{2n, \F_{q}})) & = f  n^2, & 
   \lvert \Sp_{2n}(\F_{q}) \rvert &=q^{ n^2}\prod_{i=1}^n(q^{2i}-1), 
   \\ 
   N(\Res_{\F_q / \F_{\ell}}({\rm U}_{n, \F_q})) & = f  n(n-1)/2, & 
\lvert {\rm U}_{n}(\F_{q}) \rvert & =
q^{n(n-1)/2}\prod_{i=1}^n(q^i-(-1)^i). 
\end{align*}
  These formulas and \eqref{eq:rdt} show that for $0 \leq c \leq \lfloor m/2 \rfloor$ 
\begin{align}
  \label{eq:lambda}
\lambda_v(P^1_{ c}) = \kappa_v(m,c)^{-1} \cdot 
\prod_{i=1}^m (q_v^{2i}-1), \quad \kappa_v(m, c) \coloneqq 
\prod_{i=1}^c(q_v^{4i}-1)\cdot \prod_{i=1}^{m-2c}(q_v^i-(-1)^i).
%\\
%& =
%\prod_{i=1}^{m-2c} (q_v^i+(-1)^i)
%\cdot 
%\prod_{i=1}^c (q_v^{4i-2}-1)
%\cdot 
%\frac{\prod_{i=1}^m (q_v^{2i}-1)}{
%\prod_{i=1}^{2c}(q_v^{2i}-1) \cdot 
%\prod_{i=1}^{m-2c}
%(q_v^{2i}-1)
%}.
\end{align}
We remark that this rational function of $q_v$ is in fact  a polynomial with integer coefficients \cite[Lemma 3.2]{IKY}.
 % 
%Moreover, the minimal degree occurs precisely when $c=0$ if $m$ is odd and when $c=m/2$ if $m$ is even.  This implies  the following:
\begin{lemma}\label{mphr} 
A parahoric subgroup $P_c^1$ of ${\rm U}_{\Q_{\ell}}(V_v, \varphi_{B_v})(\Q_{\ell})$ has the maximum volume precisely when $c=0$ if $m$ is odd,  and when $c=m/2$ if $m$ is even. 
\end{lemma}
\begin{proof}
It suffices to show that $\lambda_v(P_c^1)$ is the minimum, that is, $\kappa_v(m,c)$ is the maximum at $c=0$ or $m/2$ according as $m$ is odd or even. 
We give a proof by induction on $m$. 
A short computation shows  the statement holds  for $m=1, 2$. 
We show that if  the statement is true for $m$  then it is also true for $m+2$. 
By \eqref{eq:lambda} we have that
\begin{align*}
%\begin{split}\label{eq:kmc}
   \kappa_v(m+2, c)/\kappa_v(m, c)
& =(q_v^{m+2-2c}-(-1)^{m+2-2c})(q_v^{m+1-2c}-(-1)^{m+1-2c}), 
    \\
   \kappa_v(m+2, c+1)/\kappa_v(m, c)
    & = q_v^{4(c+1)}-1.      
%\end{split}
    \end{align*}
These  functions of $c$ are the maximum at $c=0$ and $c=\lfloor m/2 \rfloor$, respectively. 
Furthermore, the maximum values are 
\begin{alignat*}{2}
  A& \coloneqq {} & \kappa_v(m+2, 0)/\kappa_v(m, 0) &  =\begin{cases*}
        (q_v^{m+2}+1)(q_v^{m+1}-1) & if $m$ is odd;
        \\
        (q_v^{m+2}-1)(q_v^{m+1}+1) & if $m$ is even.  
    \end{cases*}
    \\ 
B& \coloneqq {} &  \kappa_v(m+2, \lfloor m/2 \rfloor+1)/\kappa_v(m, \lfloor m/2 \rfloor) & =
       \begin{cases*}
        (q_v^{m+1}+1)(q_v^{m+1}-1) & 
        if $m$ is odd; 
        \\
        (q_v^{m+2}+1)(q_v^{m+2}-1) & if $m$ is even.  
    \end{cases*}\label{eq:c+1}
    \end{alignat*}
We see that   $A>B$ (resp. $B>A$) if $m$ is odd (resp. even). 
These facts and the induction hypothesis imply the assertion. 
\end{proof}
\begin{proof}[\bf Proof of Theorem~\ref{main.3}]
We write    
   $J_b(Y)$ for the stabilizer in $J_b(\Q_p)$ of an irreducible component $Y$ of  the affine Deligne-Lusztig variety $X_{\mu}(b)$. 
We fix identifications  ${\bf G}(\A_f^p)=I(\A_f^p)$ and ${J_b} (\Q_p)={I}(\Q_p)$ as in Theorem \ref{inner} (1), and we regard $J_{b}(Y) \bfG(\wh{\Z}^p)$ as a subgroup of $I(\A_f)$. 

The action of $J_b(\Q_p)$ on the set $\Irr(X_{\mu}(b))$ induces  a bijection 
\[  
\coprod_{ [Y] \in  J_b(\Q_p)\backslash {\Irr} (X_{\mu}(b))}  J_b(\Q_p)/J_{b}(Y)\xrightarrow{\sim} {\Irr}(X_{\mu}(b)).\]  
  Moreover,  the isomorphism  $\Theta$ in Theorem \ref{inner} (3)  induces  a bijection 
    %(see \cite[Theorem A]{HZZ}) 
 \begin{align}\label{eq:irr}
 \coprod_{[{Y}] \in J_b(\Q_p)\backslash {\rm Irr}(X_{\mu}(b))}
 I(\Q) \backslash {I}(\A_f)/  J_{b}(Y){\K}^p(N) \xrightarrow{\sim} \Irr(\mathcal M_{{\K}}^{\rm ss}).
 \end{align}
The assumption $N \geq 3$ implies  $I(\Q) \cap (g^{-1}  J_{b}(Y) \K^p(N)   g)=1$ for any $g \in I(\A_f)$ (cf. \cite[Lemma, p. 207]{mumford:av}). 
%Moreover, the assumption $p \nmid N$ shows $[\bfG(\wh{\Z}^p) : \K^p(N)]=\lvert \bfG(\Z/N\Z) \rvert$. 
Hence  we have that 
 \begin{align}\label{eq:level}
 \begin{split}
   \lvert  I(\Q) \backslash {I}(\A_f)/   J_{b}(Y){\K}^p(N) \rvert & = \Mass(I, J_b(Y){\K}^p(N)) 
   \\  
  &=
    \Mass(I,  J_{b}(Y) \bfG(\wh{\Z}^p))
    \cdot \lvert \bfG(\Z/N\Z) \rvert.
  \end{split}
   \end{align}
   
    The open compact subgroup $J_{b}(Y) \bfG(\wh{\Z}^p) \subset I(\A_f)$ satisfies the assumption in Lemma \ref{massI}.  
In fact, descriptions in \eqref{eq:G(Z)}, \eqref{eq:J(Y)}, and Lemmas \ref{c_sym}, \ref{c} show that  for any   $r \in \Z_{\ell}^{\times}$ with $\ell \neq p$ (resp. $\ell=p$)   there exists  an element $g \in \bfG(\Z_{\ell})$ (resp. $g \in J_b(Y)$) such that $c(g)=r$. 
Therefore, if  we write  $J_b^1(Y) \coloneqq J_b(Y) \cap J_b^1(\Q_p)$, then 
\begin{equation}\label{eq:I^1}
 \Mass(I,  J_{b}(Y) \bfG(\wh{\Z}^p))=\Mass(I^1,  J^1_{b}(Y) \bfG^1(\wh{\Z}^p)).  
\end{equation}

We apply Proposition \ref{mass} to the subgroup $J^1_{b}(Y) \bfG^1(\wh{\Z}^p)$. 
Suppose first that $\ell \neq p$  and $v \mid \ell$. 
We identify $B_v$ with $D_v$. 
The $\Q_{\ell}$-valued skew-Hermitian $O_{D_v}$-lattice  $\Lambda_v \coloneqq \Lambda \otimes_{O_{F}}O_{F_v}$ is self-dual by the assumption.  
Let  $\bfG_v^1(\Z_{\ell})$ be the stabilizer of  $\Lambda_v$ in $\bfG_v^1(\Q_{\ell})={\rm U}_{\Q_{\ell}}(V_v, \psi_v)(\Q_{\ell})$, so that  
 $\bfG^1(\wh{\Z}^p)=\prod_{\ell \neq p, v \mid \ell}\bfG_v^1(\Z_{\ell})$. 
 When $v \nmid \Delta'$, 
  one can attach to $(\Lambda_v, \psi_{v})$ a self-dual 
$O_
{F_v}$-lattice in the symplectic $F$-space of dimension $2m$ as in Section \ref{sec:LC.2}, and  
this induces an equality   $\bfG_v^1(\Z_{\ell}) = \Sp_{2m}(O_{F_v})$ under the identification in \eqref{eq:G_v^1}. 
Suppose  $v \mid \Delta'$, and let  $(\Lambda_v, \varphi_{D_v})$ be  the associated  Hermitian $(O_{D_v}, \bar{\cdot}\,)$-lattice  of rank $m$. 
By \eqref{eq:LC.15},  we have that  
\[\Lambda_v=\Lambda_v^{\vee, \psi_v}=\Pi^{i} \cdot \Lambda_v^{\vee, \varphi_{D_v}}, \quad  i \coloneqq \ord_{\Pi_{v}}(\gamma)-1.\] 
where $\gamma$ is the element in $B^{\times}$ defined as in \eqref{eq:gamma} and $\ord_{\Pi_v}(\, \cdot \,)$ is the $\Pi_v$-adic valuation for a uniformizer $\Pi_v$ of $D_v$.   
As in Lemma \ref{quatclass}, such a lattice $\Lambda_v$ exists if and only if either $m$ or $i$ is even. 
Further, one has that $\Lambda_v \simeq H(i)^{m/2}$ or $H(i)^{(m-1)/2} \oplus (\Pi^i_v)$, according as $m$ is even or odd. 
Note that 
the stabilizer of a lattice  remains unchanged when the lattice  is multiplied by a power of  $\Pi_v$. 
These facts, \eqref{eq:L_c}, and \eqref{eq:P_c} imply that, when $v \mid \Delta'$,  we have 
 $\bfG^1_v(\Z_{\ell}) =  P^1_0$ or $P^1_{m/2}$ according as $\ord_{\Pi_v}(\gamma)$ is odd or even.  
In summary, for each $\ell \neq p$ and $v \mid  \ell$,  we have 
\begin{equation}
\label{eq:G(Z)}\bfG_v^1(\Z_{\ell}) = 
\begin{cases*}
\Sp_{2m}(O_{F_v}) 
& if $v \nmid \Delta'$; 
\\ 
P_0^1 & if  $v \mid \Delta'$ and $\ord_{\Pi_v}(\gamma)$ is odd; 
\\
P^1_{m/2} & if  $v \mid \Delta'$ and $\ord_{\Pi_v}(\gamma)$ is even. 
\end{cases*}
\end{equation}
In particular,  if $v \nmid \Delta'$ then $\bfG_v^1(\Z_{\ell})$ is a hyperspecial parahoric subgroup of  $\bfG_v^1(\Q_{\ell})$. 
%We note that the rank $m$ is always even in the last case; see Proposition \ref{prop:herm-self-dual}. 

Next we consider the case that $\ell=p$. 
%we  take an irreducible component $Y$ of  the affine Deligne-Lusztig variety $X_{\mu}(b)$, and write   
  % $J_b(Y)$ for its stabilizer in $J_b(\Q_p)$. 
   By the results of 
   He-Zhou-Zhu  \cite[Theorem 4.1.2 and Proposition 2.2.5]{HZZ} and Nie \cite{Nie}, 
   the stabilizer   $J_b(Y)$ of a $Y \in {\rm Irr}(X_{\mu}(b))$  in $J_b(\Q_p)$ is a parahoric subgroup, and has  the maximum volume among all the parahoric subgroups of $J_b(\Q_p)$. 
  Note that a  hyperspecial subgroup, if it exists, has the maximum volume.
%We write  $J_b^1(Y) \coloneqq J_b(Y) \cap J_b^1(\Q_p)$. 
The subgroup $J_b^1(Y)$ consisting elements with trivial similitude factor has a decomposition  $J_b^1(Y) = \prod_{v\mid p} 
J^{1}_{b, v}(Y)$ where $J^{1}_{b, v}(Y)\subset J_{b,v}^1(\Q_p)$ is a parahoric subgroup with the maximum volume. 
As in \eqref{eq:J_v^1}, we identify $J_{b,v}^1(\Q_p)$ with $\Sp_{2m}(F_v)$ or ${\rm U}_{\Q_p}(V_v, \varphi_{D_v})(\Q_p)$ according as  $v \nmid \Delta'$ or $v \mid \Delta'$. 
 From Lemma \ref{mphr} it follows that for $v \mid p$
\begin{align}\label{eq:J(Y)}
J^{1}_{b, v}(Y) \sim_{\text{conj}} \begin{dcases*}
\Sp_{2m}(O_{F_v}) & if $v \nmid \Delta'$; 
\\ 
P^1_{0} 
& 
if $v \mid \Delta'$ and $m$ is odd; 
\\ 
P^1_{m/2} & if $v \mid \Delta'$ and  $m$ is even.
\end{dcases*}
\end{align}
%Here, we use the  notation  $P_c$ to denote the stabilizer of a Hermitian lattice $L_c$ as in \eqref{eq:P_c}, replacing $(B_v, V_v)$ by $(D_v, V_v')$. 
 
Finally, suppose  $v \mid \Delta'$ and 
let $\lambda_v(\, \cdot \,)$ denote  the reciprocal of a volume  as in \eqref{lpKp}. 
From   \eqref{eq:lambda}--\eqref{eq:J(Y)} it follows  that 
\begin{align} \label{eq:lambda_ell} 
\begin{split}
\lambda_v(\bfG^1_v(\Z_{\ell})) & = 
\begin{dcases*}
\prod_{i=1}^m (q_v^i+(-1)^i)  & if $v \mid \ell$, $\ell \neq p$, and $\ord_{\Pi_v}(\gamma)$ is odd;
\\ 
\prod_{i=1}^{m/2} (q_v^{4i-2} -1) & if $v \mid \ell$, $\ell \neq  p$, and $\ord_{\Pi_v}(\gamma)$ is even, 
\end{dcases*}
\\  \lambda_v(J^{1}_{b, v}(Y)) & =      \begin{dcases*}
\prod_{i=1}^m (q_v^i+(-1)^i) & 
        if $v \mid p$ and $m$ is odd; 
        \\ 
\prod_{i=1}^{m/2} (q_v^{4i-2} -1) & 
if $v \mid p$ and $m$ is even.     \end{dcases*}
     \end{split}
     \end{align}
     
 Propositions \ref{prop:ADLV}, \ref{mass}, and 
equalities  \eqref{eq:irr}--\eqref{eq:lambda_ell}   
%, \eqref{eq:irr}, \eqref{eq:level}, \eqref{eq:I^1},
 imply Theorem \ref{main.3}.
\end{proof}
 
%Let $U_{1, p} \subset  J_1(\Q_p)=I_1(\Q_p)$ be a parahoric subgroup of  $I_1(\Q_p)$. 
% Let  $\underline{U}_{1, p}$ be the smooth  model of $U_{1, p}$ over $\Z_p$,  $\overline{U}_{1, p}$ be the maximal reductive quotient of  $\underline{U}_{1, p} \otimes \F_p$, and $N(\overline{U}_{1, p})$ be  the number of positive roots of $\overline{U}_{1, p}$. 
  
%The mass of $I_1$ relative to the open compact subgroup $U_1=U_{1,p} \cdot I_1(\widehat{\Z}^p) \subset I_1(\A_f)$ is 
%\[ \Mass(I_1, U_1)=  \Mass(\Lambda)  \cdot \frac{\kappa_p}{p^{-N(\overline{U}_{1, p})}
%\cdot \# \overline{U}_{1, p}(\F_p)}\]
%  where $\Mass(\Lambda)$ is the mass of the quaternionic hermitian lattice $(\Lambda, \varphi)$, and $\kappa_p$ is given by 

%\begin{thm}
%We write $\zeta_F(s)$ for  the Dedekind zeta function. 
%The number of irreducible components of the supersingular locus  $\mathcal M^{\rm ss}$ is 
%\begin{align*}
%\lvert G(\mathbb Z/N\mathbb Z)\rvert  \cdot \prod_{v \mid p}
%\begin{pmatrix}
%f_{v}
%\\ 
%\lfloor f_{v}/2 \rfloor 
%\end{pmatrix}^n 
%\cdot 
%\frac{(-1)^{dm(m+1)/2}}{2^{md}}
%\cdot \prod_{j=1}^{m}
%\left[ \zeta_F(1-2j)\prod_{v \mid \Delta'}(q_v^j+(-1)^j). \right]
%\end{align*}
%\end{thm}
%\subsection{Proof of Theorem \ref{main.3}}\label{ss:mass}

  \begin{remark}[Correction to ``An exact geometric mass formula'']\label{rem:correction}
    In 
       \cite[(4.1)]{yu:IMRN-2008}, 
    a description of the stabilizer $\bfG^1_v(\Z_{\ell})\subset \bfG_v^1(\Q_{\ell})$ of a self-dual local lattice $\Lambda_v$    contains an error: If $\ell \neq p$, $v\mid \ell$ and  $v \mid \Delta'$,  then it is claimed in {\it loc. cit.} that  $\bfG_v^1(\Z_{\ell})$ is always identified with $P_0^1$. The correct description is in \eqref{eq:G(Z)}.    
Hence we also correct the formula in \cite[Theorem 1.3]{yu:IMRN-2008} for local terms at places $v \nmid p$ with $v \mid \Delta'$:    
       The cardinality of the superspecial locus 
  $\calM^{\rm sp}_{\sf K}$  
  is equal to 
\begin{align}\label{eq:mass_sp}
\lvert \bfG(\mathbb Z/N\mathbb Z)\rvert  
\cdot 
\frac{(-1)^{dm(m+1)/2}}{2^{md}}
\cdot \prod_{j=1}^{m}
\zeta_F(1-2j) \cdot \prod_{v \mid p \text{ or } v \mid \Delta'}\lambda'_v,
\end{align}
where $\lambda'_v$ for a place $v \mid p$ or $v \mid \Delta'$ is given by 
 \begin{align}
     \lambda'_v= 
     \begin{dcases*}
        \prod_{i=1}^m (q_v^i+1) & if $v \mid p$ and $v \nmid \Delta'$; 
        \\ 
        \prod_{i=1}^{m/2} (q_v^{4i-2} -1) & 
if  $v \nmid p$ and {$\ord_{\Pi_v}(\gamma)$ is even (and $m$ is even)}; \\ 
        \prod_{i=1}^m (q_v^i+(-1)^i) & otherwise. 
     \end{dcases*}
 \end{align}
  \end{remark}

\appendix

\section{Bad reduction of Shimura Curves: counting irreducible components of special fibers}\label{appx}
We consider the moduli schemes of type C for the case $m=1$ and $d=1$ (that is, $F=\Q$). These are called (variants of) Shimura curves. 
As before 
 $B$ denotes an indefinite quaternion $\Q$-algebra with a positive involution $*$,  $O_B$ a maximal order in $B$ stable under $*$, and 
 $\Delta$ the discriminant of $B/\Q$. 
 For each prime $\ell \mid \Delta$, let 
  $\Pi_{\ell}$ be a uniformizer of the division algebra $B \otimes _{\Q}\Q_{\ell}$. 
For a fixed prime $p$, let $\bfM$ (resp.~$\bfM^{\rm unp}$) be the coarse moduli scheme over 
$\Z_{(p)}$ of principally polarized  (resp.~unpolarized) $O_B$-abelian surfaces  that satisfies the determinant condition.\footnote{When $p \mid \Delta$, the determinant condition on an $O_B$-abelian surface is the same as the {\it special  condition} 
%imposed on $O_B$-abelian scheme over a $\Z_p$-scheme $S$ of relative dimension $2$ 
in the sense of \cite[III. 3.1]{BC} and  \cite{Drinfeld}.} 
 The natural forgetful  map  $(A,\lambda,\iota)\mapsto (A,\iota)$ gives rise to a morphism 
\[ f: \bfM\to \bfM^{\rm unp}. \] 
%\begin{equation}\label{eq:gamma1}
%   \gamma+\bar\gamma=0\ \text{and}\ \gamma^2<0  
%\end{equation}
%such that 
%\begin{equation}
%    \label{eq:*}
%    b^*=\gamma \bar b \gamma^{-1}, \ \forall\, b\in B
%\end{equation}
%Conversely, if $\gamma$ satisfies \eqref{eq:gamma1}, then the involution defined by \eqref{eq:*} is positive. 

We recall that the involution $*$ on $B$ can be written as 
\begin{equation}
    \label{eq:A.2}
    b  \mapsto b^*= \gamma \bar{b} \gamma^{-1}
\end{equation}
for an element $\gamma \in B$ with $\gamma^2<0$. 
In \cite[Section 4]{Drinfeld} and \cite[III. 0.5]{BC},  Drinfeld and    Boutot-Carayol chose a  $\gamma$ such that 
\begin{equation}\label{special}
 \gamma^2=-\Delta. 
\end{equation}
In this case, there is a maximal order $O_B$ such that $\gamma\in O_B$ since all maximal orders are conjugate under $B^\times$. This order $O_B$ is  also stable under the  involution $*$. 
Further, every object $(A,\iota)_S$ in $\bfM^{\rm unp}(S)$ admits a unique $O_B$-linear principal polarization (\cite[Proposition 4.3]{Drinfeld} and \cite[Proposition 3.3]{BC}) so that the morphism $f$ is an isomorphism. 
The geometry of $\bfM\simeq \bfM^{\rm unp}$ in this case has been studied and is well-understood; see~\cite{Ogg, Carayol}. 
Conversely, we have the following characterization of such a positive involution. 

\begin{lemma}
    If there exists a principally polarized complex $O_B$-abelian surface, then the  involution $*$ is given as in \eqref{eq:A.2} and \eqref{special}.
\end{lemma}
Therefore, condition \eqref{special} is equivalent to the conditions in Theorem~\ref{main.1} under $m=1$ and $d=1$. 
\begin{proof}
One can write $b^*=\gamma \bar b \gamma^{-1}$  and $\gamma^2=-r$ for some  $\gamma\in B^\times$ and $r\in \Q_{>0}$. 
For a prime $\ell \mid \Delta$, one has that  $\ord_{\Pi_{\ell}}(\gamma)$ is odd by Theorem~\ref{main.1},  and hence $\ord_\ell(r)$ is odd. 
For $\ell\nmid \Delta$, we may identify $O_B\otimes \Z_{\ell}$ with $\Mat_2(\Z_{\ell})$ so we have $\gamma\in \Q_{\ell}^\times \cdot \GL_2(\Z_\ell)$, which is the normalizer of $\Mat_2(\Z_\ell)$, and hence $\ord_\ell(r)$ is even. Therefore, $r=\Delta s^2$ for some $s\in \Q$. 
Replacing $\gamma$ by $\gamma/s$, we get the desired result. 
\end{proof}
The Cherenik-Drinfeld theorem  \cite[Section 4]{Drinfeld} states that, under condition \eqref{special} and that $p \mid \Delta$,  
the formal completion of $\bfM^{\rm unp}\otimes W(\Fpbar)$ along the special fiber admits a $p$-adic unformization by 
one-dimensional Deligne's formal scheme $\wh \Omega^{\rm nr}=\wh \Omega \hat \otimes W(\Fpbar)$.
Note that in this case we have  $\ord_{\Pi_p}(\gamma)=1$. 

Consider now the general case: Let $\gamma \in B$ such that $b^*=\gamma \bar b \gamma^{-1}$, % as in Section \ref{ss:PEL} 
but no longer suppose condition \eqref{special}. 
Let 
\[ S \coloneqq \{ \text {primes}\ \ell : \ell \mid \Delta, \,  \ord_{\Pi_\ell}(\gamma) \text{ is even} \}. \]
By Theorem~\ref{main.1}, the moduli scheme $\bfM$ is non-empty if and only if $S=\emptyset$. 
We relax the condition on the moduli scheme  $\bfM$ by removing the determinant condition, and write  $\wt \bfM$ for the coarse moduli scheme over $\Z_{(p)}$ of principally polarized $O_B$-abelian surfaces for a fixed prime $p$.  
Note that we have  $\bfM_{\Q} = \wt \bfM_{\Q}$.

Suppose that $\wt \bfM$ is non-empty and let $(A,\lambda, \iota)$ be a principally polarized $O_B$-abelian surface over an algebraically closed field of \ch zero or $p$. 
Then  the  $\ell$-adic Tate module $T_\ell(A)$, for any prime $\ell\neq p$, is a self-dual skew-Hermitian $O_B\otimes \Z_\ell$-lattice of rank one. 
This and Proposition  \ref{prop:herm-self-dual} imply  that $S\subset \{p\}$. 
It is natural to know whether the case $S=\{p\}$ can occur. 
 In this case, $\wt \bfM$ is  necessarily supported in its special fiber $\wt \bfM\otimes \Fp$ by Theorem \ref{main.1}. 
In the following we will show that when $S=\{p\}$, 
the special fiber is still non-empty, so that the case $S=\{p\}$ does occur.

Assume that $p \mid \Delta$. 
We begin by reviewing some basic properties of  \dieu $O_B\otimes \Zp$-modules of rank four. 
We simply write $\Pi=\Pi_{p}$. 
Then one has $O_B\otimes \Zp=\Z_{p^2}[\Pi]$ subject to relations \eqref{eq:Pi}.  
Let $k$ be an algebraically closed field $k$ of \ch $p$ and $W(k)$ be the ring of Witt vectors over $k$. 
Write $\Hom(\Z_{p^2},W(k))=\{\tau_0,\tau_1\}\simeq \Z/2\Z$ with Frobenius action $\sigma$ by $j\mapsto j+1$. 
Let $(M, \iota_p)$ be a  \dieu $O_B\otimes \Zp$-module of $W(k)$-rank four. 
We have a decomposition $M=M^0 \oplus M^1$ as in \eqref{eq:dec_Mj} on which three operators act as follows 
\[ \sfF,\sfV, \Pi: M^{0}\to M^{1} \ \text{and } \ M^{1} \to M^{0}. \]
We also have the respective decompositions 
\[ M/\sfV M =( M/\sfV M)^0 \oplus ( M/\sfV M)^1, \quad M/(\sfF,\sfV)M =( M/(\sfF,\sfV) M)^0 \oplus ( M/(\sfF,\sfV) M)^1.\]

\begin{defn}
For each $j\in \Z/2\Z$, put 
\[ c_j \coloneqq \dim_k M^{j}/\sfV M^{j+1}, \quad p_j\coloneqq \dim_k M^{j}/\Pi M^{j+1}, \quad a_j\coloneqq \dim_k ( M/(\sfF,\sfV) M)^j. \]
We call respectively $(c_0,c_1)$ 
the {\it Lie type}, $(a_0,a_1)$ the {\it $a$-type}, and $(p_0,p_1)$ the {\it $\Pi$-type} of $(A,\iota)$ or of the \dieu $O_B\otimes\Zp$-module $(M, \iota_p)$.    
\end{defn}
If $(M, \iota_p)$ does not satisfy the determinant condition, then  either $(c_0,c_1)=(2,0)$ or $(c_0,c_1)=(0,2)$. 
 Suppose that  $(c_0,c_1)=(2,0)$ for simplicity. 
We have 
\begin{equation}
    \label{eq:VMj}
    \sfV M^1=pM^0, \quad \sfV M^0=M^1.
\end{equation}
This implies that $\sfF M^0=M^1$ and $\sfF M^1=pM^0$. So we have 

\[ (\sfF,\sfV)M^1=pM^0,\ (\sfF,\sfV)M^0=M^1,\  \text{and}\  (a_0,a_1)=(2,0). \]
Therefore, $M$ is superspecial. 
From the commutative diagram 
\[ \begin{CD}
    M^0 @>\sfF>> M^1 \\
    @VV{\Pi}V  @VV{\Pi}V \\
    M^1 @>\sfF>> M^0
\end{CD}\]
and that $\sfF M^0=M^1$, we have
\[ pM^0 \subseteq \Pi M^1=\Pi\cdot \sfF(M^0)=\sfF\cdot \Pi(M^0)\subseteq \sfF(M^1)=pM^0.\]
So we have  
\begin{equation}\label{eq:PiMj}
   \Pi M^1= pM^0, \quad \Pi M^0=M^1, \quad \text{and $(p_0,p_1)=(2,0)$.}
\end{equation}

Conversely, 
let $M$ be a  $W(k)$-module of rank four which is  equipped with the following three structures: 
\begin{enumerate}
    \item [(i)] $\sfF$ and $\sfV:M^{j} \to M^{j+1}$ are $\sigma$-linear and $\sigma^{-1}$-linear maps, respectively such that $\sfF \sfV=\sfV \sfF=p$. 
    \item [(ii)] $M=M^0\oplus M^1$, where $b\in \Z_{p^2}$ acts on $M^j$ by the multiplication by $\tau_j(b)$ and the map $\Pi:M^j\to M^{j+1}$ for $j\in \Z/2\Z$ satisfies 
    \begin{equation}\label{eq:FPi}
        \Pi^2=-p,\quad \text{and}\quad  \sfF \cdot \Pi=\Pi \cdot \sfF.
    \end{equation}
    \item [(iii)] $(\, ,):M\times M\to W(k)$ is a $W(k)$-bilinear, symmetric and perfect pairing, and it satisfies
    \begin{equation}\label{eq:anti-pol}
        (\sfF x,y)=(x,\sfV y)^\sigma, \quad (M^0,M^0)=(M^1,M^1)=0, \quad (\Pi x, \Pi y)=p(x,y), \quad \forall\, x, y\in M.
    \end{equation}
\end{enumerate} 
Then, by putting $\<x,y\> \coloneqq 
(x,\gamma^{-1} y)$, we obtain a  principally polarized \dieu $O_B\otimes \Zp$-module $(M, \<\, , \>,\iota_p)$ with Lie type $(c_0,c_1)=(2,0)$. 
\begin{lemma}\label{lm:DM(2,0)}
Assume that $p \mid \Delta$. 
Then 
   there exists a  unique principally polarized \dieu $O_B\otimes \Zp$-module of $W(k)$-rank four  with Lie type $(c_0,c_1)=(2,0)$ up to isomorphism.   % {\rm (}or $(0,2)${\rm )}.  
\end{lemma}
Such a \dieu module is superspecial as seen above. 
The case $(c_0,c_1)=(0,2)$ can be obtained from this result with the index shifted by one.
\begin{proof} 
% Put
%\[ (x,y)\coloneqq \<x,\gamma y\>, \quad \forall\, x, y \in M.\]
We first prove existence. 
Let $M=M^0\oplus M^1$, where $M^0$ and $M^1$ are free $W(k)$-modules of rank two, with bases $\{e_1,e_2\}$ and $\{e_3, e_4\}$, respectively. Define an action of  $\Z_{p^2}$ on $M$ by (ii) and  $\Pi$ on $M$ by the representative matrix $[\Pi]$ with respect to $\{e_i\}$:
\begin{equation}
   \label{eq:[Pi]}
   [\Pi]=\begin{pmatrix}
      0 & p \bbJ_1 \\
      \bbJ_1 & 0 
   \end{pmatrix}, \ \text{where} \ \bbJ_1= \begin{pmatrix}
      0 & -1 \\
      1 & 0 
   \end{pmatrix}.
\end{equation}
Then $\Pi^2=-p$ and $M$ is an $O_B\otimes W(k)$-module of rank one. 
Define a $W(k)$-bilinear pairing $(\, ,):M\times M\to W(k)$ using the matrix: 
\begin{equation}\label{eq:pairing}
   ( (e_i,e_j) ) =\begin{pmatrix}
    0 & \bbI_2 \\
    \bbI_2 & 0 
\end{pmatrix}, \ \text{where} \ \bbI_2= \begin{pmatrix}
      1 & 0 \\
      0 & 1 
   \end{pmatrix}.  
\end{equation}
Then $(\, ,)$ is perfect and symmetric and each $M^j$ is an isotropic submodule. 
Further we define a $\sigma$-linear map $\sfF:M\to M$ whose representative matrix with respect to $\{e_i\}$  is given by 
\begin{equation}\label{eq:[F]}
 [\sfF]=\begin{pmatrix}
    0 & p \bbI_2 \\
    \bbI_2 & 0 
\end{pmatrix}.    
\end{equation} 
Then we have $\sfF \cdot \Pi=\Pi \cdot \sfF$ and hence condition \eqref{eq:FPi} is satisfied.
Moreover  we set   $\sfV\coloneqq p \sfF^{-1}$ and then  condition \eqref{eq:anti-pol} is satisfied. 
Hence $M$ is equipped with structures (i), (ii), (iii) as desired.

  Next we prove uniqueness. 
  Let $(M, \<\, , \>,\iota_p)$ be a principally polarized \dieu $O_B\otimes \Zp$-module of $W(k)$-rank four with Lie type $(c_0,c_1)=(2,0)$. 
  We put $(x,y)\coloneqq \<x,\gamma y\>$ for $x, y\in M$. 
  We show that  there exist $W(k)$-bases $\{e_1,e_2\}$ and $\{e_3,e_4\}$ for $M^0$ and $M^1$, respectively,  such that  conditions \eqref{eq:[Pi]}, \eqref{eq:pairing} and  \eqref{eq:[F]} are satisfied. 
Put $M^\diamond 
\coloneqq \{m\in M \mid  \sfF^2 m=p m\}$. 
Then 
$M^\diamond$ is a principally polarized \dieu $O_B\otimes \Zp$-module over $\F_{p^2}$ 
of $W(\F_{p^2})$-rank four such that $M^\diamond\otimes_{W(\F_{p^2})} W(k)=M$. 
As before, we have the decomposition $M^\diamond=M^{\diamond,0}\oplus M^{\diamond,1}$. 
On $M^\diamond$, we have $\sfF^2=\sfF \cdot \sfV=p$ and $\sfF=\sfV$. 
Let $\varphi :M^{\diamond,0}\times M^{\diamond,0}\to W(\F_{p^2})$ be a pairing given by $\varphi(x, y) \coloneqq (x,\sfF y)$. 
One easily checks that $\varphi$ is a unimodular Hermitian form over $W(\F_{p^2})$ of rank two. 
Since $W(\F_{p^2})/W(\Fp)$ is unramified, there exists an orthonormal basis 
$\{e_1,e_2\}$ of $M^{\diamond, 0}$ for $\varphi$. 
Set $e_3\coloneqq \sfF e_1$ and $e_4\coloneqq \sfF e_2$. Then $\{e_i\}$ satisfy  conditions \eqref{eq:pairing} and \eqref{eq:[F]}. 
From \eqref{eq:PiMj}, we can  write 
\[ [\Pi]=\begin{pmatrix}
    0 & p B \\
    C & 0
\end{pmatrix} \quad \text{for some $B,C\in \Mat_2(W(\F_{p^2}))$}.\]
Putting $\scrB \coloneqq [e_1,e_2,e_3,e_4]$, we compute that
\[ \sfF\cdot \Pi (\scrB)=\scrB \cdot [\sfF]\cdot \ol {[\Pi]}=\scrB\cdot\begin{pmatrix}
   p \ol C & 0\\
   0  & p \ol B     
\end{pmatrix}, \quad \Pi \cdot \sfF (\scrB)=\scrB \cdot [\Pi]\cdot [\sfF]=\scrB\cdot \begin{pmatrix}
   p B & 0\\
   0  & C
\end{pmatrix}. \]
So we have $B=\ol C$. 
From $(\Pi x, \Pi y)=p(x,y)$ and $\Pi^2=-p$, 
we obtain $\ol C^t C=\bbI_2$ and $\ol C \cdot C=-\bbI_2$ and  
these imply  $C^t=-C$. 
Hence we can write $C=\begin{psmallmatrix}
    0 & -c \\
    c &  0
\end{psmallmatrix}$ and then we have $c \bar c=1$ from $\ol C C=-\bbI_2$. 
Hence \eqref{eq:[Pi]} is satisfied. 

If $\scrB'=[e_1', e_2',e_3',e_4']$ is another $W(k)$-basis for $M^{\diamond,0}\oplus M^{\diamond,1}$ satisfying conditions \eqref{eq:pairing} and \eqref{eq:[F]}, then $\scrB'=\scrB \cdot P$ and the transformation matrix $P$ satisfies
\[ P=\begin{pmatrix}
    A & 0 \\
    0 & \ol A
\end{pmatrix}, \quad \ol A^t \cdot A=\bbI_2.\]
With respect to the new basis $\scrB'$, the representative matrix $[\Pi]'$ of $\Pi$ satisfies 
\[ [\Pi]'=P^{-1} [\Pi] P=\begin{pmatrix}
    0 & p \ol{C'} \\
    C' & 0 
\end{pmatrix}, \quad {\text{where}} \quad C'=\ol A^{-1} C A=A^ t C A. \]
We choose $A=\begin{psmallmatrix}
    c^{-1} & 0 \\
    0 &  1
\end{psmallmatrix}$ 
and compute that $C'=\bbJ_1$. This proves the lemma. 
\end{proof}

%converse holds true, namely, the necessary condition $S\subset \{p\}$ also implies that $\wt \bfM$ is non-empty. 
Now we consider the case  where $S=\{p\}$. 
Note that we can also assume that $\ord_{\Pi_p}(\gamma)=0$. 
\begin{lemma}
 There exists a triple $(B,*,O_B)$ such that $S=\{p\}$. 
 \end{lemma}
 \begin{proof}Choose a prime $p$ and an odd number of distinct primes $p_1,\dots, p_t$ such that $p$ is inert or ramified in $K\coloneqq \Q(\sqrt{-p_1\dots p_t})$. 
Let $B$ be the quaternion $\Q$-algebra ramified exactly at $\{p, p_1,\dots, p_t\}$. 
Then there is an embedding $K \embed B$ of $\Q$-algebras, since $B$ is indefinite  and any prime  ramified in $B$  is either inert or ramified in $K$. 
Therefore, there exists an element $\gamma\in B^\times$ such that $\gamma^2=-p_1 \cdots p_t$. Define a positive involution $*$ on $B$ by $b\mapsto b^*=\gamma \bar b \gamma^{-1}$.
Choose a maximal order $O_B$ of $B$ containing $\gamma$. 
Then $O_B$ is stable under the involution $*$. 
Clearly, $S=\{p\}$. 
\end{proof}
When $S=\{p\}$, Theorem \ref{main.1} implies that 
an $O_B$-abelian surface $(A,\iota)$ does not satisfy the determinant condition  and   
hence we have either $(c_0,c_1)=(2,0)$ or $(c_0,c_1)=(0,2)$. 
\begin{prop}
    \label{prop:exist_av}
    Assume that 
    $S=\{p\}$. 
    For $(c_0,c_1)=(2,0)$ or $(c_0,c_1)=(0,2)$, there exists a principally polarized $O_B$-abelian surface  over $k$ of Lie type $(c_0,c_1)$. 
    Furthermore, such an abelian surface is  superspecial. 
\end{prop}
\begin{proof}
    The assertion follows from Lemma \ref{lm:DM(2,0)}, using the same argument as in Theorem~\ref{exist}.  
\end{proof}
\begin{cor}
    \label{cor:wtM}
  % Let the notation and assumption be as in Proposition~\ref{prop:exist_av}. 
  Assume that 
    $S=\{p\}$. 
    Then $\wt \bfM$ is a non-empty and zero-dimensional scheme whose points are contained in the special fiber $\wt \bfM \otimes \Fp$.   
\end{cor}
\begin{proof}
Non-emptiness of $\wt \bfM$ follows from Proposition~\ref{prop:exist_av}. 
As all $k$-points of $\wt \bfM$ are superspecial, $\bfM$ has dimension zero. 
\end{proof}

%\red{Do we need to mention that $\SU_2(\Zp)$ is isomorphic to $\SL_2(\Zp)$, any good reference?} 

\begin{prop}\label{prop:mass}
Assume that $S=\{p\}$. Then we have
\begin{equation}
    \Mass(\wt \bfM(k))\coloneqq \sum_{[(A,\lambda,\iota)]\in \wt \bfM(k)} \frac{1}
    {|\Aut(A,\lambda,\iota)|}=\frac{1}{12}\prod_{\ell|(\Delta/p)} (\ell-1). 
\end{equation}
\end{prop}

\begin{proof}
One can easily obtain the result by modifying the mass formula in \eqref{eq:mass_sp}. 
The local factor at $p$ satisfies that $\lambda_p'=1$ because the local compact subgroup $U_p$ at $p$ is hyperspecial by the lemma below. 
%Lemma~\ref{lm:AutM}. 
Also we need to multiple the mass by two  since there are two different Lie types associated to objects. 
\end{proof}
\begin{lemma}\label{lm:AutM}
    Let $(M,\<\, ,\>,\iota_p)$ be a principally polarized \dieu $O_B\otimes \Zp$-module of $W(k)$-rank four  with $(c_0,c_1)=(2,0)$ or $(0,2)$. 
    Then 
    \[ \Aut_{\rm DM}(M,\<\, ,\>,\iota_p)\simeq {\rm SU}_2(\Zp)\coloneqq \{A\in \GL_2(\Z_{p^2})\mid  \ol A^t A=\bbI_2, \det(A)=1\}. \]  
\end{lemma}

\begin{proof}
    We may show the case where $(c_0,c_1)=(2,0)$ and the proof for $(c_0,c_1)=(0,2)$ is the same. Choose a $W(k)$-basis $\scrB=\{e_1,\dots, e_4\}$ as in Lemma~\ref{lm:DM(2,0)}. With respect to $\scrB$, an element $Q\in \Aut_{\rm DM}(M,\<\, ,\>,\iota_p)$ is represented by a matrix
    \[ \begin{pmatrix}
        A & 0 \\
        0 & \ol A
    \end{pmatrix}, \quad A\in \GL_2(\Z_{p^2}).\]
We check that $Q\circ \sfF=\sfF \circ Q$. The conditions $(Qx,Qy)=(x,y)$ and $Q\circ \Pi=\Pi \circ Q$ give rise to the conditions $\ol A^t A=\bbI_2$ and $\det A=1$, respectively. 
This proves the lemma. 
\end{proof}
We return to the standard setting for Shimura curves and  assume condition~\eqref{special}. 
Let $\mathscr D=(B,*,O_B,V,\psi,\Lambda,h_0)$ be a principal integral PEL datum of rank one. 
We further 
assume $p\mid \Delta$. 
Let $\bfG$, $N\ge 3$, and $\bfM_{\sf K}$ be defined as in Sections~\ref{ss:PEL} and~\ref{ss:Sh}.
%(\red{refer to a more precise place}). 
Then the geometric special fiber $\calM_{\K}=\bfM_{\K} \otimes k$  is equal to its supersingular locus $\calM_{\sf K}^{\rm ss}$, and  the singular locus of $\calM_{\sf K}$ is exactly the superspecial locus. Moreover, at each superpecial point there are exactly two components passing through and intersecting transversally. Namely, $\calM_{\sf K}$ has ordinary singularities exactly at superspecial points.

Let $\Sigma_{{\sf K}^p}$ be the set of isomorphism classes of polarized superspecial $O_B$-abelian surfaces $(A_0,\lambda_0,\iota_0,\eta_0)$ over $\Fpbar$ with level $N$-structure and with Lie type $(2,0)$ or $(0,2)$ such that $\ker \lambda_0\simeq \alpha_p\times \alpha_p$. 

\begin{prop}\label{prop:component}
    There is a natural bijection between $\Sigma_{{\sf K}^p}$ and the set $\Irr(\calM_{\sf K})$ of irreducible components of $\calM_{\sf K}$.
\end{prop}
\begin{proof}
The statement without $O_B$-action has been established by Katsura and Oort using the Moret-Bailly family~\cite[Section 2]{katsura-oort:surface}. We sketch the proof using the geometry of $\calM_{\sf K}$ aforementioned. For each member $\ul A_0=(A_0,\lambda_0,\iota_0,\eta_0)$ in $\Sigma_{{\sf K}^p}$, we construct a family $\calX_{\ul A_0}$ of $O_B$-linear isogenies $\rho: \ul A_0 \to \ul A=(A,\lambda, \iota,\eta)$ of degree $p$. 
The map $\rho\mapsto \ul A$ induces an isomorphism from $\calX_{\ul A_0}$ onto an irreducible component $X$ of $\calM_{\sf K}$. Conversely, for each $X\in \Irr(\calM_{\sf K})$, choose a non-singular point $\ul A$ of $X$, so $a(A)=1$.
Let $\rho:A_0\to A$ be the minimal isogeny of $A$~\cite[1.8]{LO}. One takes the pull-back polarization  $\lambda_0=\rho^* \lambda$ and level $N$-structure $\eta_0=\rho^* \eta$. 
Then $\ker \lambda_0\simeq \alpha_p\times \alpha_p$. 
The $O_B$-action $\iota$ on $A$ can be lifted uniquely to an  $O_B$-action $\iota_0$ on $A_0$~\cite[Proposition 4.8]{yu:endo}. 
This gives rise to an object $\ul A_0=(A_0,\lambda_0,\iota_0,\eta_0)$, and the point $\ul A_0$ depends only on $X$ as is done in \cite[Theorem 2.1]{katsura-oort:surface}. 
Now we show that $\ul A_0$ has Lie type $(2,0)$ or $(0,2)$. 
Let $M_0$ and $M$ be the \dieu modules of $\ul A_0$ and $\ul A$, respectively. 
We have 
\begin{equation}\label{eq:A.11}
\text{$M_0=(\sfF,\sfV)M$, \quad $M_0\subsetneq M \subsetneq \sfV^{-1}M_0$ \ and \  $M_0^j\subseteq M^j\subseteq (\sfV^{-1}M_0)^j$}     
\end{equation}
for all $j\in \Z/2\Z$. Since $\ker \lambda_0\simeq \alpha_p\times \alpha_p$, we also have $\sfV^{-1}M_0=M_0^{\vee,\<\,,\,\>_0}$ where $M_0^{\vee,\<\,,\,\>_0}$ denotes the dual $W(k)$-lattice of $M_0$ with respect to the polarization $\<\,,\>_0$. 
Since $a(M)=1$, we have $M_0^{j'} \subsetneq M^{j'}$ and $M_0^{{j'}+1} =M^{{j'}+1}$ for some ${j'}\in \Z/2\Z$.
Since $M$ has Lie type $(1,1)$, we have $(VM)^{{j'}+1}\subsetneq M^{{j'}+1}=M_0^{{j'}+1}$. This and Equation ~\eqref{eq:A.11} give
\[ M_0^{j'} \subsetneq M^{j'}\subsetneq (\sfV^{-1} M_0)^{j'}, \quad  \dim_k (M_0/\sfV M_0)^{{j'}+1}=2, \quad \text{and} \quad (c_{{j'}},c_{{j'}+1})=(0,2). \]
Therefore, $\ul A_0$ has Lie type $(2,0)$ or $(0,2)$ and it is a member of $\Sigma_{{\sf K}^p}$.
This gives the desired correspondence. 
\end{proof}

\begin{remark}
Irreducible components of $\calM_\K$  can be classified  into two types via the Lie types of superspecial abelian surfaces, using  Proposition~\ref{prop:component}.
This classification corresponds to the one given by two types of vertices in the Bruhat-Tits tree of $\SL_2(\Qp)$ \cite{Ogg, Carayol}.
    Alternatively, we may associate to each irreducible component the $a$-type of any of its non-singular points, which is either $(1,0)$ or $(0,1)$, as shown in the proof of Lemma~\ref{prop:component}. From this, one sees that the intersection of two irreducible components of different types, if non-empty, has points of $a$-type $(1,1)$, which are superspecial. Conversely, every superpsecial point, which has $a$-type $(1,1)$, lies in one irreducible component with $a$-type $(1,0)$ and the other component with $a$-type $(0,1)$. 
\end{remark}
\begin{lemma}
    \label{lm:M0}
    Assume $p \mid \Delta$ and that $*$ satisfies condition~\eqref{special}. 
    For $(c_0,c_1)=(2,0)$ or $(c_0,c_1)=(0,2)$, there is one isomorphism class of polarized \dieu $O_B\otimes \Zp$-modules $(M_0,\<\,,\>_0,\iota_p)$ of $W(k)$-rank four  with Lie type $(c_0,c_1)$ such that $\sfV M_0^{\vee,\<\,,\,\>_0}=M_0$. Moreover, we have
    \begin{equation}
       \Aut_{\rm DM} (M_0, \<\,,\>_0,\iota_p)\simeq {\rm SU}_2(\Zp). 
    \end{equation}
    % \[ \Aut_{\rm DM} (M_0, \<\,,\>_0,\iota_p)\simeq {\rm SU}_2(\Zp). \]
\end{lemma}
\begin{proof}
    Suppose that $(c_0,c_1)=(2,0).$ In this case, equalities  \eqref{eq:VMj} and \eqref{eq:PiMj} imply $\Pi M_0=\sfV M_0$. 
    Let $(\, , \,)_0 :M_0\times M_0 \to W(k)[1/p]$ be a pairing given by  $(x,y)_0\coloneqq \<x, p^{-1} \gamma y\>$.
   Then 
    \[ \begin{split}
        M_0^{\vee,(\,,\,)_0}&\coloneqq\{x\in M_0[1/p] \mid (x,M)_0\subset W(k)\} \\ 
          &=\{x\in M_0[1/p] \mid \<x,p^{-1} \gamma M\>_0\subset W(k)\}\\
          &=p \cdot \gamma^{-1} M_0^{\vee,\<\,,\,\>_0}=p\cdot \Pi^{-1} \sfV^{-1} M_0=M_0.  
    \end{split}
     \]
 Hence  $(M_0=M_0^0\oplus M_0^1,(\,,)_0,\iota_p)$ satisfies the properties (i), (ii) and (iii) above. 
    By Lemmas~\ref{lm:DM(2,0)} and~\ref{lm:AutM},  there is one isomorphism class of such modules and hence one isomorphism class of the polarized \dieu $O_B\otimes \Zp$-modules  $(M_0,\<\,,\>_0,\iota_p)$. We also obtain $\Aut_{\rm DM} (M_0, \<\,,\>_0,\iota_p)\simeq {\rm SU}_2(\Zp)$ from Lemma~\ref{lm:AutM}.
The proof of the case $(c_0, c_1)=(0,2)$ is similar. 
\end{proof}

Using Propositions~\ref{prop:mass} and \ref{prop:component}, and Lemma ~\ref{lm:M0}, we get the following result.

\begin{prop}\label{prop:IrrMK}
   The moduli space $\calM_{\K}$ has 
   \[ \lvert\bfG(\Z/N\Z)\lvert \cdot \frac{1}{12} \cdot \prod_{\ell|(\Delta/p)} (\ell-1)\]
   irreducible components. 
\end{prop}

\end{document}